\documentclass[preprint,3p,12pt]{elsarticle}

\usepackage{amssymb, bm}
\usepackage{amsthm, amsmath}
\newtheorem{theorem}{Theorem}[section]        % resets each section
 
\usepackage{graphicx,color}
\usepackage{caption}
\usepackage{subcaption}
\usepackage{multirow}
\usepackage{mathtools}
\usepackage{xparse}
  
\usepackage{listings}%% for code
  
\usepackage{tikz}
\usetikzlibrary{spy,decorations.markings}
\usetikzlibrary{arrows.meta,positioning,shapes.geometric}
\usetikzlibrary{calc}
\definecolor{OIblue}{HTML}{0072B2}
\definecolor{OIverm}{HTML}{D55E00}
\definecolor{OIgreen}{HTML}{009E73}
\definecolor{OIorange}{HTML}{E69F00}
\definecolor{OIsky}{HTML}{56B4E9}
\definecolor{OIyellow}{HTML}{F0E442}
\definecolor{OIpurple}{HTML}{CC79A7}
\definecolor{OIblack}{HTML}{000000}

\usepackage{pgfplots}
\pgfplotsset{compat=1.18} 

\usepackage{float}
 
 \usepackage{booktabs}
 \usepackage{siunitx}
\usepackage{autobreak} % try to keep its use minimal (see notes below)
\usepackage{eqparbox}                  % equal-width boxes for LHS alignment

\usepackage[hidelinks]{hyperref}  % or: \usepackage[colorlinks=true,linkcolor=blue]{hyperref}

\newtheorem{assumption}{Assumption}

\journal{Arxiv}

\begin{document} 
%{\color{red}
%\begin{itemize}
%%\item Fig 6b, 6e need to have DG overshoot/undershoot in the colorbar.
%%\item Fig. 3a need to change title from CCFV to DG
%%\item Need to fix the DG vs FV percentanges for badia example
%\item iteration counts?
%\item comparisons to other methods FCT?
%\item 1D or 3D examples?
%\item better doodle for voronoi mesh showing $K$-orthogonality
%\item describe the limiter used
%\item title of paper
%\item maybe present the paper as a way to find FV and DG regions, while being BP.
%\end{itemize}
%}
\begin{frontmatter}



\title{ 
A bound-preserving multinumerics scheme for steady-state convection–diffusion equations
}

 
\author[inst1]{Maurice S. Fabien}

\affiliation[inst1]{organization={Center for Computational Science and Engineering, Schwarzman College of Computing, 
\\
Massachusetts Institute of Technology},%Department and Organization
            addressline={77 Massachusetts Avenue}, 
            city={Cambridge},
            state={MA},
            postcode={02139}, 
            country={USA}} 

\begin{abstract}
	We solve the convection–diffusion equation using a coupling of cell-centered finite volume (FV) and discontinuous Galerkin (DG) methods. The domain is divided into disjoint regions assigned to FV or DG, and the two methods are coupled through an interface term.  DG is stable and resolves sharp layers in convection-dominated regimes, but it can produce sizable spurious oscillations and is computationally expensive; FV (two-point flux) is low-order and monotone, but inexpensive. We propose a novel adaptive partitioning strategy that automatically selects FV and DG subdomains: whenever the solution's cell average violates the bounds, we switch to FV on a small neighborhood of that element. Viewed as a natural analog of $p$-adaptivity, this process is repeated until all cell averages are bound-preserving (up to some specified tolerance). Thereafter, standard conservative limiters may be applied to ensure the full solution is bound-preserving.  Standard benchmarks confirm the effectiveness of the adaptive technique.
\end{abstract}


% \begin{highlights}
% \item 
% \end{highlights}

\begin{keyword}
finite elements 
\sep
 discontinuous Galerkin methods 
 \sep 
 bound-preserving
\sep
slope limiters
\sep
convection-dominated flow 
\end{keyword}

\end{frontmatter}


\section{Introduction} \label{sec:intro} 
	As a core transport model, the convection-diffusion equation captures the interplay of flow and diffusion. Transport of scalar fields (e.g., temperature in energy models, concentration in mass transport) is commonly represented by partial differential equations that combine convection by a velocity field with diffusion. When convection outweighs diffusion, the problem enters an convection-dominated regime and solutions develop narrow boundary or internal layers with very steep spatial variation. Such layers occur in both steady and time-dependent cases; however, the latter typically do not introduce similarly small time scales \cite{stynes2005steady}. Our focus concentrates on the steady-state problem.
 
	Mesh-based discretizations struggle when convection overwhelms diffusion because the characteristic layers are much narrower than a feasible mesh spacing. Stabilization is essential for many classical methods; otherwise oscillations contaminate the entire solution \cite{stynes2018convection}.  A thorough review of stabilization techniques in the finite element setting is given in \cite{roos2014numerical}.
			%When dealing with numerical methods for convection-dominated problems, it is well known that the solution order of accuracy drops in the vicinity of steep gradients or discontinuities.  
	This work studies the coupling of classical cell-centered finite volume (CCFV or FV)~\cite{EYMARD2000713} and discontinuous Galerkin (DG) methods \cite{dolejvsi2015discontinuous,di2011mathematical}. Among methods for convection-diffusion, CCFV and DG are two of the most commonly employed discretizations \cite{augustin2011assessment}. Cell-centered finite volume schemes are locally conservative, inexpensive, and under standard two-point flux approximations are monotone, which makes them robust for convection-dominated problems. Their drawbacks are low formal accuracy (typically first order), layer smearing, possible grid-orientation effects, and loss of consistency on non-orthogonal meshes or with strong anisotropy unless more elaborate fluxes or reconstructions are used~\cite{aavatsmark2002introduction}.
	
	 Discontinuous Galerkin methods, by contrast, offer high-order accuracy, natural $hp$-adaptivity, compact element-local couplings that parallelize well, and sharp resolution of layers via upwind and penalty fluxes; however, they introduce many more unknowns, in many cases require parameter tuning, are costlier in memory and time, and are not inherently monotone; overshoots/undershoots can appear without limiting. Coupling CCFV and DG capitalizes on their strengths, and results in a locally conservative scheme.  The analyses in \cite{chidyagwai2011coupling,riviere2014convergence} establish theoretical foundations for CCFV–DG coupling for linear convection-diffusion problems, and \cite{doyle2020multinumerics} demonstrates its effectiveness for nonlinear multiphase flow. Nevertheless, a key unresolved issue is how to decide in an automated way where within the computational domain to deploy FV versus DG, i.e., when and where each scheme should be activated.
  
%	We introduce a novel bound-preserving technique for the coupling of CCFV–DG for steady convection-diffusion problems. Overshoot and undershoot are present due to the high-order DG region. 
	We introduce a novel bound-preserving technique for coupling cell-centered finite volume (CCFV) with DG in steady convection-diffusion problems. It targets the spurious overshoots and undershoots that can arise in the high-order DG region.  One way to deal with these spurious oscillations is to leverage slope limiters \cite{berger2005analysis}.  Most popular slope limiters assume the cell averages are already correct and therefore preserve them during limiting. To repurpose conservative limiters, we focus on generating bound-preserving cell averages. Rather than relying on fix-ups or postprocessing of the DG cell averages, we perform a procedure similar to $p$-adaptivity. If the DG cell average bound-violating (up to some tolerance, e.g., a scaling of machine precision), we revert the approximation to CCFV in a small neighborhood of the troubled cell. Our working ansatz is that a suitable threshold between high-order DG and CCFV yields cell averages that are bound-preserving up to a prescribed tolerance.  The intuition is that if we use CCFV over the entire mesh, then the method will be bound-preserving.  Once satisfactory cell averages are obtained, any standard conservative slope limiters can be used to ensure the entire solution is bound-preserving.
	
%	It is well known that, for elliptic problems, the IPDG bilinear form with piecewise-constant trial and test spaces fails to support standard convergence analysis. In this case, we adopt a multinumerics scheme.  The classical two-point finite volume method is leveraged in lieu of piecewise constants. Here, there are additional constraints on the mesh, as the classical CCFV requires $K$-orthogonality for optimal convergence.
 
  % The approach we take is more a philosophical one: we do not manipulate the cell averages in any way. This is in contrast to postprocessings such as flux corrected transport. The cell averages resulting from the adaptivity technique are bound-preserving up to some specific tolerance.  In practice, we can always expect numerical round-off errors.
   
	Instead of manipulating cell averages via postprocessing, we leave them unchanged (we force the cell averages from the approximation to be bound-preserving, without manipulation \cite{fabien2025high,fabien2024positivity}). The adaptivity yields bound-preserving cell averages up to a prescribed tolerance (and, irrespective of adaptivity, cell averages are generally only accurate up to floating-point round-off).  Moreover, $hp$-adaptivity is widely recognized as a state-of-the-art strategy, often attaining a specified error tolerance with markedly fewer degrees of freedom~\cite{demkowicz2006computing}.  Our adaptive technique can easily be incorporated into standard $hp$-adaptivity routines, and does not modify the CCFV scheme, DG scheme, or postprocess cell averages.
   
	As an alternative, full DG frameworks may be employed where applicable (e.g., the hybridized LDG method \cite{cockburn2010hybridizable} permits piecewise-constant approximations, and leads to a finite volume analog \cite{sevilla2018face}).  In addition, the proposed adaptivity can be modified to ensure nodal bound-preservation.   
   
	We next survey some related works. Due to the large body of literature surrounding the subject, this survey is non-exhaustive. The work \cite{FRERICHS2021113487} develops nonintrusive postprocessings for DG solutions to suppress spurious oscillations. While the procedure is computationally cheap, it does not enforce bounds preservation. A style of FV-DG limiter computation is introduced in \cite{Dolejsi2002,Dolejsi2003}.  Here, the authors modify the convective form by using the $L^2$ projection of the high-order DG solution onto piecewise constants.  Their error indicator is based on the jump of the DG approximation across element interfaces, and this is used to determine which cells activate the modified convective form. 
 
	Several nonlinear approaches are known to perform well in mitigating spurious oscillations. Artificial viscosity approaches introduce extra diffusion into the discretization, in the hope that it will stabilize as well as prevent unphysical oscillations~\cite{guermond2011entropy,michoski2016comparison,yu2020study}. In flux-corrected transport, one uses a diffusive, positivity-preserving low-order scheme, adds a high-order correction to improve accuracy, and limits the resulting antidiffusive fluxes to avoid new oscillations~\cite{kuzmin2021new,joshaghani2022maximum}. Finally, we mention a review paper on nonlinear bound-preserving  modifications (such as variational-inequality formulations)~\cite{badia2015discrete,barrenechea2024finite}.  Each of the aforementioned approaches has trade-offs: some are inexpensive but not bound-preserving; others are bound-preserving but costly; some depart from the original DG discretization or lose conservation of mass; and most lack a formal error analysis.
 
\section{Model problems}   
   The model PDE considered in this work are detailed here.  Let $\Omega\subset\mathbb{R}^d$ ($d=2,3$) be a bounded polygonal domain.
   \begin{align} 
   -\nabla \cdot(K \nabla u) + \nabla \cdot (\vec{\beta}u) + c u & = f, &&\text{in } \Omega, \label{model_problem1}
   \\
	u &= g_D ,&&\text{on } \Gamma_D, \notag
	\\
-K\nabla u \cdot {\bm n} &= 0 ,&&\text{on } \Gamma_N \notag,
   \end{align} 
   with $\Omega\subset \mathbb{R}^n$ is a bounded domain with polyhedral Lipschitz boundary $\partial\Omega = \Gamma_D\cup \Gamma_N$ (and $\Gamma_D\cap \Gamma_N=\emptyset$). The remaining parameters are as follows: $K$ is the diffusion coefficient (bounded above/below by positive constants), $\vec{\beta}$ is the convection field, $c$ is the reaction coefficient, and $f$ is a source term.  In addition, we have nonhomogeneous Dirichlet boundary conditions $(g_D)$ imposed on $\Gamma_D$, and homogeneous Neumann conditions on $\Gamma_N$, where ${\bm n}$ is the unit outward normal to $\partial\Omega$.  The inflow boundary is defined as $\Gamma_{\text{inflow}} = \{x\in \Gamma: \vec{\beta}(x)\cdot {\bm n} <0\}\subset \Gamma_D$.    
\section{Preliminaries}
%The set \(\mathcal{T}_h\) is partitioned into two nonoverlapping submeshes, the DG region \(\mathcal{T}_h^{DG}\) and the CCFV region \(\mathcal{T}_h^{FV}\), so that \(\mathcal{T}_h = \mathcal{T}_h^{DG} \cup \mathcal{T}_h^{FV}\) and \(\mathcal{T}_h^{DG} \cap \mathcal{T}_h^{FV} = \emptyset\).  That is, the union of the DG and FV submeshes comprises of the entire mesh, and, the DG and CCFV regions are nonoverlapping. 
%\subsection*{Mesh and interface notation}
	The bounded polygonal domain $\Omega\subset\mathbb{R}^d$ is decomposed into a finite volume region $\Omega_F$ and a discontinuous Galerkin region $\Omega_D$ with matching interface $\Gamma_{DF} := \partial\Omega_F \cap \partial\Omega_D$, so that $\Omega = \Omega_F \cup \Omega_D$ and $\Omega_F \cap \Omega_D = \emptyset$. Let $\mathcal{T}_h^{FV}$ and $\mathcal{T}_h^{DG}$ be conforming partitions of $\Omega_F$ and $\Omega_D$, respectively. FV cells are denoted by $V\in\mathcal{T}_h^{FV}$ (Voronoi in the paper), and DG cells by $W\in\mathcal{T}_h^{DG}$ (simplices/hexahedra or Voronoi). Write $h_V:=\mathrm{diam}(V)$, $h_W:=\mathrm{diam}(W)$, $h_{FV}:=\max_{V\in\mathcal{T}_h^{FV}} h_V$, $h_{DG}:=\max_{W\in\mathcal{T}_h^{DG}} h_W$, and $h:=\max(h_{FV},h_{DG})$. For an edge/face $\gamma$, let $|\gamma|$ denote its $(d-1)$–dimensional measure. We split facets into interior and boundary parts:
\[
\Gamma_h^{FV} = \Gamma_{h,I}^{FV} \cup \Gamma_{h,\partial}^{FV},\qquad
\Gamma_h^{DG} = \Gamma_{h,I}^{DG} \cup \Gamma_{h,\partial}^{DG},\qquad
\Gamma_h^{FV-DG} := \{\gamma \subset \Gamma_h^{FV}\cap \Gamma_h^{DG}\}.
\]
Here $\Gamma_{h,I}^{FV}$ are edges interior to $\Omega_{FV}$, $\Gamma_{h,\partial}^{FV}$ lie on $\partial\Omega_{FV}\cap\partial\Omega$, and analogously for $DG$. We will also use the following notations
\begin{alignat*}{2}
  \Gamma_{h,\text{dir}}^{DG} & = \Gamma_{h}^{DG} \cap \Gamma_{D},  &\qquad
  \Gamma_{h,\text{dir}}^{FV} & = \Gamma_{h}^{FV} \cap \Gamma_{D}, \\[2pt]
  \Gamma_{h,ID}^{DG}         & = \Gamma_{h,I}^{DG} \cup \Gamma_{h,\text{dir}}^{DG}, &\qquad
  \Gamma_{h,ID}^{FV}         & = \Gamma_{h,I}^{FV} \cup \Gamma_{h,\text{dir}}^{FV}, \\[2pt]
  \Gamma_{h,\partial-}^{DG}  & = \Gamma_{h}^{DG} \cap \Gamma_{\text{inflow}}, &\qquad
  \Gamma_{h,\partial-}^{FV}  & = \Gamma_{h}^{FV} \cap \Gamma_{\text{inflow}}.
\end{alignat*}

For the FV region, we require a so-called two-point flux (TPFA) admissible mesh \cite{EYMARD200131}.
\begin{assumption}[TPFA admissibility]\label{ass:tpfa-admissible}
Let \(\mathcal{T}_h^{FV}\) be a mesh for the finite volume region, and $V,W\in \mathcal{T}_h^{FV}$. We say that \(\mathcal{T}_h^{FV}\) is an admissible mesh, in the classical two-point flux finite volume sense if:
\begin{enumerate}
  \item (Cell centers and orthogonality) There exist points \(\{x_V\}_{V\in\mathcal{T}_h^{FV}}\) with \(x_V\in V\) such that, for every interior edge \(\gamma=\partial V\cap\partial W\) with \(V\neq W\), the line segment between \( x_V \) and \( x_W \) is orthogonal to \(\gamma\).
  
  \item (Boundary edges) For each boundary edge \(\gamma=\partial V\cap\partial\Omega\), one has \(x_V\notin\gamma\). Define \(y_\gamma\) as the orthogonal projection of \(x_V\) onto \(\gamma\).% (This requirement may be relaxed; see Remark~1, Section~3.)
\end{enumerate}
\end{assumption}
	Although Assumption~\ref{ass:tpfa-admissible} may appear to be restrictive, robust unstructured Voronoi generators are well established~\cite{ebeida2011uniform,weyer2002automatic}.  In addition, enforcing Assumption~\ref{ass:tpfa-admissible} ensures optimal convergence of the CCFV method. Using meshes that do not satisfy the assumption forfeits any convergence guarantee. Fig.~\ref{fig:mesh_ex} has an examples Voronoi meshes. An unstructured Voronoi mesh is given in Fig.~\ref{figmesh:1}, and Fig.~\ref{figmesh:2} has an example FV/DG partitioning.
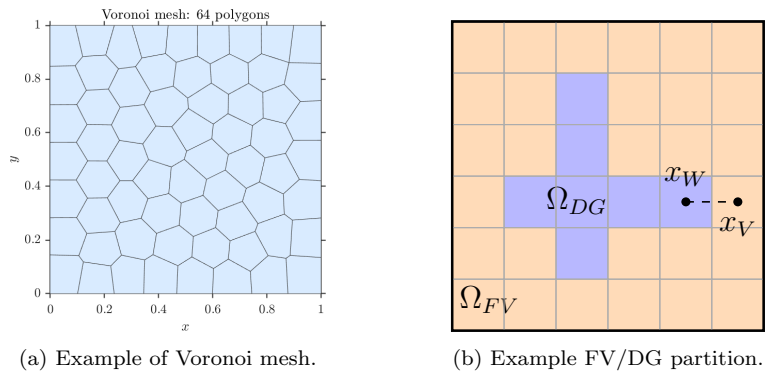
\begin{figure}[H]
  \centering

  \begin{subfigure}[t]{0.49\textwidth}
    \centering
\includegraphics[scale=0.4]{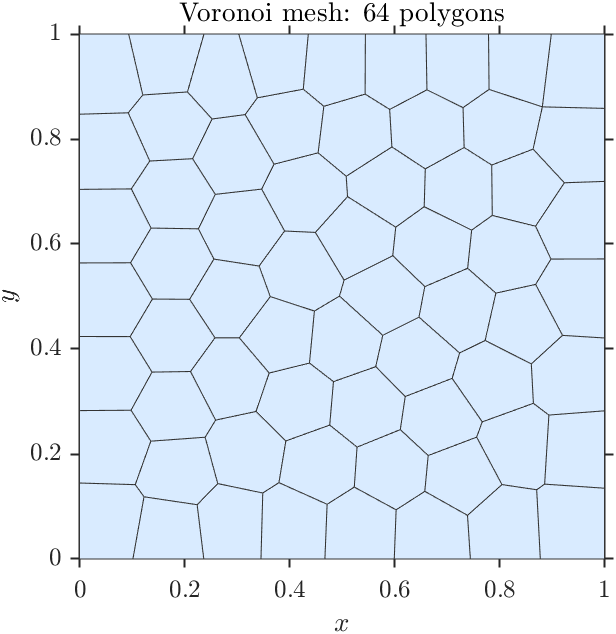}
    \subcaption{Example of Voronoi mesh.}\label{figmesh:1}
  \end{subfigure} %~ %\hfill
  \hspace*{-14ex}
  \begin{subfigure}[t]{0.49\textwidth}
    \centering
    % \resizebox{\linewidth}{!}{ ... }  % optional autosize
\begin{tikzpicture}[scale=4.1]
  % styles
  \tikzset{
    grid/.style={step=1/6, line width=0.4pt, gray!65},
    boundary/.style={line width=0.9pt}
  }

  % --- fill 5 nonoverlapping subdomains, each a union of 1/6×1/6 squares ---
  % Ω1: a central plus-shape (7 cells)
  \foreach \i/\j in {1/2, 2/2, 3/2, 4/2, 2/1, 2/3, 2/4}{
    \fill[blue!28] (\i/6,\j/6) rectangle ++(1/6,1/6);
  }

  % Ω2: top-left non-rectangular cluster (7 cells)
  \foreach \i/\j in {0/3, 1/3, 0/4, 1/4, 0/5, 1/5, 2/5}{
    \fill[orange!28] (\i/6,\j/6) rectangle ++(1/6,1/6);
  }

  % Ω3: top-right L-shape (8 cells)
  \foreach \i/\j in {3/4, 3/5, 4/3, 4/4, 4/5, 5/3, 5/4, 5/5}{
    \fill[orange!28] (\i/6,\j/6) rectangle ++(1/6,1/6);
  }

  % Ω4: bottom-left L-shape (6 cells)
  \foreach \i/\j in {0/0, 0/1, 0/2, 1/0, 1/1, 2/0}{
    \fill[orange!28] (\i/6,\j/6) rectangle ++(1/6,1/6);
  }

  % Ω5: remaining non-rectangular cluster (8 cells)
  \foreach \i/\j in {3/0, 4/0, 5/0, 3/1, 4/1, 5/1, 5/2, 3/3}{
    \fill[orange!28] (\i/6,\j/6) rectangle ++(1/6,1/6);
  }

%\fill (0.3,0.6) circle (0.5pt);            % solid dot 
%\draw[fill=black] (0.3,0.6) circle (0.1pt);  % hollow dot

  % --- overlay the 6×6 Cartesian grid and outer boundary ---
  \draw[grid] (0,0) grid (1,1);
  \draw[boundary] (0,0) rectangle (1,1);

  % --- optional labels for subdomains ---
  \node at (0.40,0.42) {$\Omega_{DG}$};
%  \node at (0.17,0.80) {$\Omega_2$};
%  \node at (0.82,0.82) { };
  \node at (0.12,0.10) {$\Omega_{FV}$};
%  \node at (0.82,0.22) { };

  \draw[step=1/6, gray!65, line width=0.4pt] (0,0) grid (1,1);
  \draw[line width=0.8pt] (0,0) rectangle (1,1);
  \tikzset{dot/.style={circle, fill=black, inner sep=1.2pt}}

  % choose base square indices (0..5) and an adjacent offset (±1,0) or (0,±1)
  \def\i{4}\def\j{2}      % base square
  \def\di{1}\def\dj{0}    % neighbor: right (use -1,0 for left; 0,1 up; 0,-1 down)

  \coordinate (C1) at ({(\i+0.5)/6},{(\j+0.5)/6});
  \coordinate (C2) at ({(\i+\di+0.5)/6},{(\j+\dj+0.5)/6});

  \node[dot, label=above:$x_W$] at (C1) {};
  \node[dot, label=below:$x_V$] at (C2) {};
  \draw[dashed, line width=0.6pt] (C1) -- (C2);                  % thicker
\end{tikzpicture}
    \subcaption{Example FV/DG partition.}\label{figmesh:2}
  \end{subfigure}

  \caption{Example Voronoi meshes and illustration of FV (orange) and DG regions (blue), as well as the $K$-orthogonality property.}
  \label{fig:mesh_ex}
\end{figure}	
	
%\begin{figure}[H]
%\centering
%\includegraphics[scale=0.4]{voronoi_mesh_example.png}
%  \caption{Example of a $K$-orthogonal mesh (Voronoi).}
%  \label{figmesh:1}
%\end{figure} 
%\paragraph{Two–point FV admissibility (FV nodes and orthogonality).}
%For each FV cell $V$ choose a point $x_V\in V$ (the FV “node”). The FV mesh is
%\emph{admissible} in the two–point sense if:
%(i) for every interior FV edge $\gamma=\partial V\cap\partial W$ with $V\neq W$,
%the segment joining $x_V$ and $x_W$ is orthogonal to $\gamma$; and
%(ii) for every boundary FV edge $\gamma=\partial V\cap\partial\Omega$, the point $x_V$
%does not lie on $\gamma$. In this case, let $y_\gamma$ be the (unique) point where the line
%through $x_V$ orthogonal to $\gamma$ meets $\gamma$.

For any facet $\gamma$ fix a unit normal ${\bm n}_\gamma$. On $\partial\Omega$ take ${\bm n}_\gamma$ outward; on $\Gamma_h^{FV-DG}$ orient ${\bm n}_\gamma$ from the DG side into the FV side; on interior edges choose either orientation consistently.  When $\gamma=\partial V\cap\partial W$ lies in $\Gamma_h^{DG}$, set $h_\gamma := \max\{\mathrm{diam}(V),\mathrm{diam}(W)\}$ (used in DG penalties).

%\paragraph{Distances associated with FV edges.}
Define the scalar $d_\gamma$ for each facet:
\[
d_\gamma :=
\begin{cases}
\|x_V - x_W\|_2, & \gamma=\partial V\cap\partial W \in \Gamma_{h,I}^{FV},\\[2pt]
\mathrm{dist}(x_V,\gamma)=\|x_V-y_\gamma\|_2, & \gamma=\partial V\cap\partial\Omega \in \Gamma_{h,\partial}^{FV},\\[2pt]
\|x_V-y_\gamma\|_2, & \gamma=\partial V\cap\partial W \in \Gamma_h^{FV-DG},
\end{cases}
\]
and assume a uniform shape-regularity: there exists $\theta>0$ such that
\[
\begin{aligned}
&\gamma=\partial V\cap\partial W\in \Gamma_{h,I}^{FV}:\quad d_\gamma \ge \theta\,\max(h_V,h_W),\\
&\gamma=\partial V\cap\partial\Omega\in \Gamma_{h,\partial}^{FV}:\quad d_\gamma \ge \theta\,h_V,\\
&\gamma=\partial V\cap\partial W\in \Gamma_h^{FV-DG}:\quad d_\gamma \ge \theta\,h_V.
\end{aligned}
\]

%%\paragraph{Edgewise harmonic averages of $K$.}
%%For each FV or FV–DG interface facet, define the harmonic average taken along the orthogonal segment used in the two–point flux:
%%For each FV (or FV-DG interface) facet $\gamma$, define the harmonic average along the normal line used by the two–point flux, i.e., the straight line orthogonal to $\gamma$ connecting the relevant FV points (interior) or the FV point to the footpoint on $\gamma$ (boundary/interface).
%In TPFA, averaging is performed along the normal line. For interior facets $\gamma=\partial V\cap\partial W$, this is the line segment joining $x_V$ and $x_W$ orthogonal to $\gamma$. For boundary or FV-DG interface facets, it is the line from $x_V$ to the orthogonal projection $y_\gamma\in\gamma$.
%\[
%K_\gamma :=
%\begin{cases}
%\displaystyle
%d_\gamma\Bigl(\int_{x_V}^{\,x_W} \tfrac{\mathrm{d}s}{K(s)}  \Bigr)^{-1},
%& \gamma=\partial V\cap\partial W \in \Gamma_{h,I}^{FV},\\[10pt]
%\displaystyle
%d_\gamma\Bigl(\int_{x_V}^{\,y_\gamma} \tfrac{\mathrm{d}s}{K(s)}  \Bigr)^{-1},
%& \gamma\in \Gamma_{h,\partial}^{FV}\ \text{or}\ \gamma\in \Gamma_h^{FV-DG}.
%\end{cases}
%\]
%These $K_\gamma$ inherit the same upper/lower bounds as $K$.  The finite volume literature often refers to $(|\gamma|/d_\gamma)K_\gamma$ as the TPFA face transmissibility~\cite{EYMARD2000713}.

For each FV facet $\gamma$ (including FV--DG interface facets), let $L_\gamma$ denote the
normal line used by the two--point flux (the straight line orthogonal to $\gamma$):
\[
L_\gamma \;=\;
\begin{cases}
[x_V,x_W], & \gamma=\partial V\cap\partial W \in \Gamma_{h,I}^{F},\\[3pt]
[x_V,y_\gamma],\ \ y_\gamma:=\operatorname{proj}_\gamma(x_V), 
& \gamma\in \Gamma_{h,\partial}^{F}\ \text{or}\ \gamma\in \Gamma_h^{DF},
\end{cases}
\]
and set $d_\gamma := |L_\gamma|$. 
The harmonic normal average along $L_\gamma$ is
\[
K_\gamma \;:=\; d_\gamma \Biggl(\int_{L_\gamma} \frac{\mathrm{d}s}{K(s)}   \Biggr)^{-1}.
\]
These $K_\gamma$ inherit the same upper/lower bounds as $K$.  The finite volume literature often refers to $(|\gamma|/d_\gamma)K_\gamma$ as the TPFA face transmissibility~\cite{EYMARD2000713}.
%For tensor coefficients $\nu(x)$, replace $K(s)$ by the normal projection 
%$n_\gamma^\top \nu(s)\,n_\gamma$:
%\[
%K_\gamma \;:=\; d_\gamma \Biggl(\int_{L_\gamma} 
%\frac{1}{\,n_\gamma^\top \nu(s)\,n_\gamma\,}\,\mathrm{d}s \Biggr)^{-1}.
%\]
% (Optional) The face transmissibility is $T_\gamma = \dfrac{|\gamma|}{d_\gamma}\,K_\gamma$.
 
\section{CCFV, DG, and coupled discretization}
 
% =========================
% Function spaces
% =========================
%\paragraph{Discrete space.}
For a fixed polynomial degree \(k\ge 1\) in the DG region, define
\[
X_h
:= \bigl\{\, v\in L^2(\Omega)\ :\ v|_W\in \mathcal{P}_k(W)\ \,\forall W\in\mathcal T_h^{DG},\ 
                                  v|_V\in \mathcal{P}_0(V)\ \,\forall V\in\mathcal T_h^{FV} \bigr\},
\]
 where \(\mathcal{P}_k(W)\) denotes the space of polynomials on \(W\) of total degree at most \(k\). 
 
  Functions in $X_h$ can be discontinuous across element interfaces, so we require some jump and average notations to provide meaning to trace values. To make the exposition more clear, given any edge $\gamma$, we fix a unit normal vector ${\bm n}_\gamma$ to $\gamma$.  More specifically, we assume that if $\gamma$ is a boundary facet (belongs to $\partial\Omega$), then ${\bm n}_\gamma$ points outward of $\partial\Omega$. If $\gamma$ belongs to the interface $\Gamma_h^{FV-DG},$ then we assume that ${\bm n}_\gamma$ points from the DG region into the FV region. In the following, put $W,V\in \mathcal{T}_h$ so that the vector ${\bm n}_\gamma$ points from $\partial V$ into $\partial W$. Given $u \in X_h$, the jump of $u$ is defined to be:
\[
\begin{aligned}
& \gamma \in\Gamma_{h,I}^{DG} :\quad
  [u]_\gamma := u|_{V} - u|_{W},\qquad
  %\{u\}_\gamma := \tfrac12\bigl(u|_{V}+u|_{W}\bigr),
  \\[2pt]
& \gamma\in\Gamma_{h,\partial}^{DG} :\quad
  [u]_\gamma := u|_{V},\qquad
 % \{u\}_\gamma := u|_{V},
 \\[2pt]
& \gamma \in\Gamma_{h,I}^{FV} 
:\quad
  [u]_\gamma := u(x_V) - u(x_W),
  \\[2pt]
& \gamma \in\Gamma_{h,\partial}^{FV} :\quad
  [u]_\gamma := u(x_V),
  \\[2pt]
& \gamma \in\Gamma_{h}^{FV-DG} :\quad
  y_\gamma:=\operatorname{proj}_\gamma(x_V),\quad
  [u]_\gamma := u|_{\Omega_{DG}}(y_\gamma) - u|_{\Omega_{FV}}(x_V).
\end{aligned}
\]
The IPDG method requires the notion of the average of $u\in X_h$.  Suppose $\gamma$ is an interior facet in the DG region, with $\gamma = E_1\cap E_2$, and $E_1,E_2\in \mathcal{T}_h^{DG}$. We then define
\[ 
\{u\}_\gamma := \tfrac12\bigl(u|_{E_1}+u|_{E_2}\bigr).
\]
Similarly, if $\gamma$ is a boundary facet in the DG region, with $\gamma \subset \partial E_1 $, and $E_1 \in \mathcal{T}_h^{DG}$, we define
\[ 
\{u\}_\gamma :=  u|_{E_1} .
\]
Given an interior facet $\gamma$, shared by elements $V$ and $W$ (with ${\bm n}_\gamma$ oriented from $V$ to $W$), we define the upwind value of $u$ to be
\[
u^{\uparrow} =
\begin{cases}
u|_V &\mbox{if } \vec{\beta}\cdot {\bm n}_\gamma\ge 0,
\\
u|_W&\mbox{if } \vec{\beta}\cdot {\bm n}_\gamma <0.
\end{cases} 
\]
% =========================
% Bilinear forms
% =========================
For $E_1,E_2\in \mathcal{T}_h^{DG}$, let $\gamma=\partial E_1 \cap \partial E_2$, and $h_\gamma = \text{max}(\text{diam}(V), \text{diam}(W))$.  With parameters \(\sigma>0\) and \(\epsilon\in\{-1,0,1\}\), define the following bilinear forms for \(u,v\in X_h\):
\begin{align*}
 a_{DG}(u,v) &=
\sum_{W\in\mathcal{T}_h^{DG}} \int_W K\,\nabla u\cdot\nabla v
%+ \sum_{\gamma\in\Gamma_h^{DG}} \int_\gamma u^{\uparrow}[v]\,(\vec{\beta}\cdot n_\gamma)
- \sum_{\gamma\in\Gamma_{h,ID}^{DG}} \int_\gamma \{K\nabla u\cdot n_\gamma\}[v]
+ \epsilon \sum_{\gamma\in\Gamma_{h,ID}^{DG}} \int_\gamma \{K\nabla v\cdot n_\gamma\}[u]
\\
&+ \sum_{\gamma\in\Gamma_{h,ID}^{DG}} \frac{\sigma}{h_\gamma}\int_\gamma [u][v]
-  \sum_{W\in\mathcal{T}_h^{DG}} \int_W  (\vec{\beta}\cdot \nabla v) u
+
\sum_{\gamma \in\Gamma_{h,I}^{DG}} \int_\gamma (\vec{\beta}\cdot {\bm{n}_\gamma}) u^\uparrow [v]_\gamma
\\
&+
\sum_{\gamma \in\Gamma_{h,\partial-}^{DG}} \int_\gamma (\vec{\beta}\cdot {\bm{n}_\gamma}) u v
+
\sum_{W\in\mathcal{T}_h^{DG}} \int_W c u v,
\end{align*}
\begin{align*}
 a_{FV}(u,v) &= 
\sum_{\gamma\in\Gamma_{h,ID}^{FV}} \frac{|\gamma|}{d_\gamma}\,K_\gamma [u]_\gamma [v]_\gamma,
+
\sum_{\gamma\in\Gamma_{h,I}^{FV}} \int_\gamma (\vec{\beta}\cdot {\bm{n}_\gamma}) u^\uparrow [v]_\gamma
+
\sum_{\gamma\in\Gamma_{h,\partial-}^{FV}} \int_\gamma (\vec{\beta}\cdot {\bm{n}_\gamma}) uv
\\
&+
 \sum_{V\in \mathcal{T}_h^{FV}} \int_V c uv,
\\
a_{FG}(u,v) &= 
\sum_{\gamma\in\Gamma_h^{FV-DG}} \frac{|\gamma|}{d_\gamma}\,K_\gamma [u]_\gamma [v]_\gamma
+
\sum_{\gamma\in\Gamma_h^{FV-DG}} \int_\gamma (\vec{\beta}\cdot {\bm{n}_\gamma}) u^\uparrow [v]_\gamma,
\\
a(u,v) &= a_{DG}(u,v) + a_{FV}(u,v) + a_{FG}(u,v) .
\end{align*}
For the linear form, we have
\begin{align*}
\ell(v)
&=
\epsilon
\sum_{\gamma\in \Gamma_{h,\text{dir}}^{DG}} \int_\gamma \frac{\sigma}{h_\gamma}  (\nabla v \cdot {\bm{n}_\gamma})  g_D
+
\sum_{\gamma\in \Gamma_{h,\text{dir}}^{DG}} \int_\gamma \frac{\sigma}{h_\gamma} g_D v
-
\sum_{\gamma\in \Gamma_{h,\partial-}^{DG}} \int_\gamma (\vec{\beta}\cdot {\bm{n}_\gamma}) g_D v
\\
&+ 
\sum_{W\in\mathcal{T}_h^{DG}} \int_W f v
+
\sum_{\gamma\in \Gamma_{h,\text{dir}}^{FV}} \frac{|\gamma|}{d_\gamma} K_\gamma g_D v|_\gamma
-
\sum_{\gamma\in \Gamma_{h,\partial-}^{FV}} \int_\gamma (\vec{\beta}\cdot {\bm{n}_\gamma}) g_D v.
\end{align*}
	The discrete coupled FV-DG scheme for \eqref{model_problem1} is: find $u_h\in X_h$ so that
\begin{align}
a(u_h,v_h) = \ell(v_h),\quad \forall v_h \in X_h.
\label{eq:discrete-problem}
\end{align}
	If \( \mathcal{T}_h^{DG}=\emptyset \), then~\eqref{eq:discrete-problem} reduces to the classical cell-centered finite-volume method (CCFV).  Similarly, if \( \mathcal{T}_h^{FV}=\emptyset \), then~\eqref{eq:discrete-problem} reduces to an interior-penalty DG (IPDG) method with upwinding for the convection.  In Section~\ref{sec:numerics}, we compare and contrast the three schemes: full CCFV \( (\mathcal{T}_h^{DG}=\emptyset) \), full DG \( (\mathcal{T}_h^{FV}=\emptyset) \), and FV--DG \( (\mathcal{T}_h^{FV}\neq\emptyset,\ \mathcal{T}_h^{DG}\neq\emptyset) \).
%\clearpage  
\section{Description of adaptive technique} \label{sec:adapt}  
	Here we describe the bound-preserving mechanism used in this paper.  The method runs in two phases: 
\begin{enumerate}[(i)]
\item perform standard $h$-adaptivity until the accuracy estimator meets tolerance; 
\item repeatedly scan cell averages against the admissible interval $I=[\,u_*-\delta_{\text{under}},\,u^*+\delta_{\text{over}}\,]$. 
\end{enumerate}	 Any violating cells (and their facet/vertex neighbors, see Fig.~\ref{fig_algo_descript}) are reassigned to FV, and update the partition: $\mathcal T_h^{FV} \cup \mathcal T_h^{DG}$. The FV and DG regions are updated until the bound-violation metric falls below tolerance, then apply a conservative slope limiter on the DG region.  We also consider a slight modification that scans nodal values instead of cell averages (see Variant 2 in subsection~\ref{sec:var2}).
 	 
\subsection{Variant 1: Bound-driven FV/DG partitioning (fixed slack)}\label{sec:var1}
Pick slack parameters $\delta_{\text{under}},\delta_{\text{over}}>0$ and tolerances
$\varepsilon_h,\varepsilon_{\mathrm{bp}}>0$. Define the admissible interval
$I=[\,u_*-\delta_{\text{under}},\,u^*+\delta_{\text{over}}\,]$ and let
$\mathcal{N}(E)$ denote the set that contains the facet and vertex neighbors of $E$. 
\begin{enumerate}
  \item \textbf{Accuracy phase ($h$-adapt).} While $\eta_h(\mathcal T_h)>\varepsilon_h$,
        refine/coarsen $\mathcal T_h$ and recompute $u_h$.

  \item \textbf{Bound-screened FV/DG partitioning.} Repeat \emph{until}
        $\eta_{\mathrm{bp}}\le \varepsilon_{\mathrm{bp}}$:
    \begin{enumerate}
      \item Compute cell averages $\overline{u}_h|_E$ for all $E\in\mathcal T_h$.
      \item Form the violation set $V:=\{\,E:\ \overline{u}_h|_E\notin I\,\}$ and expand
            $V \leftarrow V \cup \mathcal{N}(V)$.
      \item Set $\mathcal T_h^{FV}\leftarrow V$ and
            $\mathcal \mathcal T_h^{DG}\leftarrow \mathcal T_h\setminus V$; recompute $u_h$ and evaluate
            \[
              \eta_{\mathrm{bp}}
              := \max_{E\in\mathcal T_h}\textrm{dist}(\overline{u}_h|_E, I),
              \qquad
              \operatorname{dist}(x,I):= \max\{\,u_* - x,\ 0,\ x - u^*\,\}.
            \]
    \end{enumerate} 
  \item \textbf{DG limiting.} Apply a conservative slope limiter on $\mathcal T_h^{D}$.
\end{enumerate}
 
\subsection{Variant 2: Node-screened FV/DG partitioning.}\label{sec:var2}
Let $I=[\,u_*-\delta_{\text{under}},\,u^*+\delta_{\text{over}}\,]$ and, for each cell $E\in\mathcal T_h$, let
$\mathcal X(E)$ denote the set of nodal points to be screened (e.g., the vertices of $E$, or quadrature points). Define
\[
\eta_{\mathrm{bp}}^{\mathrm{node}}
:= \max_{E\in\mathcal T_h}\ \max_{x\in \mathcal X(E)}
   \operatorname{dist}(u_h|_E(x),\, I),
\qquad
\operatorname{dist}(z,I):=\max\{\,u_* - z,\ 0,\ z - u^*\,\}.
\] 
\begin{enumerate}
  \item \textbf{Accuracy phase ($h$-adapt).}
        While $\eta_h(\mathcal T_h)>\varepsilon_h$, refine/coarsen $\mathcal T_h$ and recompute $u_h$.

  \item \textbf{Bound-screened FV/DG partitioning (node-based).} Repeat \emph{until}
        $\eta_{\mathrm{bp}}^{\mathrm{node}}\le \varepsilon_{\mathrm{bp}}$:
    \begin{enumerate}
      \item For every $E\in\mathcal T_h$, evaluate $u_h|_E(x)$ at all $x\in\mathcal X(E)$.
      \item Form the violation set
            \[
              V := \Bigl\{\,E\in\mathcal T_h:\ \exists\,x\in\mathcal X(E)\ \text{s.t.}\ 
              u_h|_E(x)\notin I \Bigr\}.
            \]
      \item Expand $V \leftarrow V \cup \mathcal{N}(V)$ and set
            $\mathcal T_h^{FV}\leftarrow V$, \quad $\mathcal T_h^{DG}\leftarrow \mathcal T_h\setminus V$.
      \item Recompute $u_h$ on the new partition and update
            $\eta_{\mathrm{bp}}^{\mathrm{node}}$.
    \end{enumerate}
  \item \textbf{DG limiting.} Apply a conservative slope limiter on $\mathcal T_h^{DG}$.
\end{enumerate} 
\begin{figure}[H]
\centering  
\begin{tikzpicture}[scale=4]
  % styles
  \tikzset{
    face/.style={line width=0.6pt},
    site/.style={circle, fill=none, inner sep=0.4pt}
  }

  % unit square boundary
  \draw[face] (0,0) rectangle (1,1);

  % Voronoi faces: vertical lines at x = k/8, k=1..7
  \foreach \k in {1,...,7} \draw[face] (\k/8,0) -- (\k/8,1);

  % Voronoi faces: horizontal lines at y = k/8, k=1..7
  \foreach \k in {1,...,7} \draw[face] (0,\k/8) -- (1,\k/8);

  % ------------ highlight one element and its vertex-neighbors ------------
  % pick the target cell by its grid indices (ix,iy) with 0<=ix,iy<=7
  \def\ix{4}\def\iy{4} % this is the cell between x in [0.5,0.625], y in [0.5,0.625]

  % main element (blue)
  \fill[blue!85,opacity=0.35]
    (\ix/8,\iy/8) rectangle ({(\ix+1)/8},{(\iy+1)/8});

% --- label the blue element (centered) ---
\coordinate (Ecenter) at ({(\ix+0.5)/8},{(\iy+0.5)/8});
\node[font=\scriptsize, fill=none, inner sep=1pt, text=black]
      at (Ecenter) {$E$};

  % all vertex-neighbor cells (green): all 8 neighbors around (ix,iy)
  \foreach \dx in {-1,0,1}{
    \foreach \dy in {-1,0,1}{
      \ifnum\dx=0\relax\ifnum\dy=0\relax
        % skip the center cell itself
      \else
        \fill[green!35,opacity=0.35]
          ({(\ix+\dx)/8},{(\iy+\dy)/8})
          rectangle
          ({(\ix+\dx+1)/8},{(\iy+\dy+1)/8});
      \fi\else
        \fill[green!35,opacity=0.35]
          ({(\ix+\dx)/8},{(\iy+\dy)/8})
          rectangle
          ({(\ix+\dx+1)/8},{(\iy+\dy+1)/8});
      \fi
    }
  }

  % 64 sites at centers of the 8×8 blocks: ((2i-1)/16, (2j-1)/16)
  \foreach \i in {1,...,8}{
    \foreach \j in {1,...,8}{
      \node[site] at ({(2*\i-1)/16},{(2*\j-1)/16}) {};
    }
  }
\end{tikzpicture}
\caption{Element $E$ marked for adaptivity (blue), and its vertex neighborhood $\mathcal{N}(E)$ (green).}
\label{fig_algo_descript}
\end{figure}
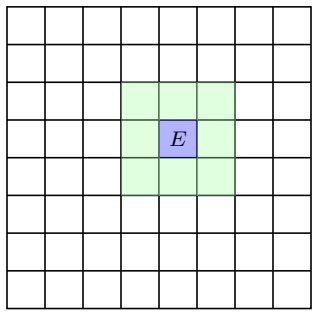
  \noindent\emph{Remark 1.}  Step~3 may be skipped if it is enough that the nodal values on $\mathcal X(E)$ are bound-preserving. For $k>1$ DG, nodal bounds do not preclude between node overshoots; $u_h|_E(x)$ can violate the bounds for $x\in E\setminus \mathcal X(E)$.
\\[12pt]
 \noindent\emph{Remark 2.} Through various numerical tests, we find that taking $\mathcal{N}(E)$ to be the vertex neighbors of $E$ provides good results. 
 
\subsection{Slope limiter} \label{sec:limiter}
To make the exposition more clear, let $\Omega \subset\mathbb{R}^2$.  We use a vertex limiter from \cite{aizinger2017anisotropic}.  For an element \(E\in\mathcal T_h\) with centroid \(c_0=(x_0,y_0)\) and vertices
\(\{v_i=(x_i,y_i)\}_{i=1}^M\), the piecewise linear DG solution is
\[
u_h|_E(x,y) \;=\; \bar u_E \;+\; u_x\,(x-x_0) \;+\; u_y\,(y-y_0),
\]
where  \((u_x,u_y)=\nabla u_h|_E\) is constant on \(E\).  For each vertex \(v_i\), let \(\omega(v_i)\) be the patch of elements that share \(v_i\).
Define vertexwise bounds
\[
u_i^{\min} \;=\; \min_{L\in \omega(v_i)} \bar u_L, \qquad
u_i^{\max} \;=\; \max_{L\in \omega(v_i)} \bar u_L.
\]
If \(v_i\) lies on a Dirichlet boundary, also include the boundary value \(g_D(v_i)\) when forming \(u_i^{\min},u_i^{\max}\).  The objective of the limiter is to construct a new polynomial
\[
\tilde{u}_h|_E(x,y) \;=\; \bar u_E \;+\; \alpha_x\,u_x\,(x-x_0) \;+\; \alpha_y\,u_y\,(y-y_0),
\qquad \alpha_x,\alpha_y\in[0,1],
\]
such that the vertex values satisfy the discrete maximum principle
\[
u_i^{\min} \;\le\; \tilde{u}_h(v_i) \;\le\; u_i^{\max} \quad \text{for all } i=1,\dots,M.
\]
The scalars \(\alpha_x\) and \(\alpha_y\) are computed as:
\begin{enumerate} 
\item  Choose \(\alpha_x\in[0,1]\) so that
\(u_i^{\min} \le u_0 + \alpha_x\,u_x\,(x_i-x_0) \le u_i^{\max}\) for all \(i\).
Define nodal factors
\[
\alpha_{x,i} \;=\;
\begin{cases}
\min\!\Bigl(1,\;\dfrac{u_i^{\max}-u_0}{\,u_x\,(x_i-x_0)\,}\Bigr), & u_x\,(x_i-x_0) > 0,\\[6pt]
1, & u_x\,(x_i-x_0) = 0,\\[6pt]
\min\!\Bigl(1,\;\dfrac{u_i^{\min}-u_0}{\,u_x\,(x_i-x_0)\,}\Bigr), & u_x\,(x_i-x_0) < 0,
\end{cases}
\qquad
\alpha_x \;=\; \min_{1\le i\le M}\alpha_{x,i}.
\]
Set the prelimiting polynomial \(\hat u_h(x,y)=u_0+\alpha_x u_x(x-x_0)\) and its vertex values
\(\hat u_i=\hat u_h(v_i)\), which are used in step 2.

\item Choose \(\alpha_y\in[0,1]\) so that
\(u_i^{\min} \le \hat u_i + \alpha_y\,u_y\,(y_i-y_0) \le u_i^{\max}\) for all \(i\).
Define
\[
\alpha_{y,i} \;=\;
\begin{cases}
\min\!\Bigl(1,\;\dfrac{u_i^{\max}-\hat u_i}{\,u_y\,(y_i-y_0)\,}\Bigr), & u_y\,(y_i-y_0) > 0,\\[6pt]
1, & u_y\,(y_i-y_0) = 0,\\[6pt]
\min\!\Bigl(1,\;\dfrac{u_i^{\min}-\hat u_i}{\,u_y\,(y_i-y_0)\,}\Bigr), & u_y\,(y_i-y_0) < 0,
\end{cases}
\qquad
\alpha_y \;=\; \min_{1\le i\le M}\alpha_{y,i}.
\]

\item Form the limited polynomial:
\[
\bar u_h|_E(x,y) \;=\; \bar u_E \;+\; \alpha_x\,u_x\,(x-x_0) \;+\; \alpha_y\,u_y\,(y-y_0).
\] 
\end{enumerate}
  
\begin{theorem}
Let $u_h$ be the solution to the FV-DG problem in~\eqref{eq:discrete-problem} after application of the adaption technique Variant 1 from Section~\ref{sec:adapt}).  Then, the cell-averages of $u_h$ are bound-preserving up to some tolerance $\delta$.
\end{theorem} 
 \begin{proof}
 Variant 1 of the adaption technique is guaranteed to terminate in a finite number of steps because the pessimistic choice $\mathcal{T}_h^{DG} = \emptyset$ (expand the neighborhoods until $\mathcal{T}_h=\mathcal{T}_h^{DG}\cup \mathcal{T}_h^{FV} = \mathcal{T}_h^{FV} $) ensures the scheme reverts to CCFV on the entire mesh; and the CCFV discretization is monotone~\cite{godunov1959difference,EYMARD2000713}.
 \end{proof}
  \noindent\emph{Remark 3.} Empirically, we find that the DG region is typically sizable (30\%-90\% of mesh elements), though this varies with the tolerance $\delta$, the mesh size, and the underlying problem.

\section{Numerical experiments}  \label{sec:numerics} 
Unless otherwise stated, all experiments for DG have $k=1$ (piecewise linear approximation), $\epsilon=-1$ (symmetric IPDG) and $\sigma=14$.  All computational meshes are unstructured (Voronoi) generated with CGAL~\cite{fabri2009cgal} and PolyMesher~\cite{talischi2012polymesher}.  The meshes are matching at the interface between the FV and DG regions.

	Variant~1 from Section~\ref{sec:adapt} governs the cell averages, while the limiter in Section~\ref{sec:limiter} constrains the higher–order DG modes/slopes.	While distinct slack parameters for undershoot and overshoot are admissible, in the experiments we take a single value, $\delta := \delta_{\text{under}} = \delta_{\text{over}}$.  We take $f=c=0$ in the model PDE~\eqref{model_problem1}. A suite of standard benchmark problems is presented in Sections~\ref{sec:Triple}, \ref{num_L}, \ref{sec:badia}, and \ref{num_H}.  Section~\ref{num_M} summarizes the results and presents supplementary metrics: iteration counts, and the size of spurious oscillations.

%\clearpage
\section{Triple layer}  \label{sec:Triple}
	We consider a triple layer problem, which is a common benchmark for convection-dominated flows.  No exact solution can be easily written down for this problem. The domain is $\Omega=[0,2]^2$, and $K=10^{-6}$, with the advective field $\vec{\beta}=[y,(1-x)^2]^T$.  On the domain $\Gamma_N:= \{2\} \times (0,1)\cup (0,2) \times\{1\}$ homogeneous Neumann conditions are imposed.  For the remaining boundaries, Dirichlet conditions are imposed:
\[
g_D(x,y) =
\begin{cases}
1, \mbox{ if $x\in (1/8,1/2),~y=0$},\\
2,  \mbox{ if $x\in (1/2,3/4),~y=0$}, \\
0, \mbox{ elsewhere on $\Gamma_D\backslash\Gamma_N$. }
\end{cases}
\]  
Due to the small amount of diffusion the discontinuous profile is basically transported along the characteristic curves leading to sharp characteristic interior layers.  %The solution satisfies $0\le u \le 2$, but has regions where it is nearly constant with value two, one, or zero.  Simply enforcing $0\le \overline{u} \le 2$ may not completely eliminate spurious oscillations.

Baseline tests for this benchmark are conducted in Fig.~\ref{figT_1:sixgrid}.  The CCFV method is monotone and bound-preserving, at the cost of reduced feature sharpness. Conversely, the DG solution offers greater detail but is not bound-preserving.  Figs. \ref{figT_1:1c} and \ref{figT_1:1f} visualize the spatial distribution of DG cell average overshoots and undershoots. Mesh refinement alone does not eliminate violations in the cell average. %It should be noted that the adaptive technique merely enforces $0\le \overline{u} \le 2$, so the approximations may not be free of spurious oscillations.
\begin{figure}[H] % use [htbp] if you want more placement flexibility
  \centering 
  % ---- Row 1 ----
  \begin{subfigure}{0.32\textwidth}
    \includegraphics[width=\linewidth]{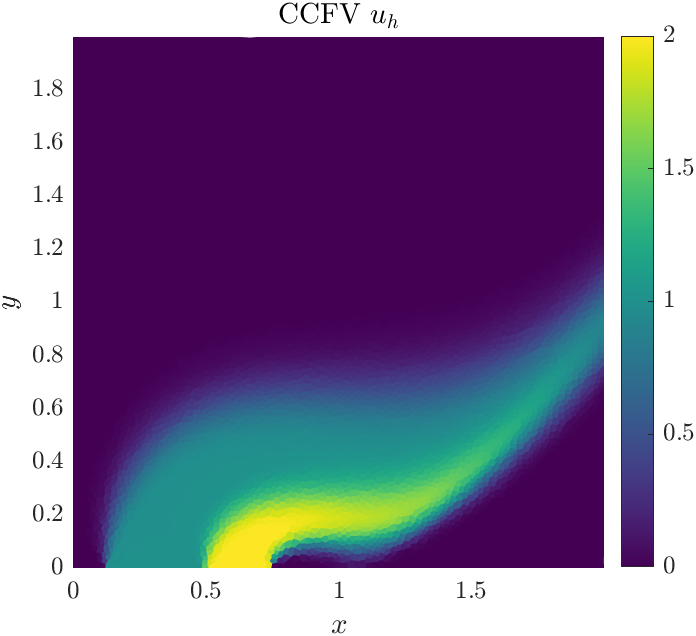}
    \subcaption{CCFV 10k elements}\label{fiTB_1:1a}
  \end{subfigure}\hfill
  \begin{subfigure}{0.32\textwidth}
    \includegraphics[width=\linewidth]{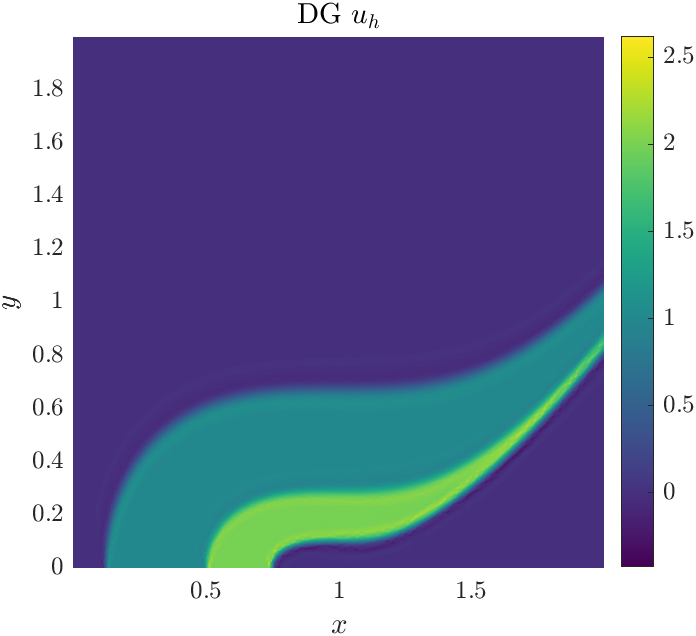}
    \subcaption{DG 10k elements}\label{figT_1:1b}
  \end{subfigure}\hfill
  \begin{subfigure}{0.34\textwidth}
    \includegraphics[width=\linewidth]{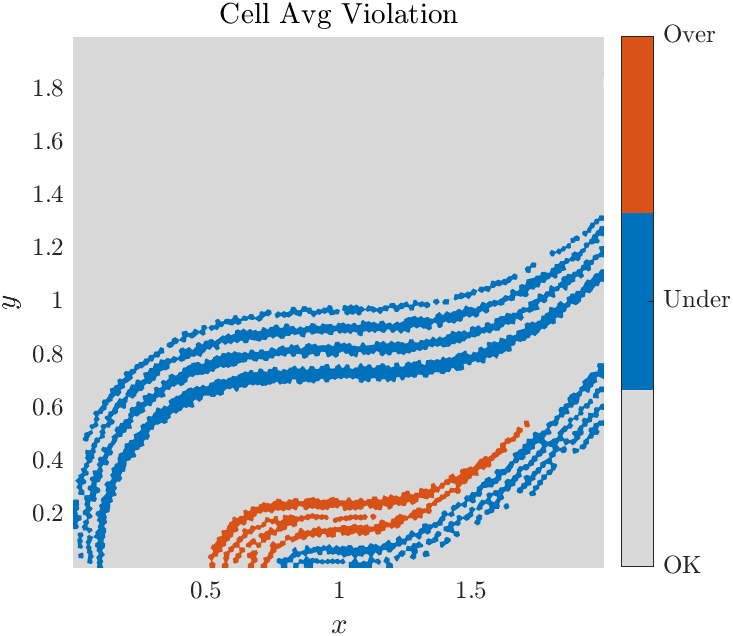}
    \subcaption{DG 10k elements (over/undershoot)}\label{figT_1:1c}
  \end{subfigure}
  \medskip % vertical gap between the two rows
  % ---- Row 2 ----
  \begin{subfigure}{0.32\textwidth}
    \includegraphics[width=\linewidth]{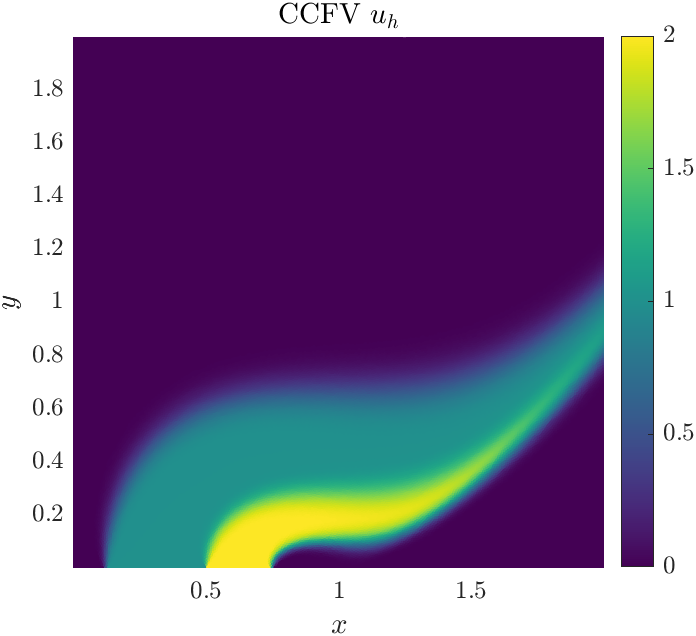}
    \subcaption{CCFV 65k elements}\label{figT_1:1d}
  \end{subfigure}\hfill
  \begin{subfigure}{0.32\textwidth}
    \includegraphics[width=\linewidth]{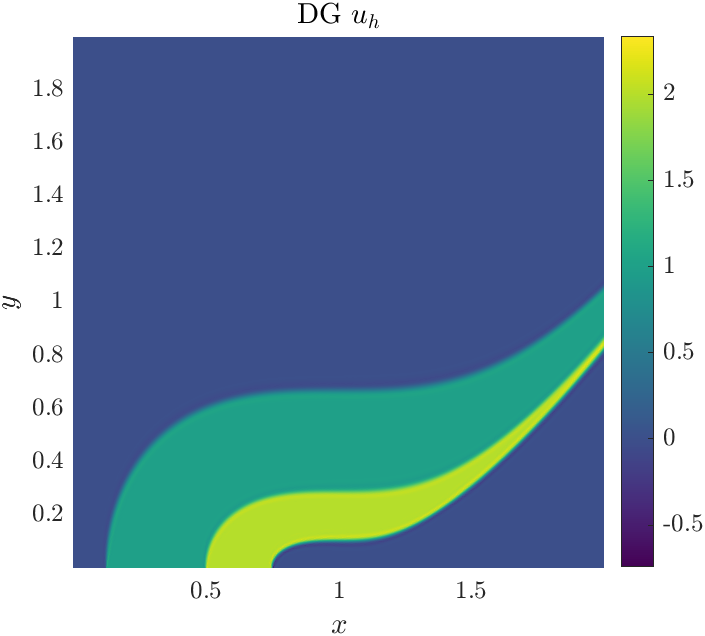}
    \subcaption{DG 65k elements}\label{figT_1:1e}
  \end{subfigure}\hfill
  \begin{subfigure}{0.34\textwidth}
    \includegraphics[width=\linewidth]{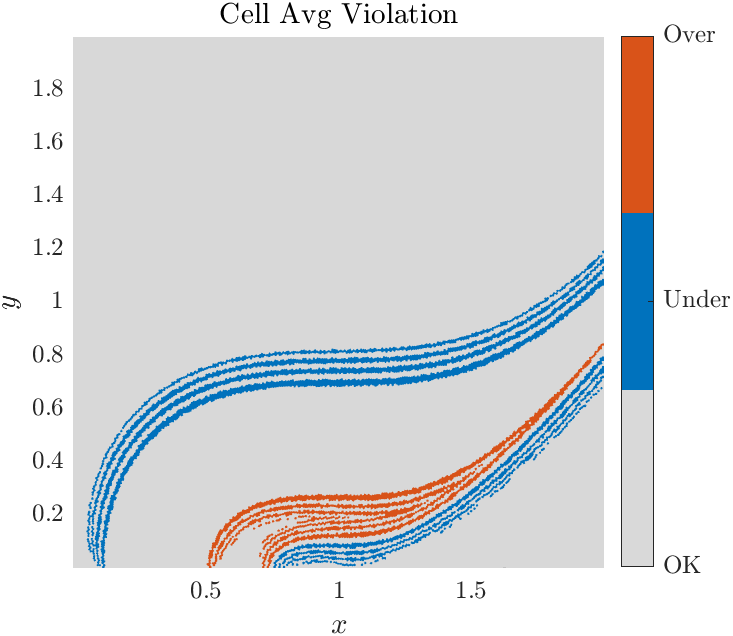}
    \subcaption{DG 65k elements (over/undershoot)}\label{figT_1:1f}
  \end{subfigure} 
  \caption{Baseline tests from Section~\ref{sec:Triple}, full CCFV ($\mathcal{T}_h^{DG}=\emptyset$) or full DG ($\mathcal{T}_h^{FV}=\emptyset$) with various mesh sizes.}
  \label{figT_1:sixgrid}
\end{figure}   
We next examine the FV-DG adaptive results in Fig.~\ref{figT_2:sixgrid}. The FV-DG solution is markedly better than CCFV and improves further with mesh refinement (Figs.~\ref{figT_2:2a}--\ref{figT_2:2c}). In addition, the FV-DG solution bound-preserving up to the specified tolerance.  Compared across refinements, the DG region becomes progressively larger (Figs.~\ref{figT_2:2d}--\ref{figT_2:2f}).
\begin{figure}[H]
  \centering

  % ---- Row 1 ----
  \begin{subfigure}[t]{0.32\textwidth}
    \includegraphics[width=\linewidth,height=0.24\textheight,keepaspectratio]{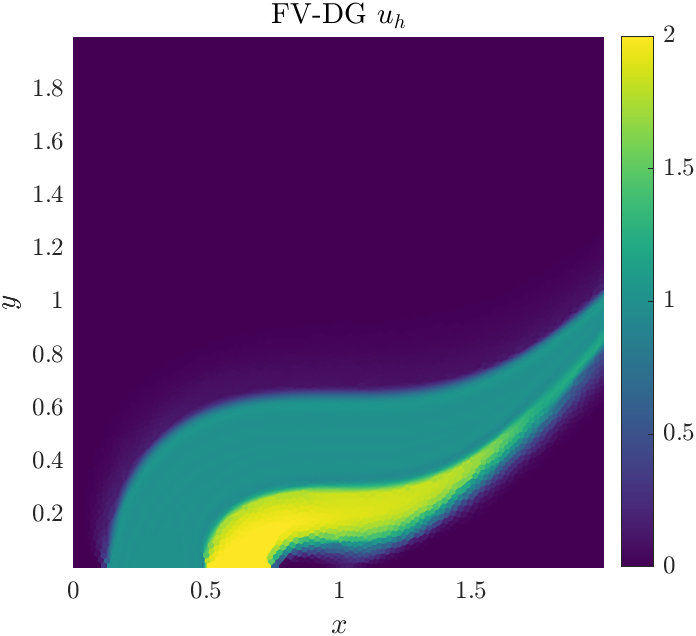}
    \subcaption{FV-DG 10k elements}\label{figT_2:2a}
  \end{subfigure}\hfill
  \begin{subfigure}[t]{0.32\textwidth}
    \includegraphics[width=\linewidth,height=0.24\textheight,keepaspectratio]{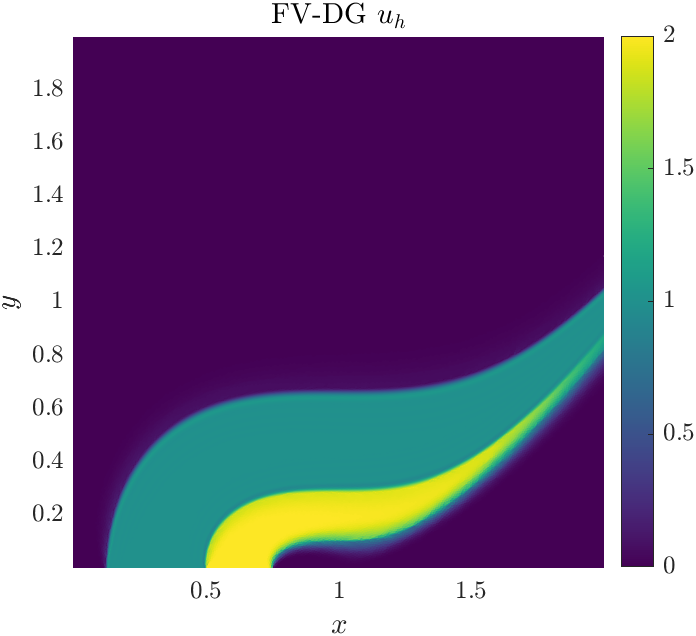}
    \subcaption{FV-DG 65k elements}\label{figT_2:2b}
  \end{subfigure}\hfill
  \begin{subfigure}[t]{0.32\textwidth}
    \includegraphics[width=\linewidth,height=0.24\textheight,keepaspectratio]{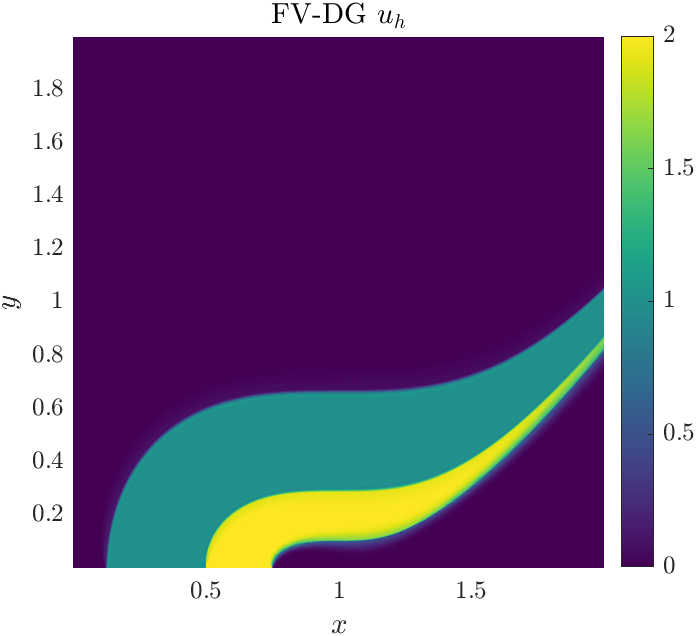}
    \subcaption{FV-DG 250k elements}\label{figT_2:2c}
  \end{subfigure}

  \medskip

  % ---- Row 2 ----%triple_FV_DG1_1_62500_1e_m7
  \begin{subfigure}[t]{0.32\textwidth}
    \includegraphics[width=\linewidth,height=0.24\textheight,keepaspectratio]{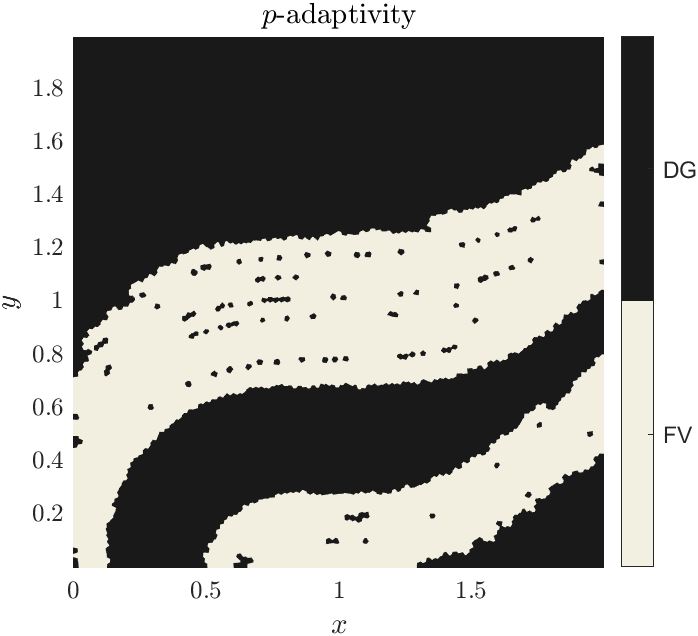}
    \subcaption{FV-DG 10k elements  (62\% DG, 38\% FV)}\label{figT_2:2d}
  \end{subfigure}\hfill
  \begin{subfigure}[t]{0.32\textwidth}
    \includegraphics[width=\linewidth,height=0.24\textheight,keepaspectratio]{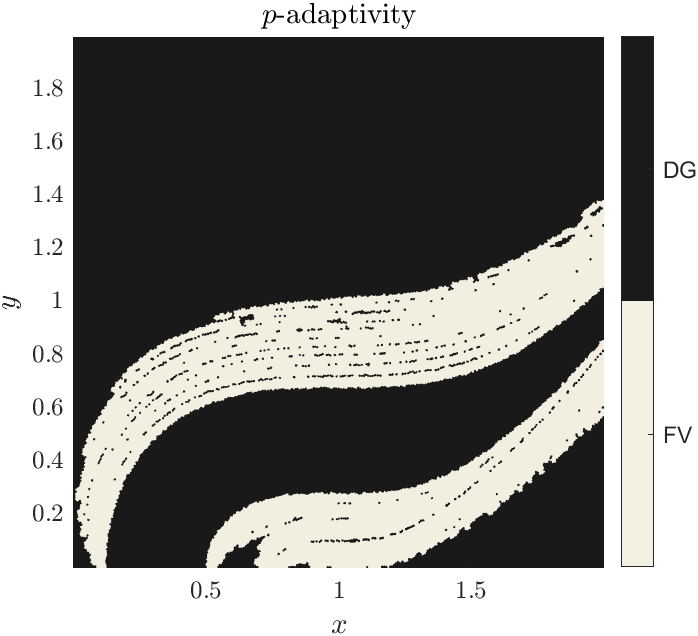}
    \subcaption{FV-DG 65k elements (75\% DG, 25\% FV)}\label{figT_2:2e}
  \end{subfigure}\hfill
  \begin{subfigure}[t]{0.32\textwidth}
    \includegraphics[width=\linewidth,height=0.24\textheight,keepaspectratio]{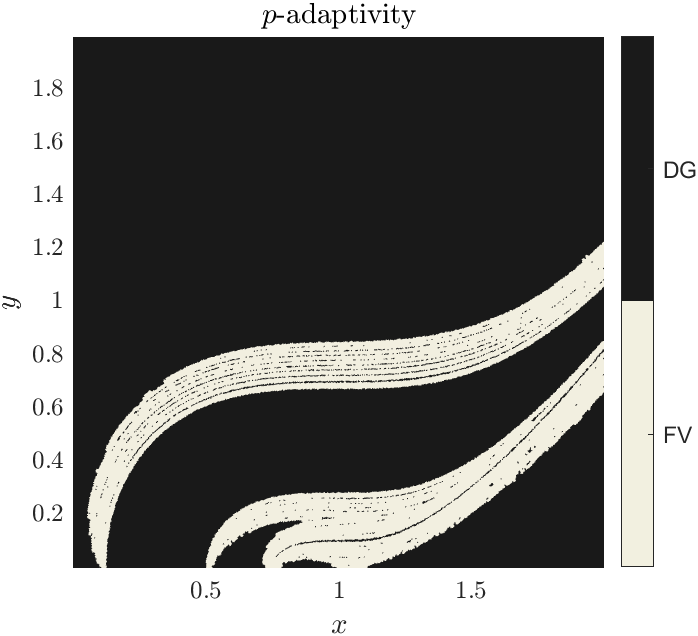}
    \subcaption{FV-DG 250k elements (85\% DG, 15\% FV)}\label{figT_2:2f}
  \end{subfigure}

  \caption{Section~\ref{sec:Triple} adaptive FV-DG with $\delta=10^{-7}$ and various mesh sizes.}
  \label{figT_2:sixgrid}
\end{figure}
\subsection{Sub-extremal oscillations} \label{secT:oscillations}
	In practice, the adaptive procedure functioned robustly across all test cases.  We note that the adaptive technique specifically enforces bound-preservation, so it is possible the the approximations have spurious oscillations within the bounds.  For this benchmark problem, the exact solution to~\eqref{model_problem1} obeys $0\le u\le 2$. When the unmodified approximation (Figs.~\ref{figT_1:1b} and~\ref{figT_1:1e}) yields \(\overline{u}_h \approx 1\), the indicator based on cell-average bounds will not flag that region (no overshoot/undershoot). To better explain this, Fig.~\ref{figT_profile:two-by-three} shows solution profiles along $x=1$. The FV-DG scheme is bound-preserving and resolves the profile more sharply than CCFV, though small local oscillations remain (e.g., near $x\approx0.3$).  
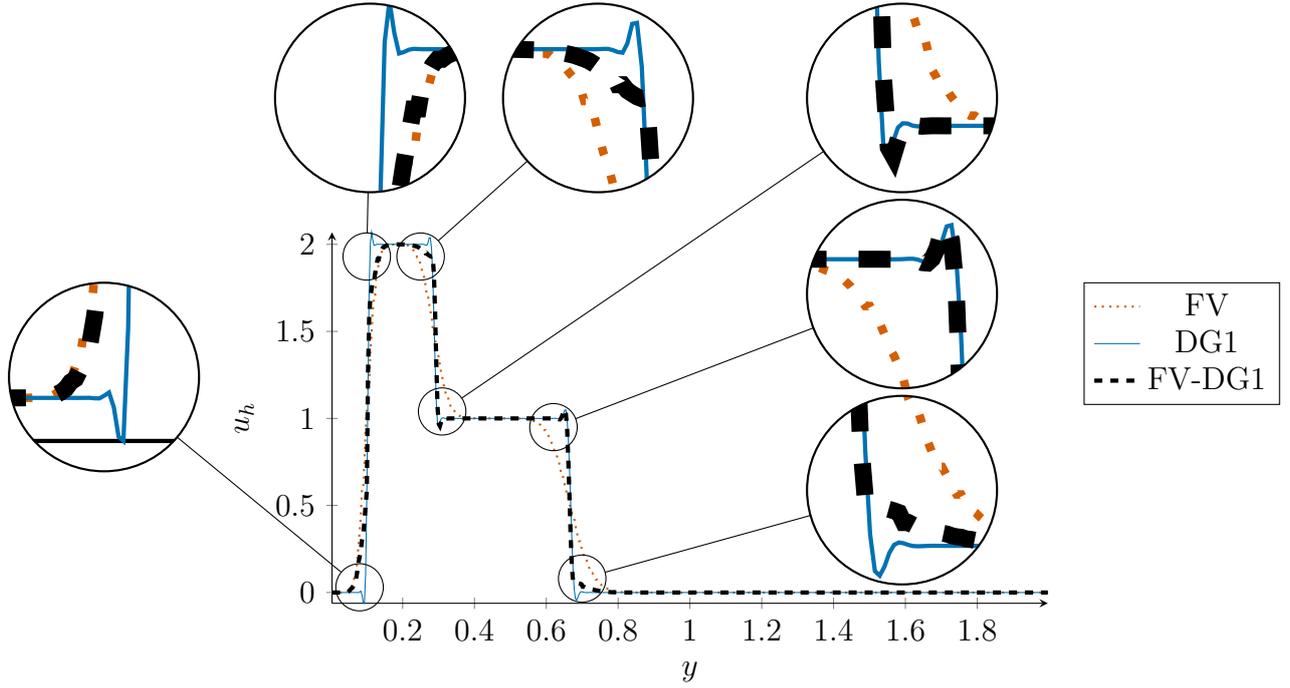
\begin{figure}[H] %triple_FV_262k_profile.dat
  \centering % Row 1
\begin{tikzpicture}[spy using outlines={
circle,
size=2.5cm,
magnification=4,
connect spies,
every spy on node/.style={fill=none, draw}
}]
\begin{axis}[
name=main,
legend style={at={(1.05,0.7)}, anchor=west},
width=11cm, height=6.5cm,
axis lines=left,
xlabel={$y$}, ylabel={$u_h$},
clip=false,
]
    \addplot+[thick, mark=none,dotted, OIverm]
      table[
        col sep=space,
        header=false,
        x index=0,  % first column = P(:,2) (your y)
        y index=1   % second column = uline
      ] {triple_FV_262k_profile.dat};
      
    \addplot+[thin, mark=none, OIblue]
      table[
        col sep=space,
        header=false,
        x index=0,  % first column = P(:,2) (your y)
        y index=1   % second column = uline
      ] {triple_DG_262k_profile.dat};     
       
    \addplot+[ultra thick, mark=none, dashed,OIblack]%+[only marks, mark=o, mark size=2pt, mark options={fill=none}, each nth point={10}]%+[thick, mark=none]
      table[
        col sep=space,
        header=false,
        x index=0,  % first column = P(:,2) (your y)
        y index=1   % second column = uline
      ] {triple_FV_DG_262k_profile.dat}; 

% define a coordinate in axis coordinates; it becomes a normal TikZ coordinate
\coordinate (spypointA) at (axis cs:0.70,0.08);
\coordinate (spypointB) at (axis cs:0.62,0.95);
\coordinate (spypointC) at (axis cs:0.31,1.04);
\coordinate (spypointD) at (axis cs:0.25,1.93);
\coordinate (spypointE) at (axis cs:0.10,1.93);
\coordinate (spypointF) at (axis cs:0.08,0.03);
   \addlegendentry{FV}
    \addlegendentry{DG1}
    \addlegendentry{FV-DG1}
\end{axis}

% place lens relative to the axis box (outside axis)
\spy on (spypointA) in node at (7.5,1.5);%([xshift=8mm]main.east);
\spy on (spypointB) in node at (7.5,4.1);
\spy on (spypointC) in node at (7.5,6.7);
\spy on (spypointD) in node at (3.5,6.7);
\spy on (spypointE) in node at (0.5,6.7);
\spy on (spypointF) in node at (-3.0,3.0);
\end{tikzpicture}
\caption{Profile along $x = 1$ for the triple layer problem, with magnifications. Fixed mesh with 262k elements.}
  \label{figT_profile:two-by-three}
\end{figure}  

Figure~\ref{figT:higher-order} compares the adaptive FV-DG scheme for varying polynomial degree $k$. The resulting FV and DG partitions are largely consistent across $k$, with only small differences.   
\begin{figure}[ht]
  \centering
  % Row 1
  \begin{subfigure}{0.23\textwidth}\centering
    \includegraphics[width=\linewidth]{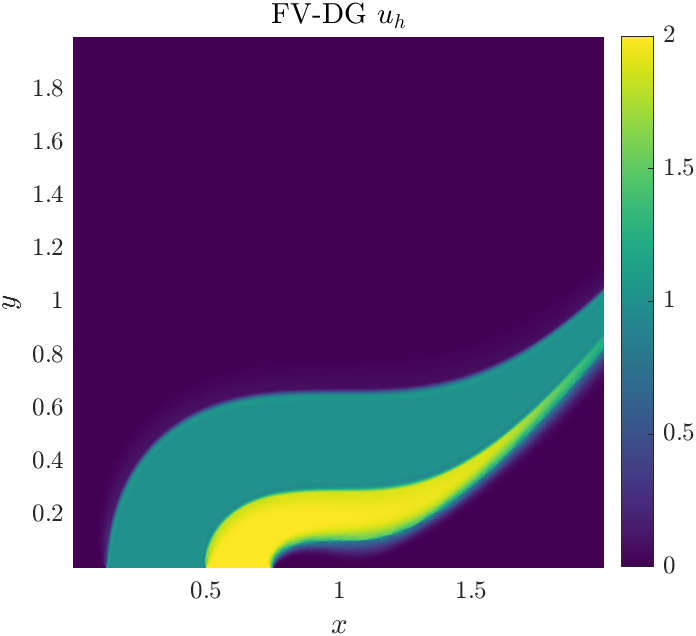}
    \caption{FV-DG1}
    %\label{fig:a}
  \end{subfigure}\hfill
  \begin{subfigure}{0.23\textwidth}\centering
    \includegraphics[width=\linewidth]{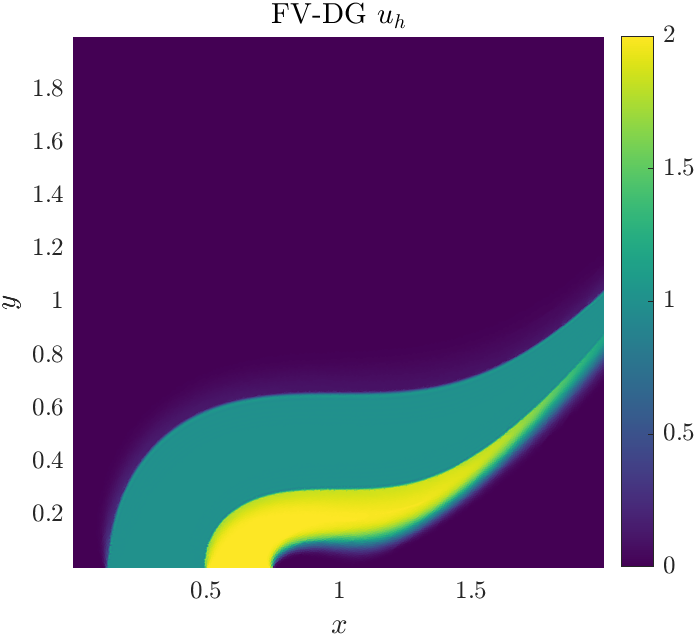}
    \caption{FV-DG2}
    %\label{fig:b}
  \end{subfigure}\hfill
  \begin{subfigure}{0.23\textwidth}\centering
    \includegraphics[width=\linewidth]{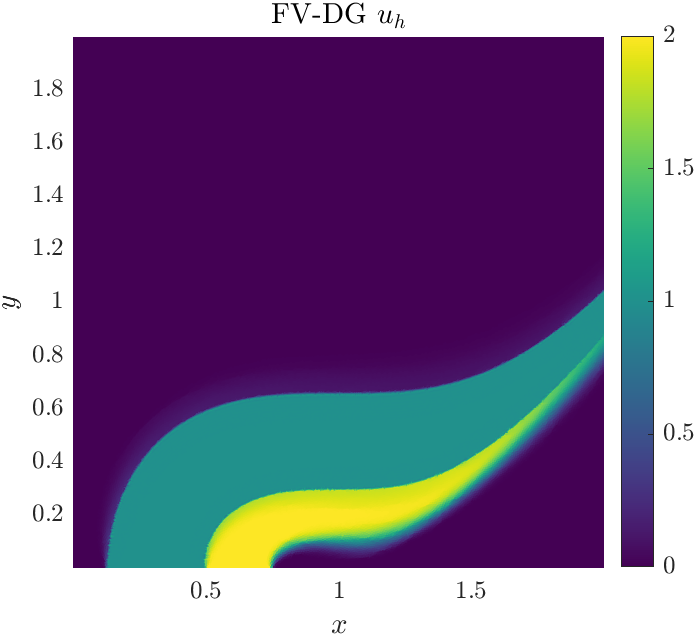}
    \caption{FV-DG3}
    %\label{fig:c}
  \end{subfigure}\hfill
  \begin{subfigure}{0.23\textwidth}\centering
    \includegraphics[width=\linewidth]{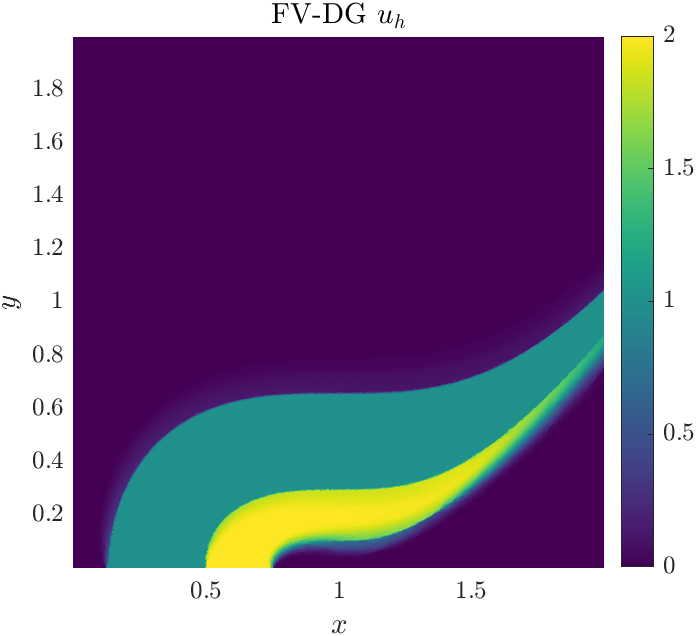}
    \caption{FV-DG4}
   % \label{fig:d}
  \end{subfigure}

  \medskip
  % Row 2
  \begin{subfigure}{0.23\textwidth}\centering
    \includegraphics[width=\linewidth]{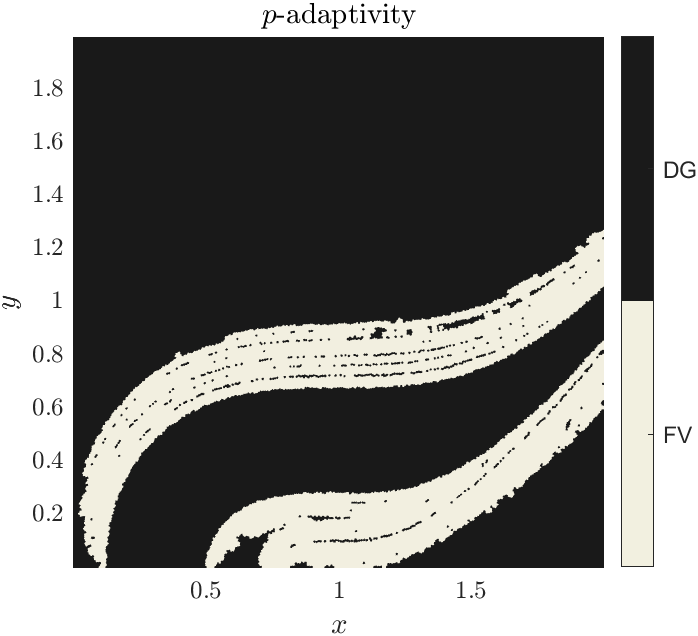}
    \caption{FV-DG1}
    %\label{fig:e}
  \end{subfigure}\hfill
  \begin{subfigure}{0.23\textwidth}\centering
    \includegraphics[width=\linewidth]{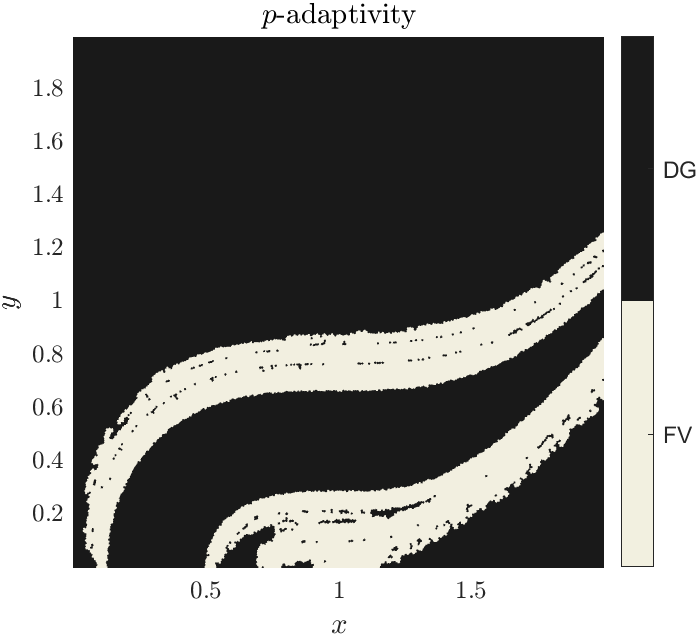}
    \caption{FV-DG2}
   % \label{fig:f}
  \end{subfigure}\hfill
  \begin{subfigure}{0.23\textwidth}\centering
    \includegraphics[width=\linewidth]{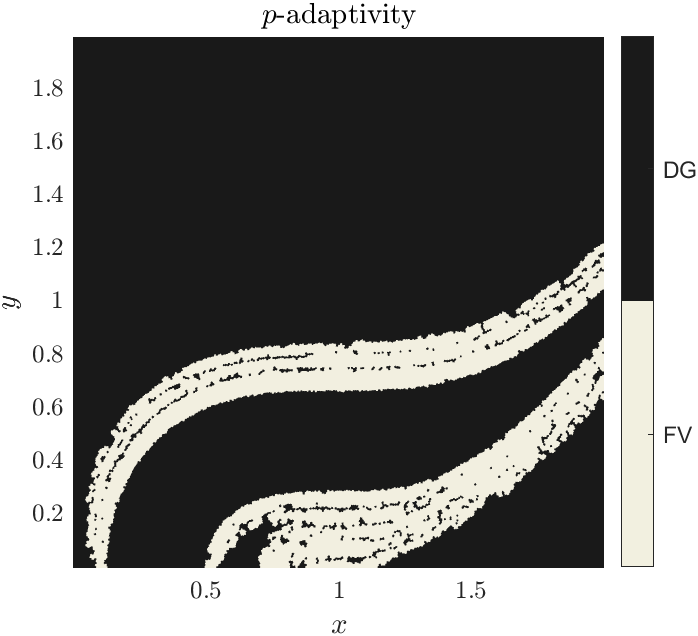}
    \caption{FV-DG3}
    %\label{fig:g}
  \end{subfigure}\hfill
  \begin{subfigure}{0.23\textwidth}\centering
    \includegraphics[width=\linewidth]{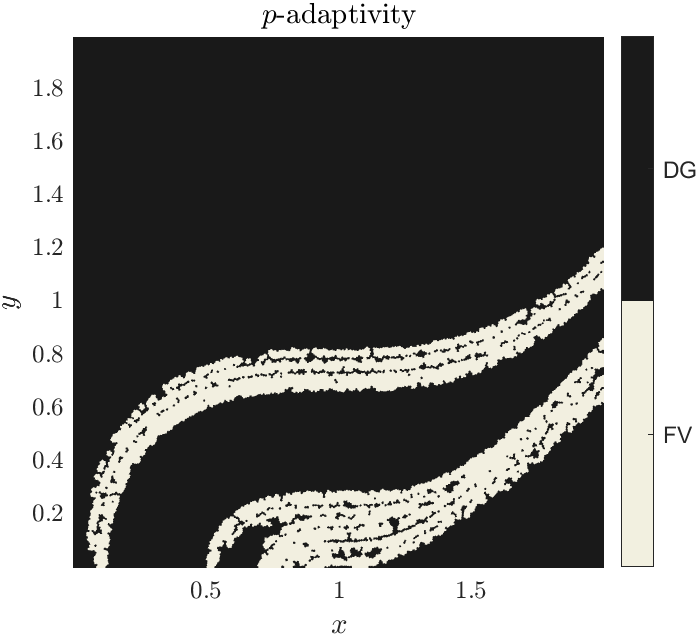}
    \caption{FV-DG4}
   % \label{fig:h}
  \end{subfigure} 
  \caption{Comparison of the adaptive FV-DG scheme for different polynomial degrees.}
  \label{figT:higher-order}
\end{figure}

\clearpage
\section{Convection-dominated problem (L-shaped domain)} \label{num_L} %We test two different cases, $K=10^{-3}$ and $K=10^{-6}$.   
 	The next benchmark was taken from~\cite{FORMAGGIA2004511}.  No exact solution can be easily written down for this problem.  The domain is L-shaped with $\Omega = [0,4]^2 \backslash [0,2]^2.$  The convective vector field is given as $\vec{\beta}=[y,-x]^T$ and $K=10^{-6}$.  The PDE boundary conditions are as follows:
\begin{align} 
u &=1 ,&&\text{on } \{(x,y): x=0,~y\in(2,4)\}, \notag
\\
K \nabla u \cdot \textbf{n} &=0 ,&&\text{on }  \{(x,y): x=4,~y\in(0,4)\} \notag
\cup
  \{(x,y): x\in (2,4),~y=0\} \notag
,
\\
u &=0 ,&&\text{elsewhere}.  \notag
\end{align} 
The solution of this problem exhibits two boundary layers along  $\{(x,y): x\in(0,4),~y=4\}$ and $\{(x,y): x\in(0,2),~y=2\}$ . Moreover, the solution obeys $0\le u \le 1$ and contains two circular-shaped interior layers \cite{Carpio}.
 
	A series of tests are conducted.  We have some baseline experiments, in particular, high-order DG on all mesh elements, and CCFV on all mesh elements. Fig.~\ref{figL_1:sixgrid} contrasts the CCFV and DG schemes. As seen in panels~\ref{figL_1:1a} and~\ref{figL_1:1d}, the CCFV solution is smeared, indicating stronger numerical diffusion. Conversely, DG is less diffusive and captures the sharp features (Figs.~\ref{figL_1:1b}, \ref{figL_1:1e}). However, the DG cell averages exhibit overshoots/undershoots, as evidenced in Figs.~\ref{figL_1:1c} and \ref{figL_1:1f}. Further 
$h$-refinement yields only marginal improvement.
\begin{figure}[H] % use [htbp] if you want more placement flexibility
  \centering 
  % ---- Row 1 ----
  \begin{subfigure}{0.319\textwidth}
    \includegraphics[width=\linewidth]{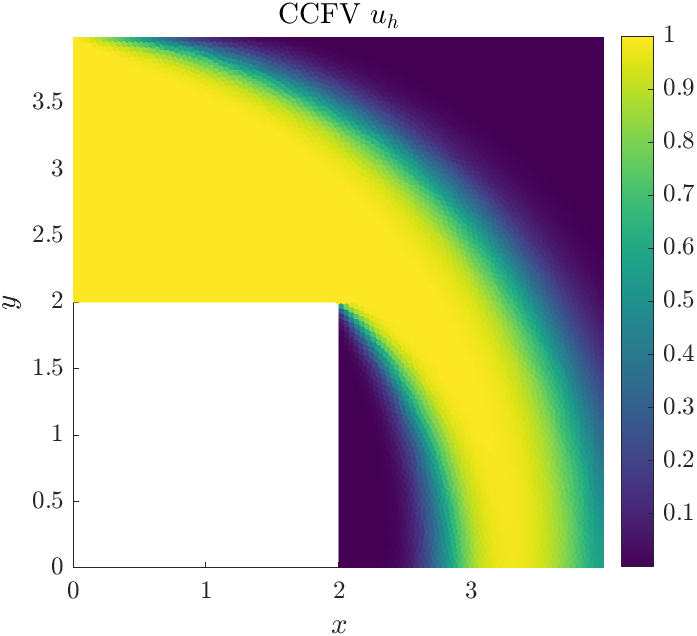}
    \subcaption{CCFV 8k elements}\label{figL_1:1a}
  \end{subfigure}\hfill
  \begin{subfigure}{0.33\textwidth}
    \includegraphics[width=\linewidth]{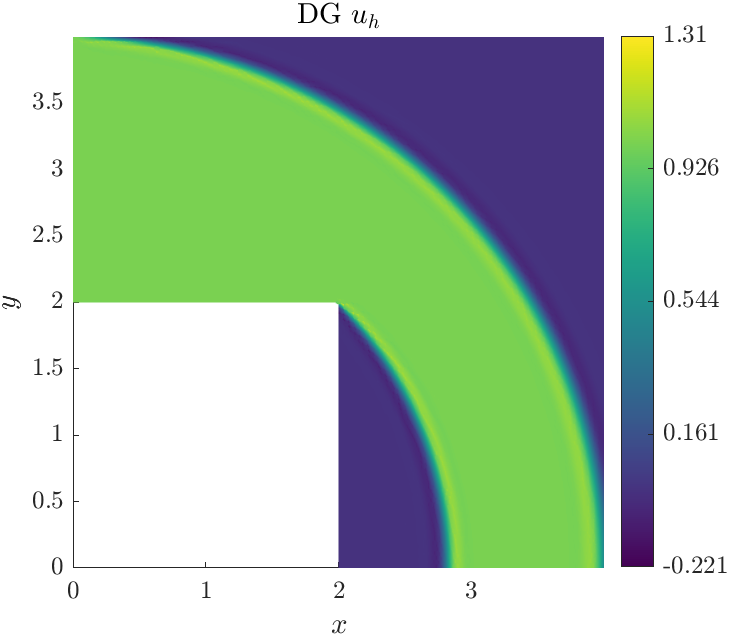}
    \subcaption{DG 8k elements}\label{figL_1:1b}
  \end{subfigure}\hfill
  \begin{subfigure}{0.33\textwidth}
    \includegraphics[width=\linewidth]{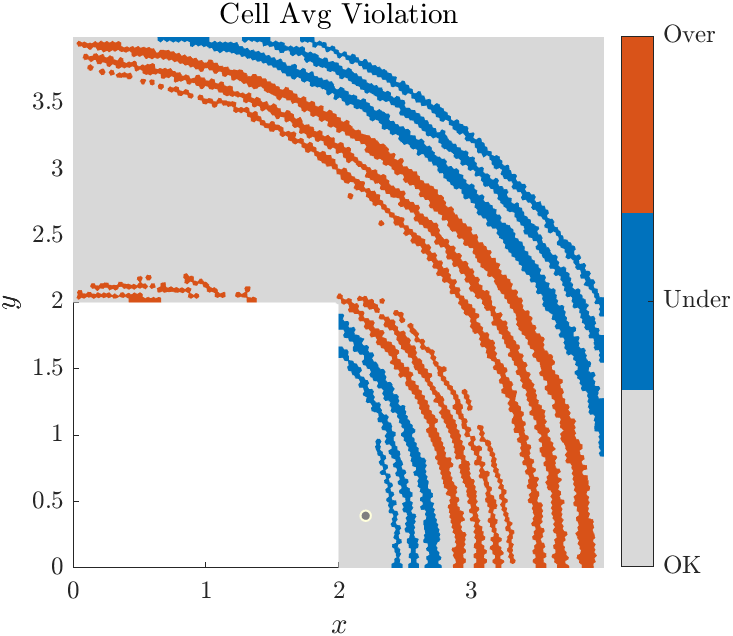}
    \subcaption{DG 8k elements (over/undershoot)}\label{figL_1:1c}
  \end{subfigure}
  \medskip % vertical gap between the two rows
  % ---- Row 2 ----
  \begin{subfigure}{0.319\textwidth}
    \includegraphics[width=\linewidth]{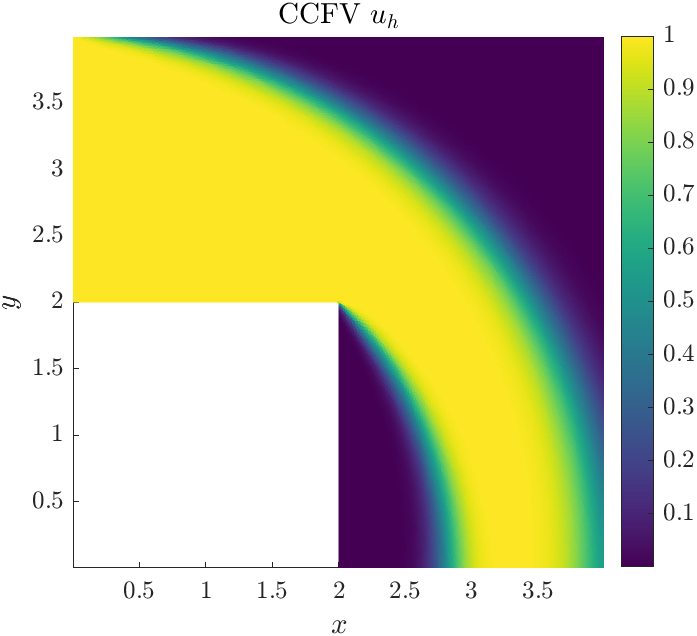}
    \subcaption{CCFV 32k elements}\label{figL_1:1d}
  \end{subfigure}\hfill
  \begin{subfigure}{0.33\textwidth}
    \includegraphics[width=\linewidth]{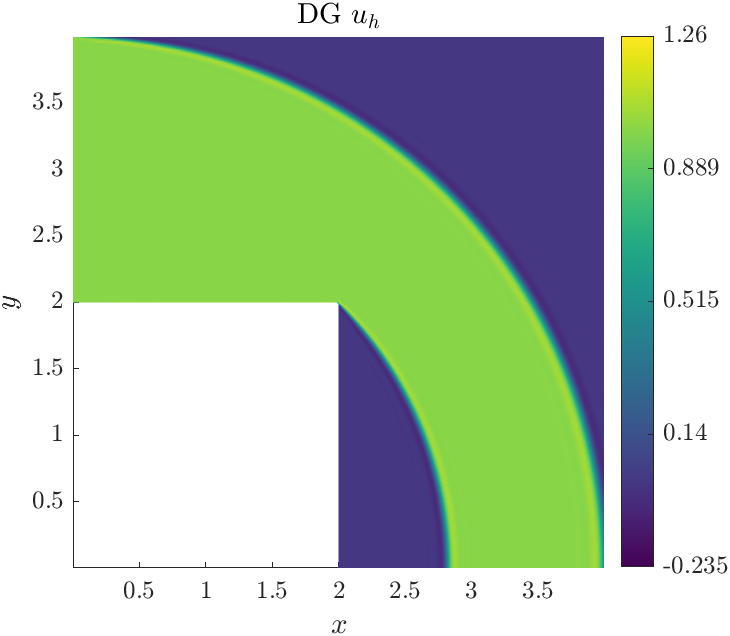}
    \subcaption{DG 32k elements}\label{figL_1:1e}
  \end{subfigure}\hfill
  \begin{subfigure}{0.33\textwidth}
    \includegraphics[width=\linewidth]{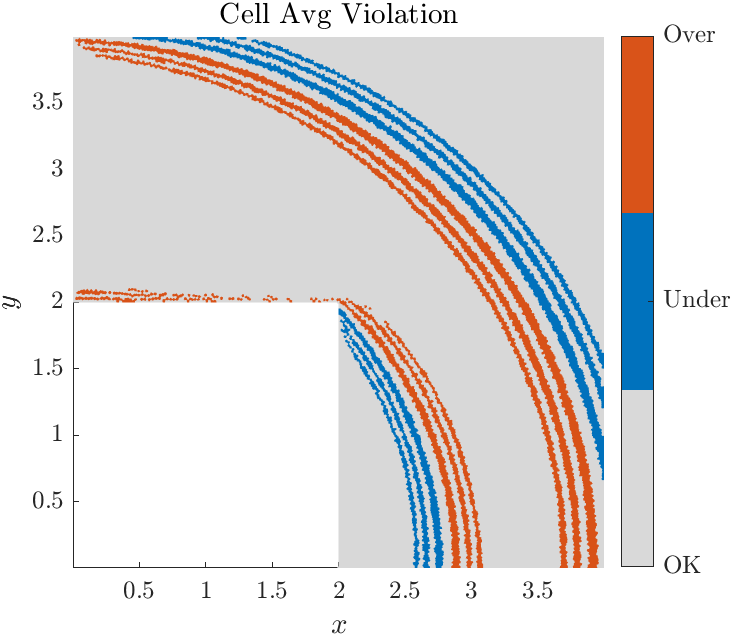}
    \subcaption{DG 32k elements (over/undershoot)}\label{figL_1:1f}
  \end{subfigure} 
  \caption{Baseline tests - use CCFV everywhere or DG everywhere with various mesh sizes.}
  \label{figL_1:sixgrid}
\end{figure}
%Next, we compare the FV-DG coupled scheme, with adaptivity. From Figs.~\ref{figL_2:2a},~\ref{figL_2:2b}, and~\ref{figL_2:2c}, we can see that the FV-DG solution is considerably less diffusive than that of the CCFV solution.  The second row of Fig.~\ref{figL_2:sixgrid} displays the FV and DG regions.  Interestingly enough, for this benchmark, as the mesh is refined, the DG region occupies more of the computational domain.
	\subsection{FV–DG adapted scheme with relaxed tolerance}
	Next, we compare the adaptive FV–DG coupled scheme with $\delta=10^{-6}$. Figs.~\ref{figL_2:2a}, \ref{figL_2:2b}, and \ref{figL_2:2c} show that the FV–DG solution is considerably less diffusive than the CCFV solution. The second row of Fig.~\ref{figL_2:sixgrid} delineates the FV and DG regions. For this benchmark, the DG region expands with mesh refinement, occupying an increasing fraction of the domain. 
\begin{figure}[H]
  \centering
  % ---- Row 1 ----
  \begin{subfigure}[t]{0.32\textwidth}
    \includegraphics[width=\linewidth,height=0.24\textheight,keepaspectratio]{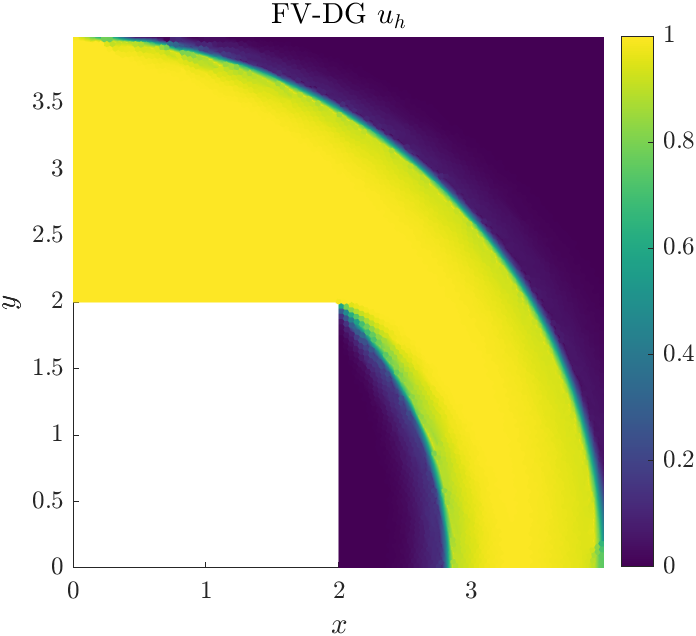}
    \subcaption{FV-DG 8k elements}\label{figL_2:2a}
  \end{subfigure}\hfill
  \begin{subfigure}[t]{0.32\textwidth}
    \includegraphics[width=\linewidth,height=0.24\textheight,keepaspectratio]{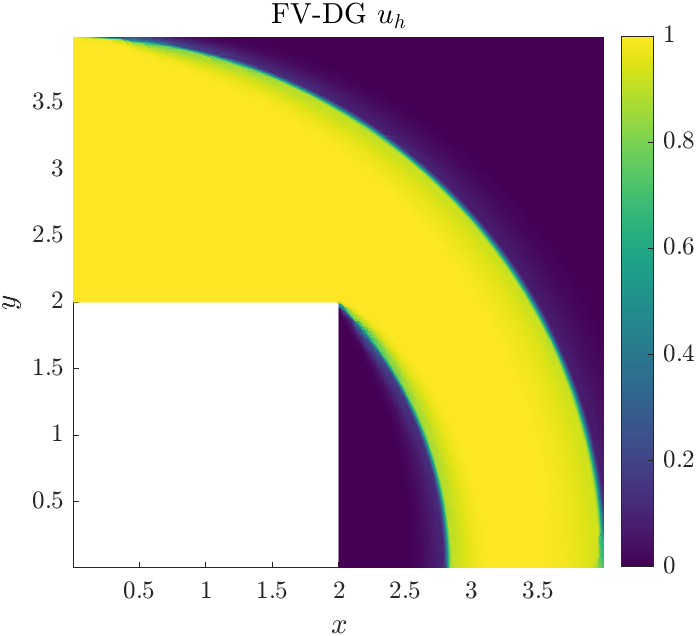}
    \subcaption{FV-DG 32k elements}\label{figL_2:2b}
  \end{subfigure}\hfill
  \begin{subfigure}[t]{0.32\textwidth}
    \includegraphics[width=\linewidth,height=0.24\textheight,keepaspectratio]{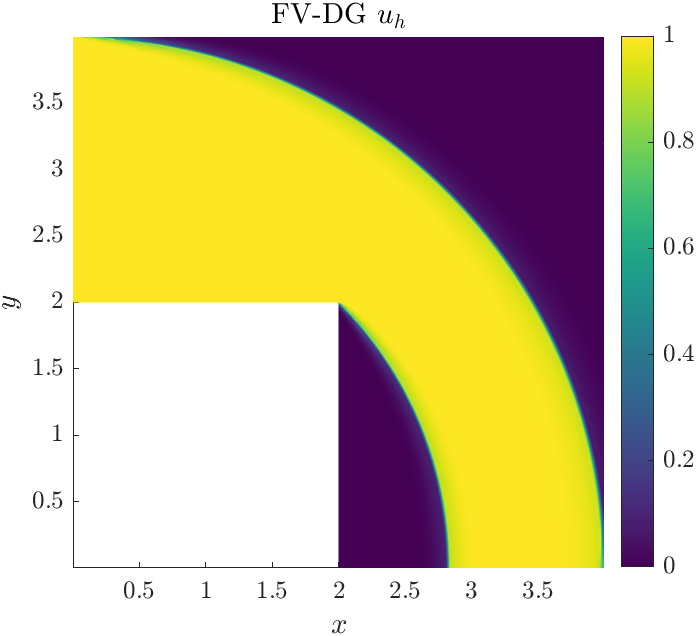}
    \subcaption{FV-DG 132k elements}\label{figL_2:2c}
  \end{subfigure}

  \medskip

  % ---- Row 2 ----
  \begin{subfigure}[t]{0.32\textwidth}
    \includegraphics[width=\linewidth,height=0.24\textheight,keepaspectratio]{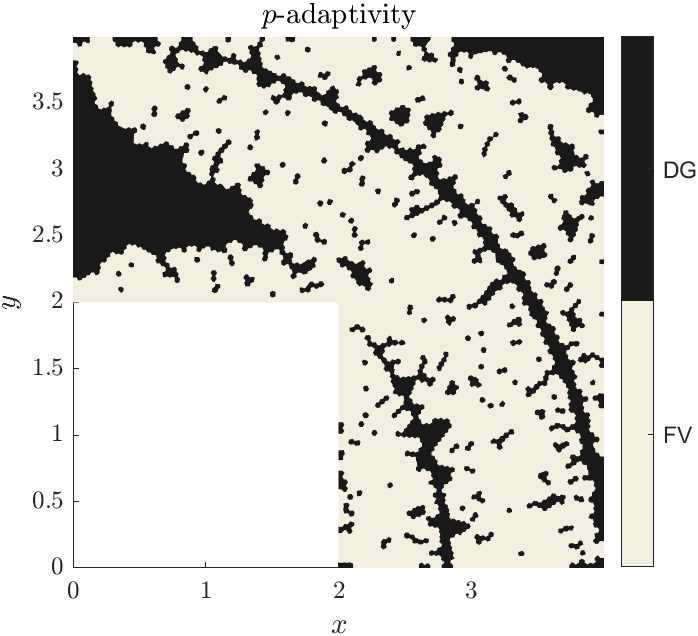}
    \subcaption{FV-DG 8k elements  (38\% DG, 62\% FV)}\label{figL_2:2d}
  \end{subfigure}\hfill
  \begin{subfigure}[t]{0.32\textwidth}
    \includegraphics[width=\linewidth,height=0.24\textheight,keepaspectratio]{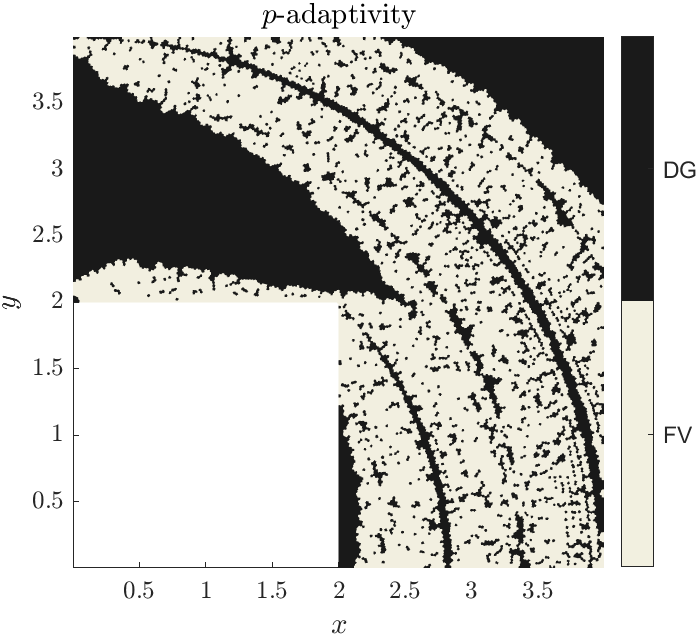}
    \subcaption{FV-DG 32k elements (42\% DG, 58\% FV)}\label{figL_2:2e}
  \end{subfigure}\hfill
  \begin{subfigure}[t]{0.32\textwidth}
    \includegraphics[width=\linewidth,height=0.24\textheight,keepaspectratio]{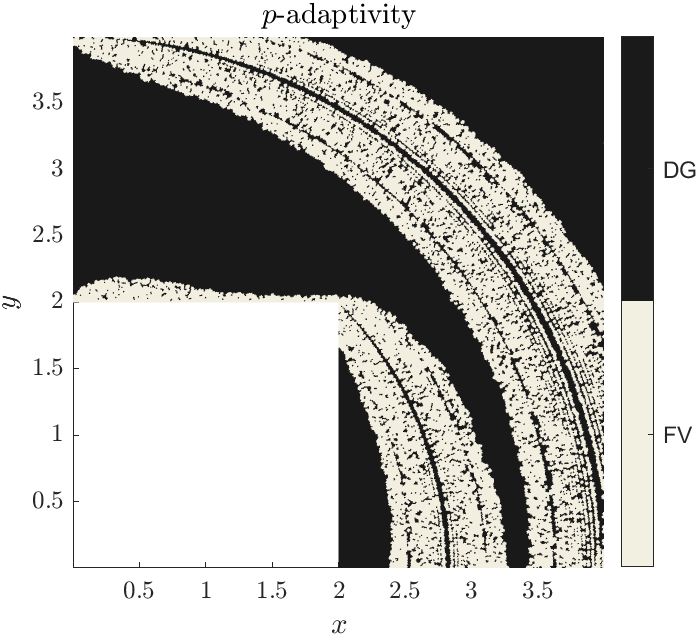}
    \subcaption{FV-DG 132k elements (61\% DG, 39\% FV)}\label{figL_2:2f}
  \end{subfigure}

  \caption{Adaptive scheme for FV--DG with $\delta=10^{-6}$ and various mesh sizes.}
  \label{figL_2:sixgrid}
\end{figure}
 \subsection{FV–DG adapted scheme with stricter tolerances}
	With $\delta=10^{-9}$ and $\delta=10^{-13}$, comparable results are obtained by the FV-DG coupled scheme (see Fig.~\ref{figL_3:twogrid}).  At identical mesh resolution, the technique provides a much less diffusive solution than CCFV (see the third row of Fig.~\ref{figL_3:twogrid}).  Specifically, Figs.~\ref{figL_3:1f} (132k elements) and~\ref{figL_3:1e} (532k elements) indicate that CCFV remains markedly more diffusive than the adaptive FV–DG method.
	
The adaptive strategy delivers good results for this challenging benchmark. The DG cell count is controlled by \(\delta\); typically, decreasing \(\delta\) reduces the DG region. For instance, Fig.~\ref{figL_3:1b} with $\delta=10^{-9}$ has a space composition of (61\% DG, 39\% FV), whereas Fig.~\ref{figL_3:1d} with $\delta=10^{-13}$ has a space composition of (53\% DG, 47\% FV). A stricter tolerance for $\delta $ increases the FV region.
\begin{figure}[H]
  \centering
  \begin{subfigure}{0.40\textwidth}
    \includegraphics[width=\linewidth]{L_FV_DG_100k_1.png}
    \subcaption{$\delta = 10^{-6}$, 132k  elements (61\% DG, 39\% FV)}
    \label{figL_3:1a}
  \end{subfigure}~%\hfill
  \begin{subfigure}{0.40\textwidth}
    \includegraphics[width=\linewidth]{L_FV_DG_100k_2.png}
    \subcaption{$\delta = 10^{-6}$, 132k elements (61\% DG, 39\% FV)}
    \label{figL_3:1b}
  \end{subfigure}    
 \medskip  
  \begin{subfigure}{0.40\textwidth}
    \includegraphics[width=\linewidth]{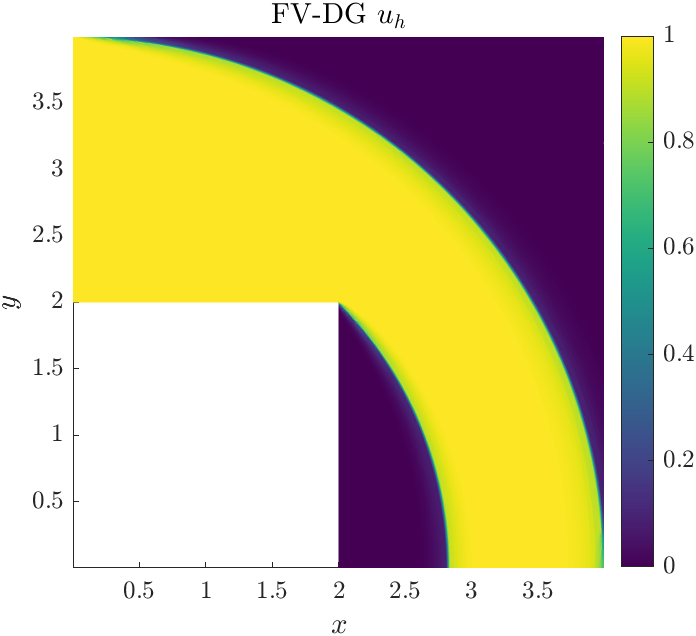}
    \subcaption{$\delta = 10^{-13}$, 132k elements (53\% DG, 47\% FV)}
    \label{figL_3:1c}
  \end{subfigure}~%\hfill
  \begin{subfigure}{0.40\textwidth}
    \includegraphics[width=\linewidth]{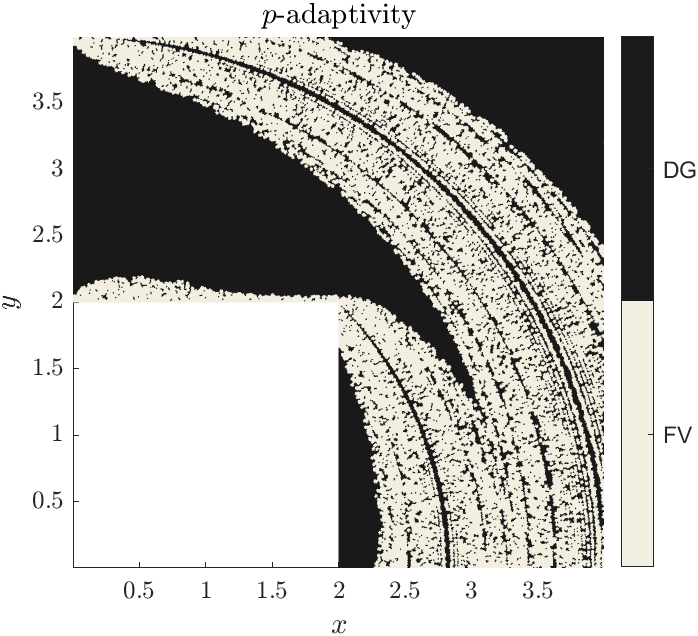}
    \subcaption{$\delta = 10^{-13}$, 132k elements (53\% DG, 47\% FV)}
    \label{figL_3:1d}
  \end{subfigure}    
\medskip
  \begin{subfigure}{0.40\textwidth}
    \includegraphics[width=\linewidth]{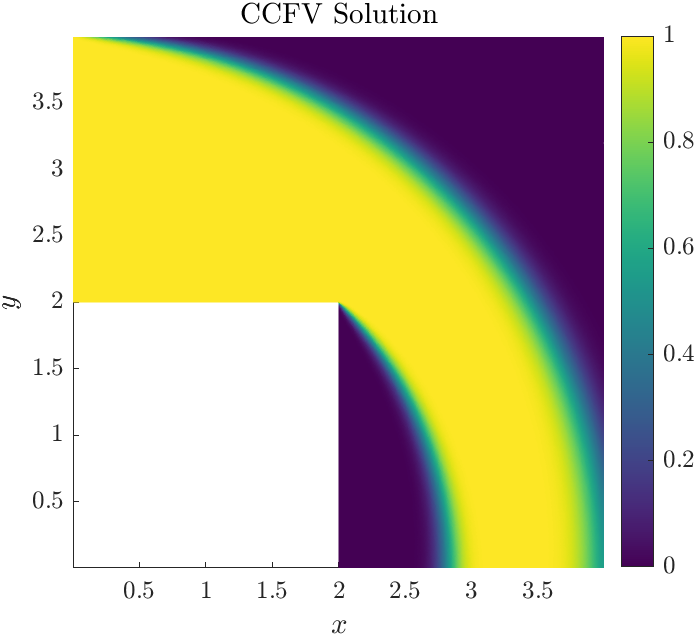}
    \subcaption{CCFV 132k elements}
    \label{figL_3:1e}
  \end{subfigure}~%\hfill
  \begin{subfigure}{0.40\textwidth}
    \includegraphics[width=\linewidth]{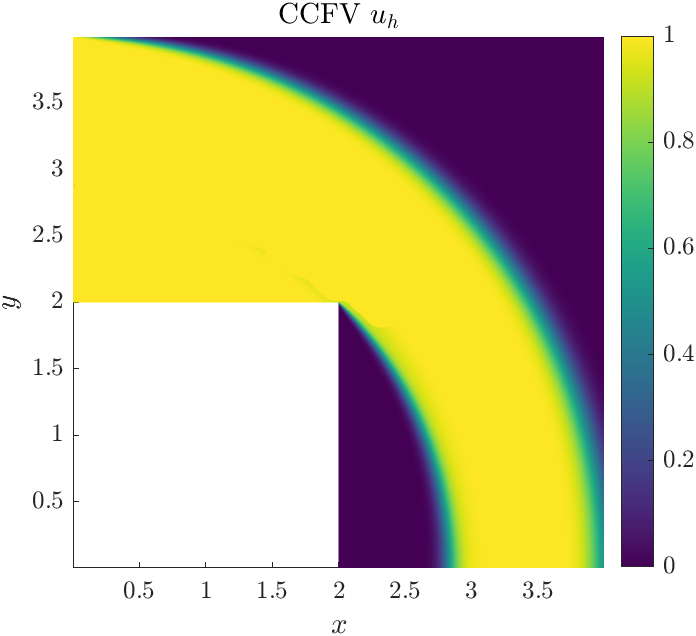}
    \subcaption{CCFV 532k elements}
    \label{figL_3:1f}
  \end{subfigure}   
  \caption{CCFV vs adaptive scheme for FV-DG with $\delta = 10^{-9}$ (row 1) and $\delta = 10^{-13}$ (row 2).   }
  \label{figL_3:twogrid}
\end{figure}
%	Next we illustrate that the adaptive technique described in Section~\ref{sec:adapt} can incorporate standard $h$-adaptivity. Here, we take $\delta=10^{-13}$. Fig.~\ref{fig4_3:adapt} (a) visualizes the initial mesh. Fig.~\ref{fig4_3:adapt} (b) visualizes the mesh after $h$-adapitivity.  Fig.~\ref{fig4_3:adapt} (c) Zooms in near $x=2$ on the refined mesh.  Fig.~\ref{fig4_3:adapt} (d) plots the FV-DG solution on the $h$-adapted mesh. Fig.~\ref{fig4_3:adapt} (e) plots the FV and DG regions for the $h$-adapted mesh.  Fig.~\ref{fig4_3:adapt} (f) shows where the full DG method has violations in the cell average. 
 \subsection{FV–DG adapted scheme with $h$-adaptivity}
To demonstrate that the adaptive technique of Section~\ref{sec:adapt} accommodates standard $h$-adaptivity, we set $\delta=10^{-13}$.  Due to the nature of the selection of DG and FV regions, we assume that the entire mesh consists of Voronoi cells. This ensures that the TPFA mesh assumption holds (see Assumption~\ref{ass:tpfa-admissible}). After marking elements for refinement, we randomly place Voronoi sites in a small neighborhood of the marked elements’ centroids. We then pass the new sites, together with the existing ones, to a Voronoi mesh generator~\cite{fabri2009cgal,talischi2012polymesher,engwirda2018generalised}.
 
Given $u_h$ on an initial mesh $\mathcal{T}_h$, we drive mesh refinement with a jump-based indicator of the solution across element faces \cite{Dolejsi2003}:
\[
\eta_E^2 \;:=\; \sum_{\gamma\subset\partial E}  \int_\gamma   \frac{[u_h]_\gamma^2}{h_\gamma}\,\mathrm{d}s,
\quad
\eta_h(\mathcal{T}_h) = \bigg(\sum_E \eta_E^2\bigg)^{1/2}
.
\]
 Fig.~\ref{fig4_3:adapt} shows: (a) the initial mesh; (b) the mesh after $h$-refinement; (c) a zoom near $x=2$; (d) the FV–DG solution on the adapted mesh; (e) the FV and DG subregions; and (f) the cells where the full DG method violates the cell-average constraint.  Thus, the adaptive technique in Section~\ref{sec:adapt} fits naturally within standard $hp$-adaptive finite element strategies.
\begin{figure}[H]  % needs \usepackage{float}
  \centering

  \subfloat[Initial mesh]{\includegraphics[width=0.32\textwidth]{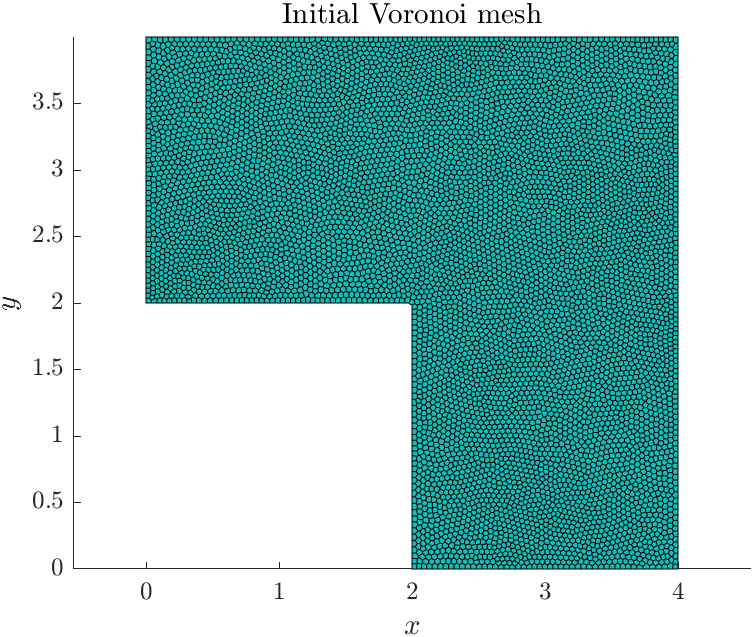}}\hfill
  \subfloat[$h$-adapted mesh]{\includegraphics[width=0.32\textwidth]{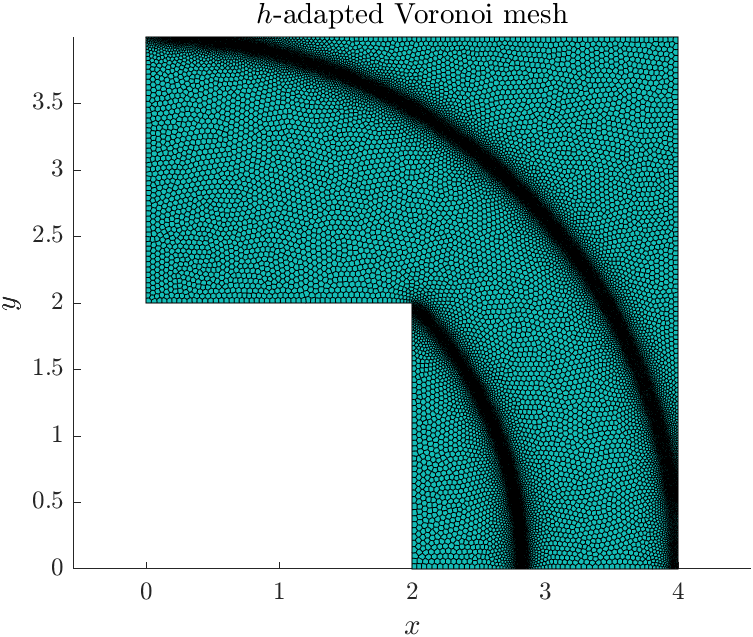}}\hfill
  \subfloat[Zoom-in of $h$-adapted mesh]{\includegraphics[width=0.32\textwidth]{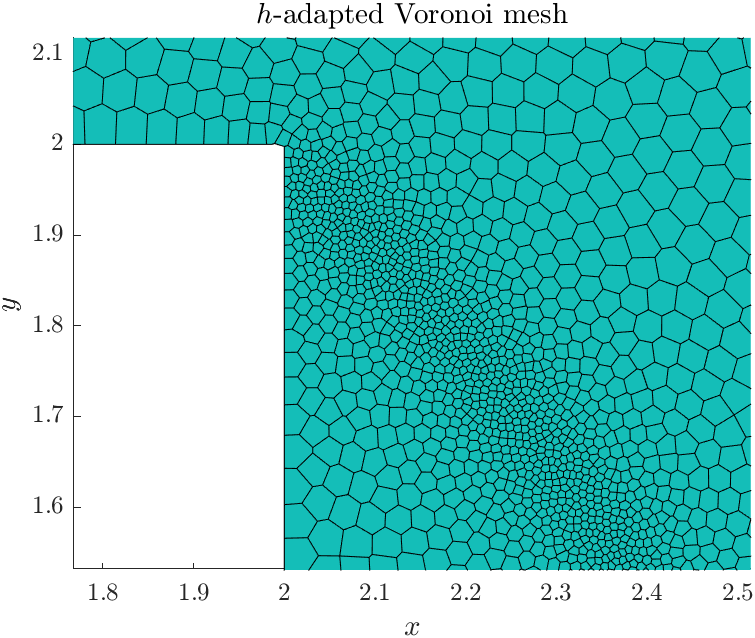}}\\[0.7em]

  \subfloat[Adapted FV-DG]{\includegraphics[width=0.32\textwidth]{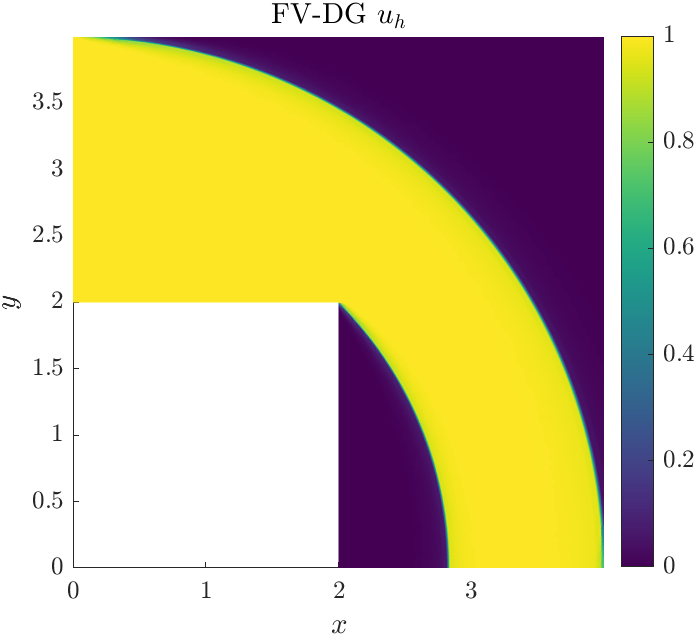}}\hfill
  \subfloat[FV/DG regions]{\includegraphics[width=0.32\textwidth]{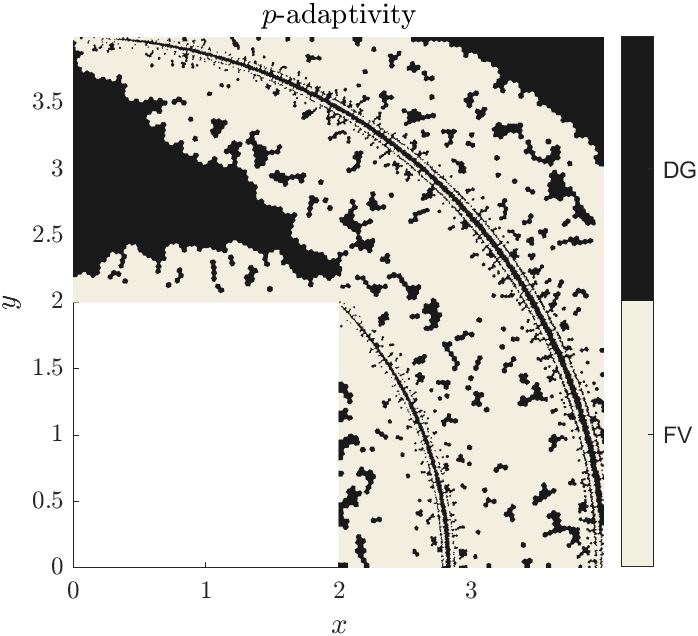}}\hfill
  \subfloat[Full DG violations]{\includegraphics[width=0.32\textwidth]{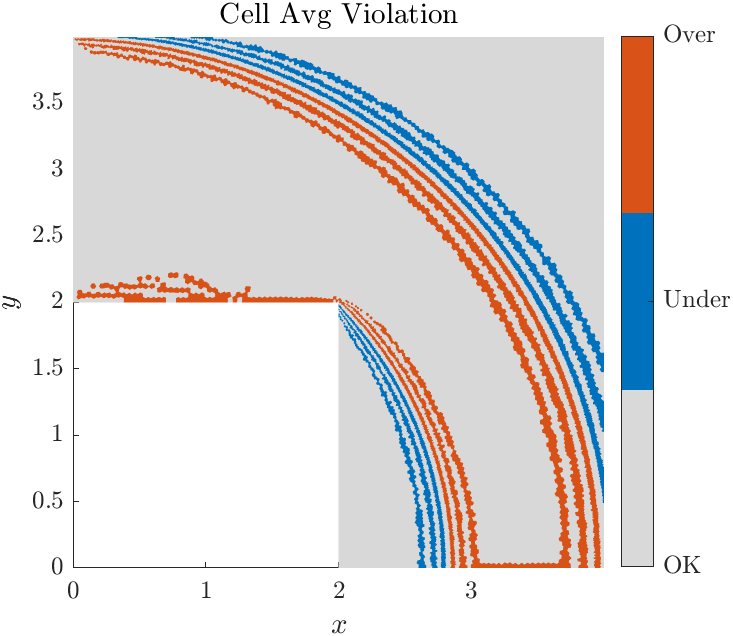}}

  \caption{Results for an $h$-adapted mesh.}
  \label{fig4_3:adapt}
\end{figure}

\clearpage
\section{Internal and boundary layers}\label{sec:badia}
	We adopt the test case from~\cite{BADIA2015107}. Because the problem is convection-dominated, the solution should carry the profile imposed at \(x=0\) downstream in the direction of \(\vec{\beta}\). The diffusion coefficient is set as  $K=10^{-4}$, and $\vec{\beta}=[\cos(\pi/3),-\sin(\pi/3)]^T$.     
 \begin{align}
%-K\Delta u + \nabla \cdot (\vec{\beta}u) & = 0,&&\text{in } \Omega, \label{problem_cd_3}
%\\
u &=0 ,&&\text{on } \{(x,y): y=0,~x\in(0,1)\} \notag,
\\
u &=1 ,&&\text{on } \{(x,y): y=1,~x\in(0,1)\} \notag,
\\
u &=\frac{1}{2}+\frac{1}{\pi} \arctan(10^4(y-0.7)) ,&&\text{on } \{(x,y): x=0,~y\in(0,1)\}, \notag
\\
u &=0 &&\text{on } \{(x,y): x=1,~y\in(0,1)\} \notag.
\end{align}
For DG1, we put $\epsilon=-1$ and $\sigma=14$.  Fig.~\ref{figB_1:sixgrid} contrasts the full CCFV and DG1 schemes. CCFV shows greater numerical diffusion, while DG1 yields sharper features at the cost of superious oscillations.
\begin{figure}[H] % use [htbp] if you want more placement flexibility
  \centering 
  % ---- Row 1 ----
  \begin{subfigure}{0.32\textwidth}
    \includegraphics[width=\linewidth]{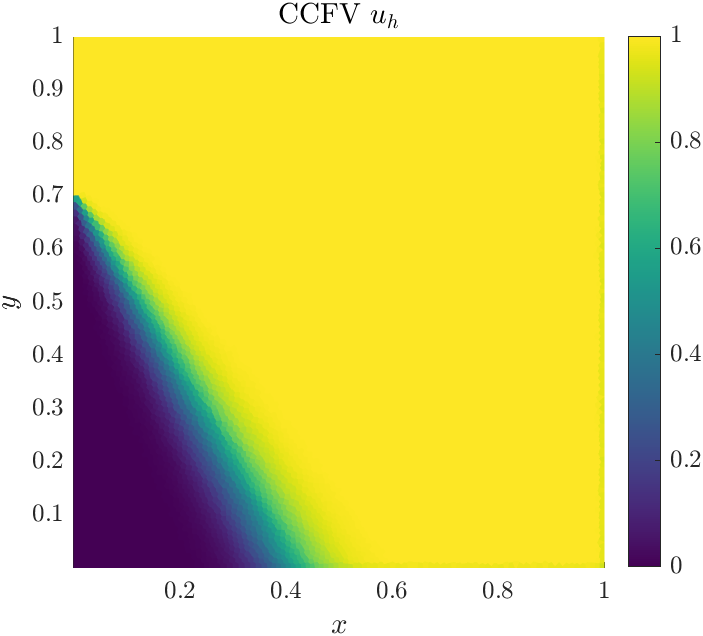}
    \subcaption{CCFV 10k elements}\label{figB_1:1a}
  \end{subfigure}\hfill
  \begin{subfigure}{0.32\textwidth}
    \includegraphics[width=\linewidth]{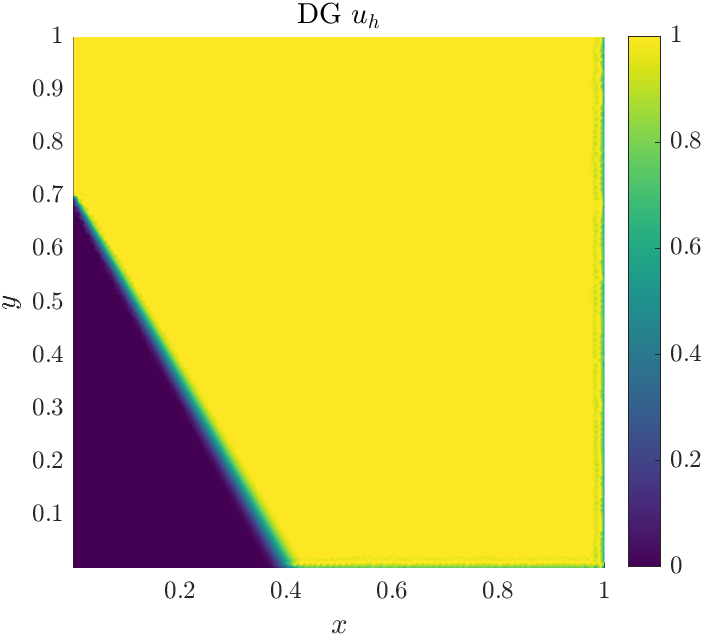}
    \subcaption{DG 10k elements}\label{figB_1:1b}
  \end{subfigure}\hfill
  \begin{subfigure}{0.34\textwidth}
    \includegraphics[width=\linewidth]{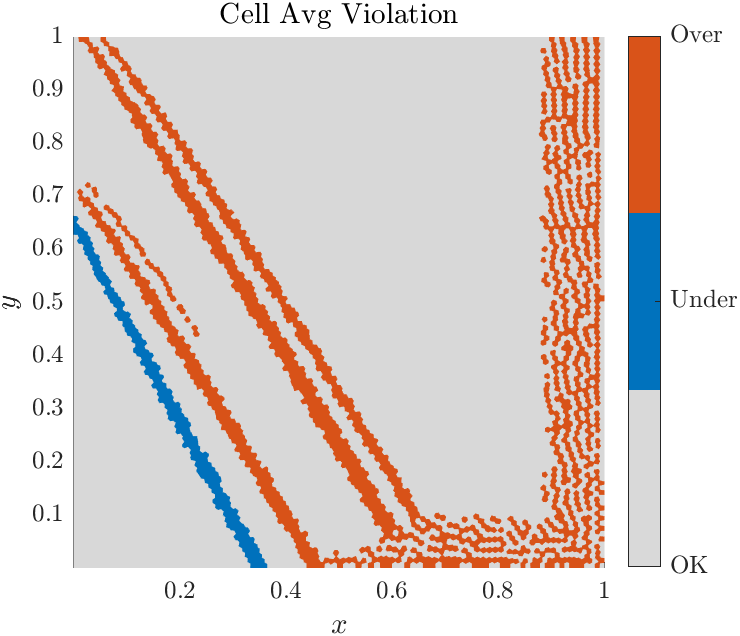}
    \subcaption{DG 10k elements (over/undershoot)}\label{figB_1:1c}
  \end{subfigure}
  \medskip % vertical gap between the two rows
  % ---- Row 2 ----
  \begin{subfigure}{0.32\textwidth}
    \includegraphics[width=\linewidth]{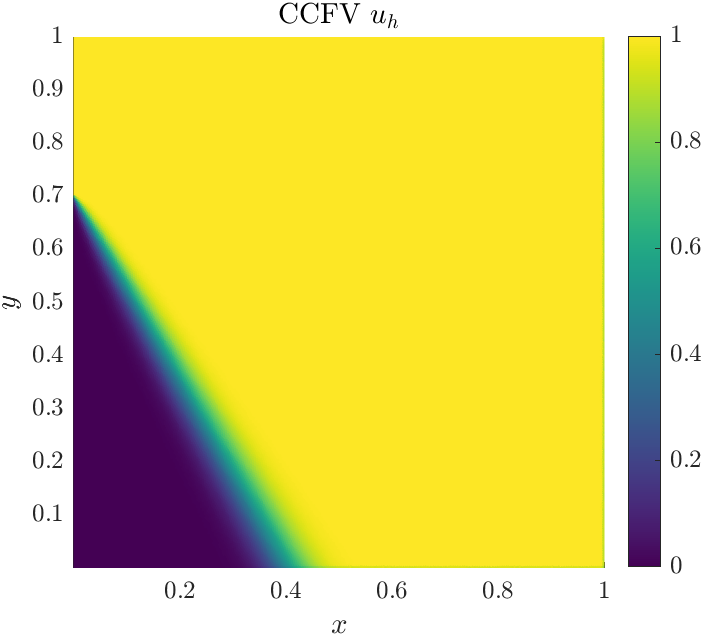}
    \subcaption{CCFV 65k elements}\label{figB_1:1d}
  \end{subfigure}\hfill
  \begin{subfigure}{0.32\textwidth}
    \includegraphics[width=\linewidth]{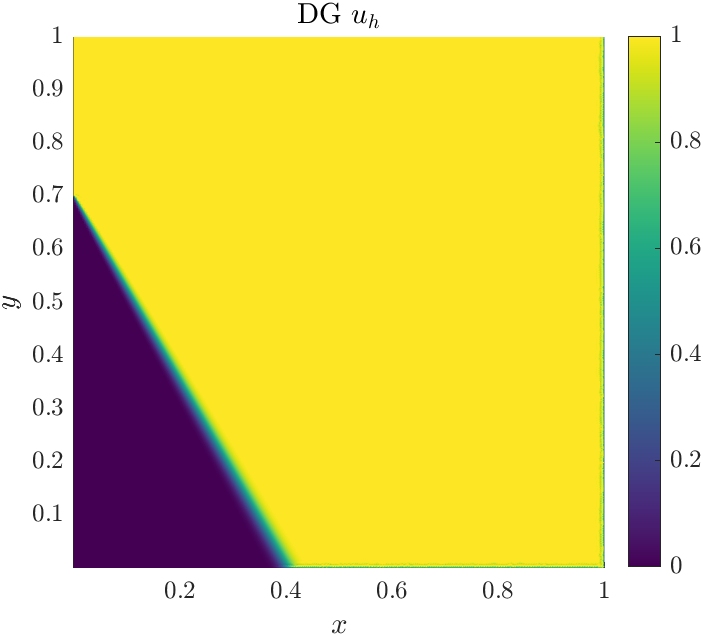}
    \subcaption{DG 65k elements}\label{figB_1:1e}
  \end{subfigure}\hfill
  \begin{subfigure}{0.34\textwidth}
    \includegraphics[width=\linewidth]{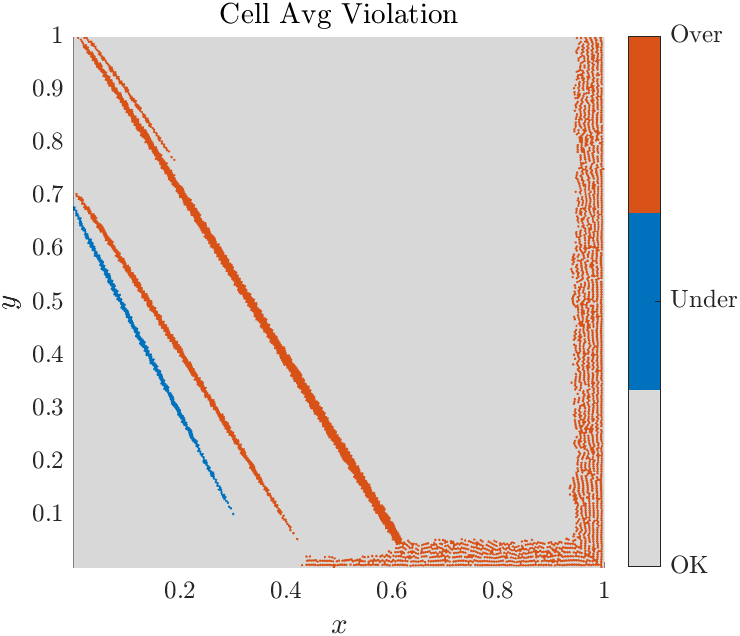}
    \subcaption{DG 65k elements (over/undershoot)}\label{figB_1:1f}
  \end{subfigure} 
  \caption{Baseline tests from Section~\ref{sec:badia} - use CCFV everywhere or DG everywhere with various mesh sizes.}
  \label{figB_1:sixgrid}
\end{figure}   
      
In Fig.~\ref{figB_2:sixgrid} we compare FV-DG1 and FV-DG2 using the OBB method~\cite{oden1998discontinuoushpfinite} with \(\epsilon=1\), \(\sigma=0\) on interior edges, and \(\sigma=1\) on boundary edges. In both configurations, DG covers most of the domain, while boundary layers near \(x=1\) and \(y=0\) are adaptively assigned FV cells to limit overshoot. Mesh refinement and increasing the DG degree from \(k=1\) to \(k=2\) both enlarge the DG regions.
\newcommand{\cellw}{0.24\textwidth} 
\begin{figure}[H]
  \centering % ---- Row 1 ----
  \begin{subfigure}[t]{\cellw}
    \centering
    \includegraphics[width=\linewidth% ,height=\imgheight,keepaspectratio
    ]{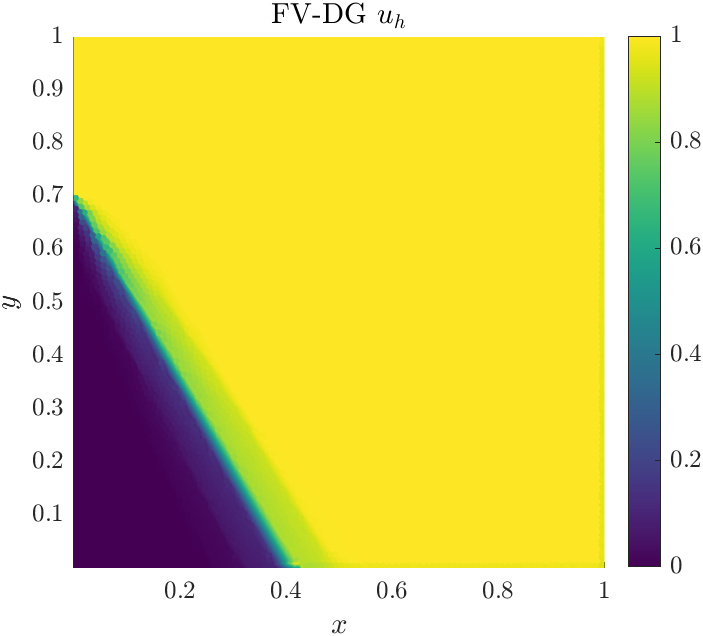}
    \caption{FV-DG1 10k elements (56\% DG, 44\% FV)}\label{figB_2:1a}
  \end{subfigure}\hfill
  \begin{subfigure}[t]{\cellw}
    \centering
    \includegraphics[width=\linewidth% ,height=\imgheight,keepaspectratio
    ]{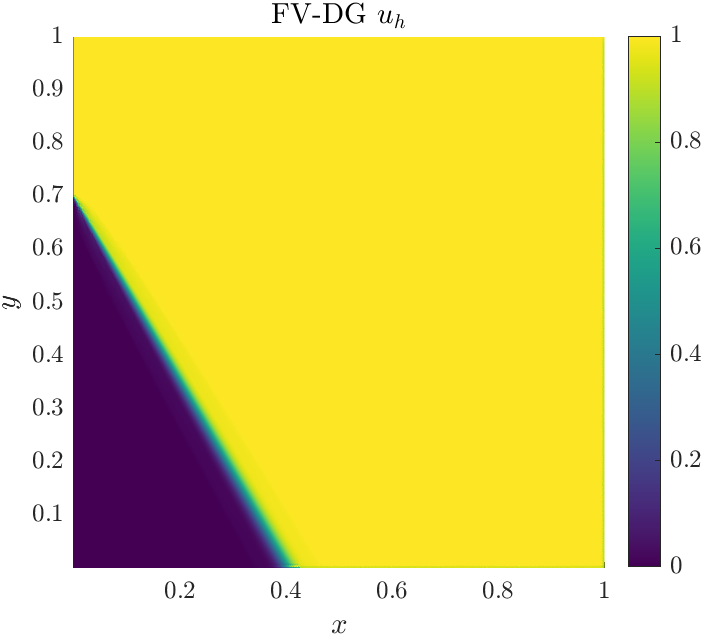}
    \caption{FV-DG1 65k elements (84\% DG, 16\% FV)}\label{figB_2:1b}
  \end{subfigure}\hfill
  \begin{subfigure}[t]{\cellw}
    \centering
    \includegraphics[width=\linewidth% ,height=\imgheight,keepaspectratio
    ]{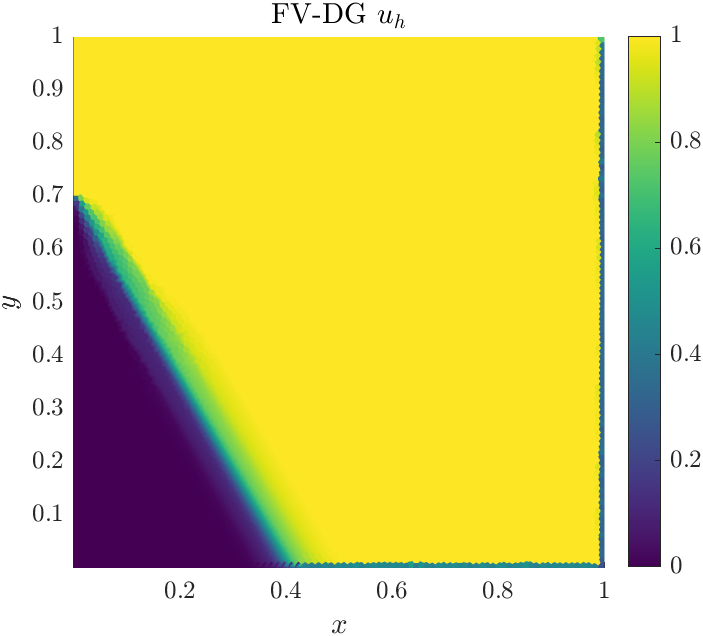}
    \caption{FV-DG2 10k elements (88\% DG, 12\% FV)}\label{figB_2:1c}
  \end{subfigure}\hfill
  \begin{subfigure}[t]{\cellw}
    \centering
    \includegraphics[width=\linewidth% ,height=\imgheight,keepaspectratio
    ]{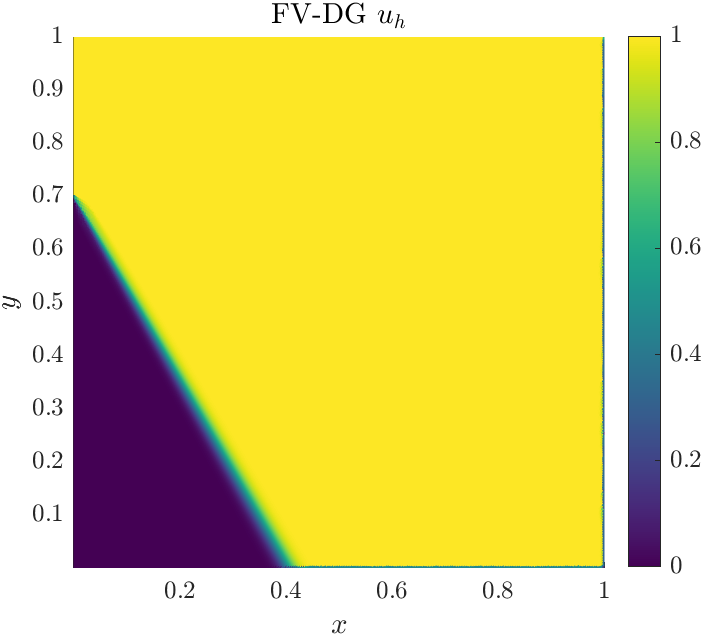}
    \caption{FV-DG2 65k elements (96\% DG, 4\% FV)}\label{figB_2:1d}
  \end{subfigure}

  \vspace{0.6em}

  % ---- Row 2 ----
  \begin{subfigure}[t]{0.24\textwidth}
    \centering
    \includegraphics[width=\linewidth% ,height=\imgheight,keepaspectratio
    ]{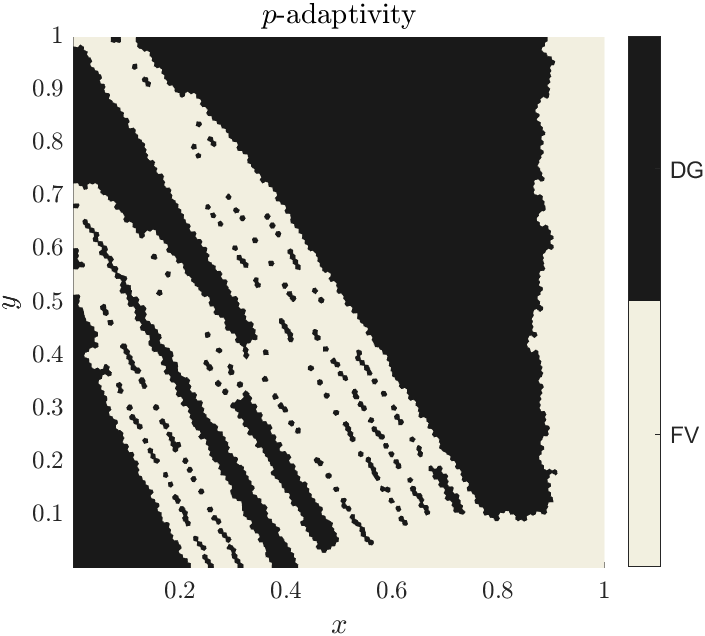}
    \caption{FV-DG1 10k elements (88\% DG, 12\% FV)}\label{figB_2:1e}
  \end{subfigure}\hfill
  \begin{subfigure}[t]{\cellw}
    \centering
    \includegraphics[width=\linewidth% ,height=\imgheight,keepaspectratio
    ]{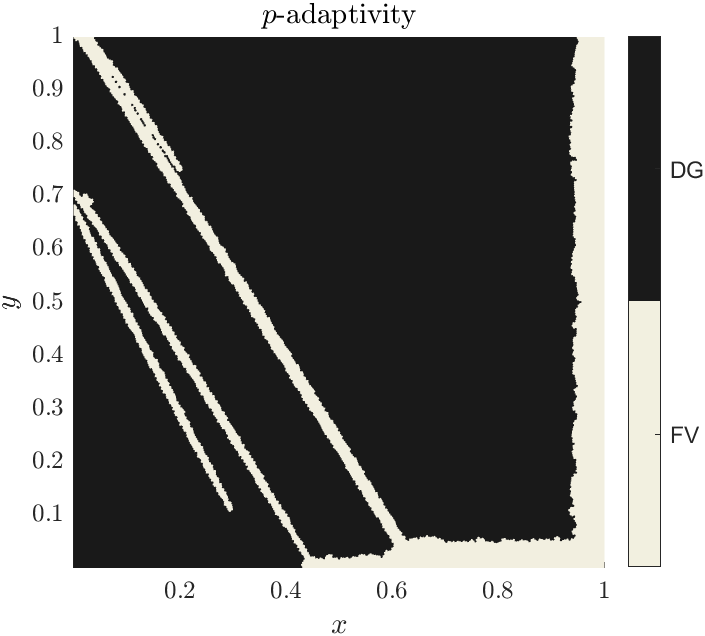}
    \caption{FV-DG1 65k elements (84\% DG, 16\% FV)}\label{figB_2:1f}
  \end{subfigure}\hfill
  \begin{subfigure}[t]{\cellw}
    \centering
    \includegraphics[width=\linewidth% ,height=\imgheight,keepaspectratio
    ]{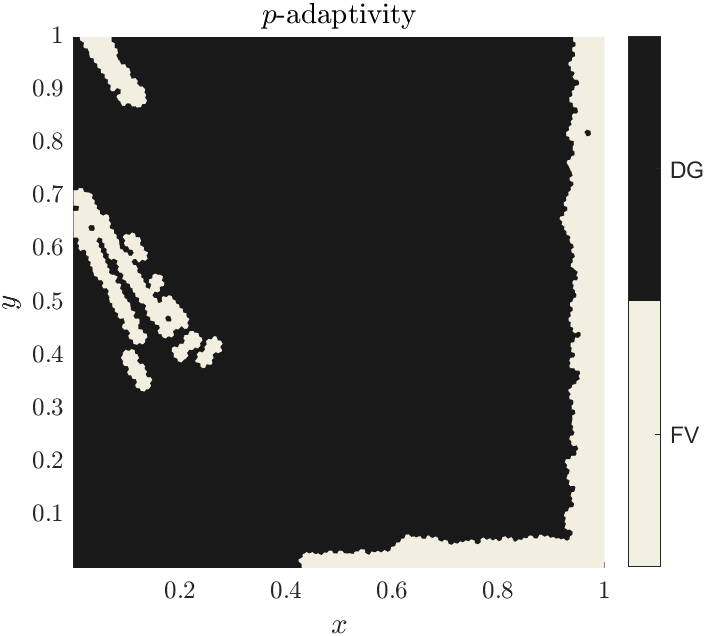}
    \caption{FV-DG2 10k elements (88\% DG, 12\% FV)}\label{figB_2:1g}
  \end{subfigure}\hfill
  \begin{subfigure}[t]{\cellw}
    \centering
    \includegraphics[width=\linewidth% ,height=\imgheight,keepaspectratio
    ]{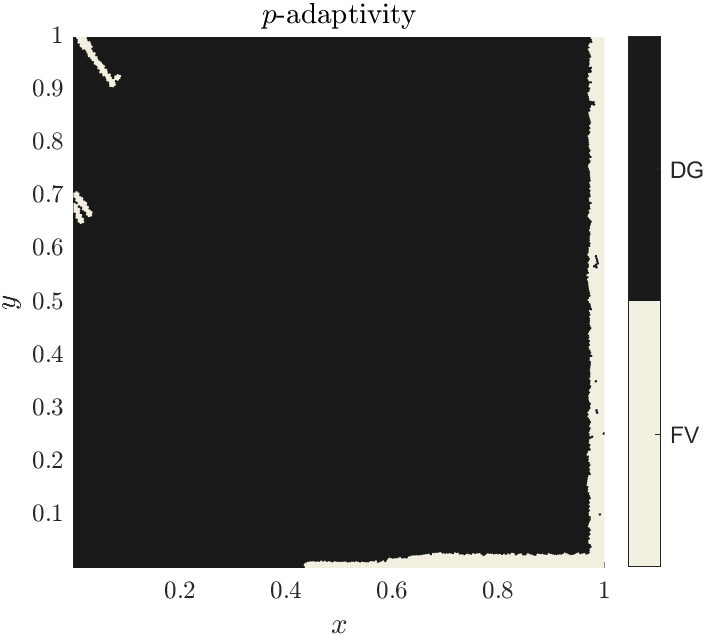}
    \caption{FV-DG2 65k elements (96\% DG, 4\% FV)}\label{figB_2:1h}
  \end{subfigure} 
  \caption{Adaptive scheme for FV-DG with $\delta = 10^{-9}$.}
  \label{figB_2:sixgrid}
\end{figure}
   The adaptive technique is quite effective for this test case.
   
\section{Hemker problem (convection-diffusion)}\label{num_H}
	As introduced in~\cite{HEMKER1996277}, the Hemker problem has become a standard benchmark for convection--diffusion models, exhibiting many of the salient features encountered in applications. It describes the transport of energy from a body through a channel. The solution remains within \([0,1]\). Sharp boundary layers form at the interior boundary, and interior layers propagate from the body in the direction of convection.
 
	The domain is a rectangle with a circular puncture at the origin with radius one: $[(-3,9)\times (-3,3)] \backslash \{(x,y): x^2+y^2\le 1\}$.  A Dirichlet condition of one is supplied on the circle boundary.  Homogeneous Dirichlet boundary conditions are supplied at $x=-3$.  The remaining boundaries are no flow (zero Neumann conditions).  A value of $K = 10^{-8}$ is taken to ensure a strongly convection-dominated regime.  Here  and $\vec{\beta}=[1,0]^T$.     
   
%Similar to the other test problems, we consider baseline approximations with full CCFV and full DG. Fig.~\ref{figH:two-by-three} shows the results for different mesh resolutions.  Full DG exhibits considerable undershoot and overshoot, but the DG cell averages mostly have overshoot.  CCFV is substantially more diffusive.
	Consistent with the other benchmarks, we consider baseline approximations using full CCFV and full DG. Fig.~\ref{figH:two-by-three} presents results at multiple mesh resolutions. Full DG exhibits considerable bound violations (undershoot/overshoot); notably, the DG cell averages are predominantly overshooting. By contrast, CCFV is markedly more diffusive.   	
   
	With \(\delta=10^{-9}\), the FV--DG adaptive scheme (Fig.~\ref{figH2:two-by-three}) delivers clear gains over Figs.~\ref{figH:a} and~\ref{figH:d}: boundary layers are resolved more sharply, and continued mesh refinement expands the DG coverage. Regions of high DG coverage coincide with areas of higher solution regularity.   	
\begin{figure}[H]
  \centering  % Row 1
  \begin{subfigure}[t]{0.32\textwidth}
    \includegraphics[width=\linewidth]{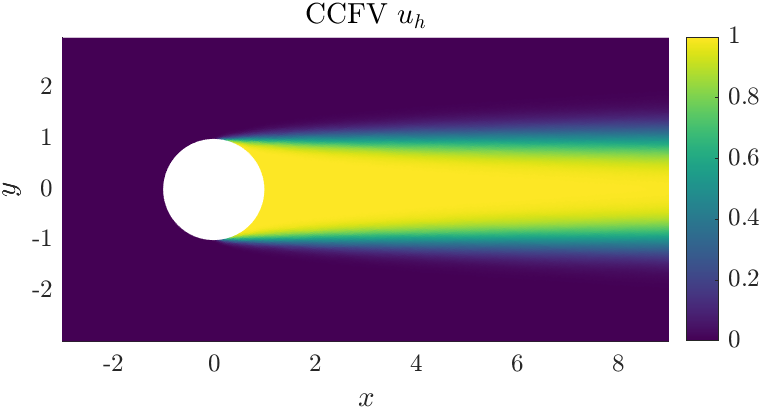}
    \caption{CCFV 65k}\label{figH:a}
  \end{subfigure}\hfill
  \begin{subfigure}[t]{0.32\textwidth}
    \includegraphics[width=\linewidth]{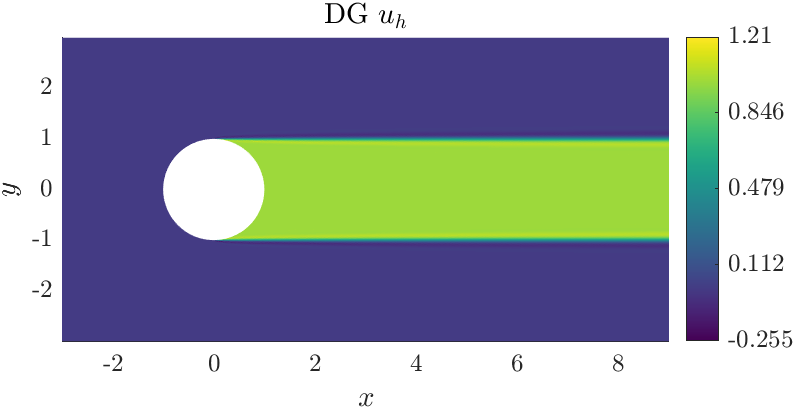}
    \caption{DG1 65k}\label{figH:b}
  \end{subfigure}\hfill
  \begin{subfigure}[t]{0.32\textwidth}
    \includegraphics[width=\linewidth]{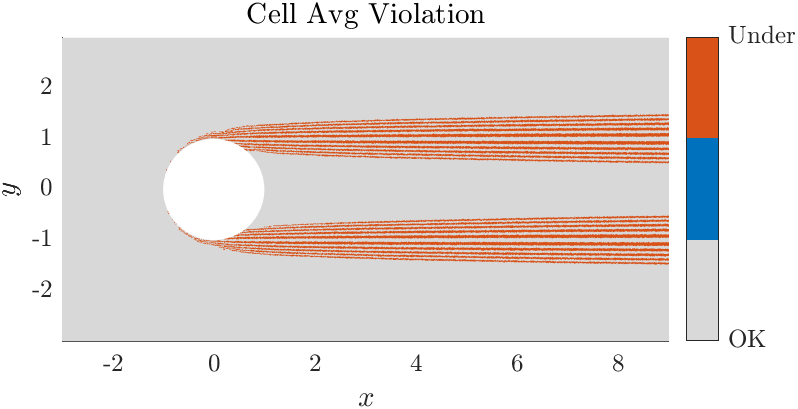}
    \caption{DG1 65k (over/undershoot)}\label{figH:c}
  \end{subfigure} 
  \vspace{0.6em} % Row 2
  \begin{subfigure}[t]{0.32\textwidth}
    \includegraphics[width=\linewidth]{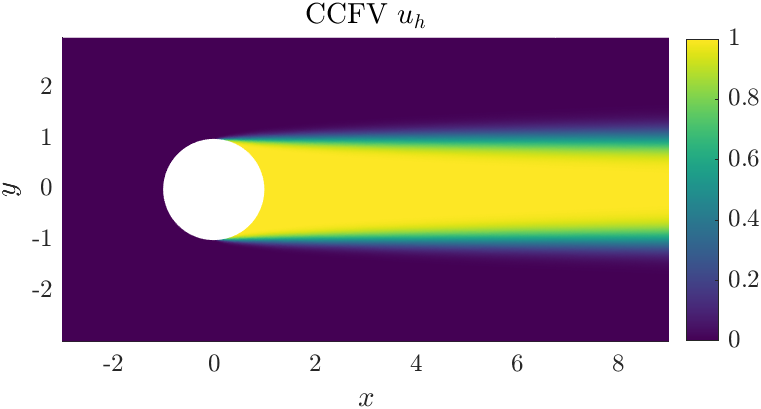}
    \caption{CCFV 262k}\label{figH:d}
  \end{subfigure}\hfill
  \begin{subfigure}[t]{0.32\textwidth}
    \includegraphics[width=\linewidth]{hemker_DG1_65k_1_1e_m_9.png}
    \caption{DG1 262k}\label{figH:e}
  \end{subfigure}\hfill
  \begin{subfigure}[t]{0.32\textwidth}
    \includegraphics[width=\linewidth]{hemker_DG1_262k_2_1e_m_9.png}
    \caption{DG1 262k (over/undershoot)}\label{figH:f}
  \end{subfigure} 
  \caption{Baseline tests for the Hemker problem.}
  \label{figH:two-by-three}
\end{figure}  
 
\begin{figure}[H]
  \centering % Row 1
  \begin{subfigure}[t]{0.32\textwidth}
    \includegraphics[width=\linewidth]{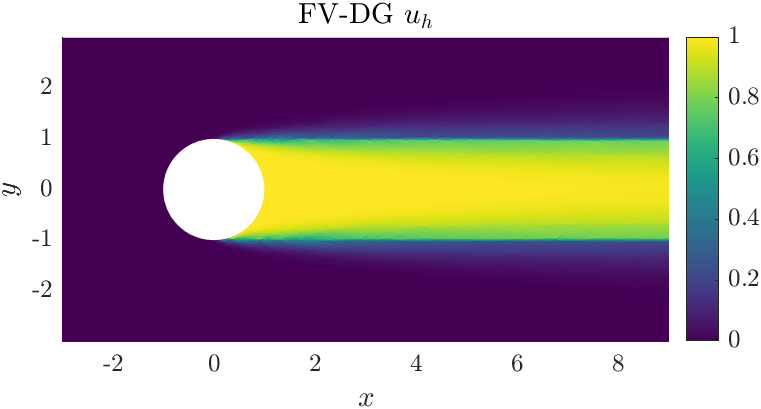}
    \caption{FV-DG 30k (41\% DG,59\% FV)}\label{figH2:a}
  \end{subfigure}\hfill
  \begin{subfigure}[t]{0.32\textwidth}
    \includegraphics[width=\linewidth]{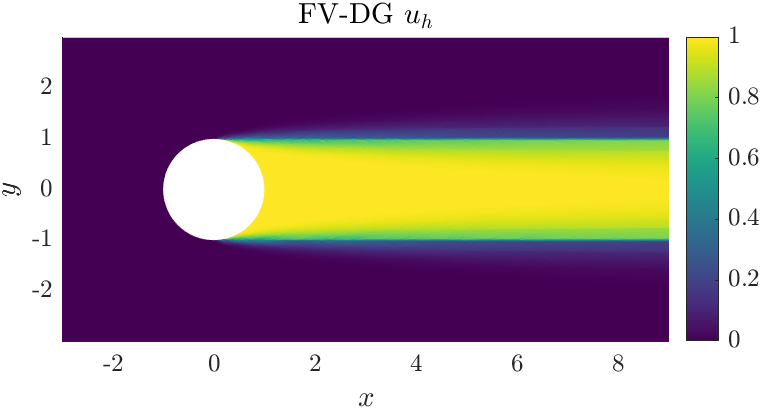}
    \caption{65k  (48\% DG,52\% FV)}\label{figH2:b}
  \end{subfigure}\hfill
  \begin{subfigure}[t]{0.32\textwidth}
    \includegraphics[width=\linewidth]{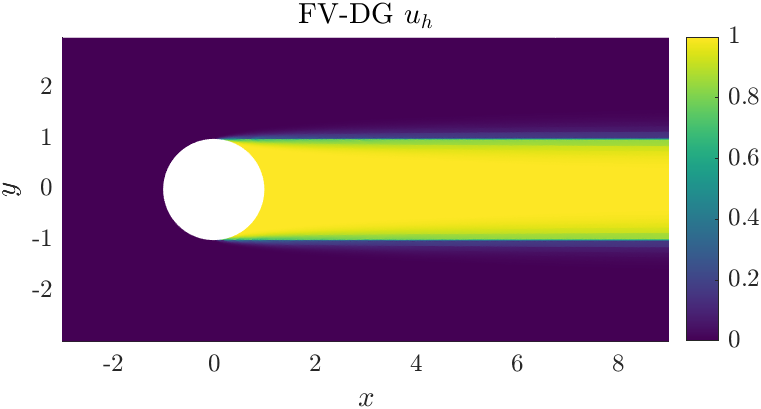}
    \caption{262k (52\% DG,48\% FV)}\label{figH2:c}
  \end{subfigure} 
  \vspace{0.6em} % Row 2
  \begin{subfigure}[t]{0.32\textwidth}
    \includegraphics[width=\linewidth]{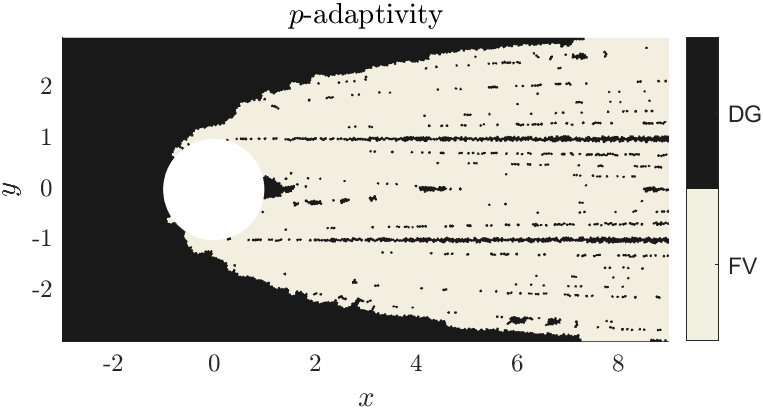}
    \caption{30k  (41\% DG,59\% FV)}\label{figH2:d}
  \end{subfigure}\hfill
  \begin{subfigure}[t]{0.32\textwidth}
    \includegraphics[width=\linewidth]{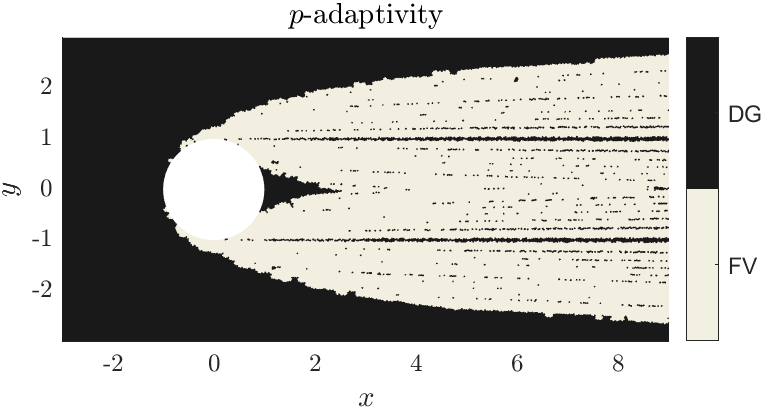}
    \caption{65k  (48\% DG,52\% FV)}\label{figH2:e}
  \end{subfigure}\hfill
  \begin{subfigure}[t]{0.32\textwidth}
    \includegraphics[width=\linewidth]{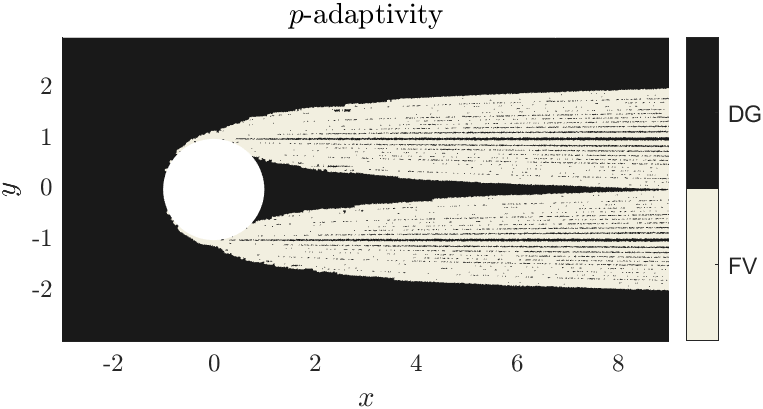}
    \caption{262k (52\% DG,48\% FV)}\label{figH2:f}
  \end{subfigure} 
  \caption{FV-DG1 results for the Hemker problem with $\delta=10^{-9}$.}
  \label{figH2:two-by-three}
\end{figure}   

Figs.~\ref{figH3:two-by-three} and~\ref{figH4:two-by-three} visualizes the various approximations from a different vantage point. A view of the solution cross section along $x = 2$ is given in Fig.~\ref{figH4:two-by-three}, along with magnification.  The overshoot and undershoot from the DG scheme is more noticeable. The solutions from CCFV and FV-DG are more diffusive, but bound-preserving.  Between the CCFV and FV-DG solutions, the latter is less diffusive.
\begin{figure}[H]
  \centering % Row 1
  \begin{subfigure}[t]{0.32\textwidth}
    \includegraphics[width=\linewidth]{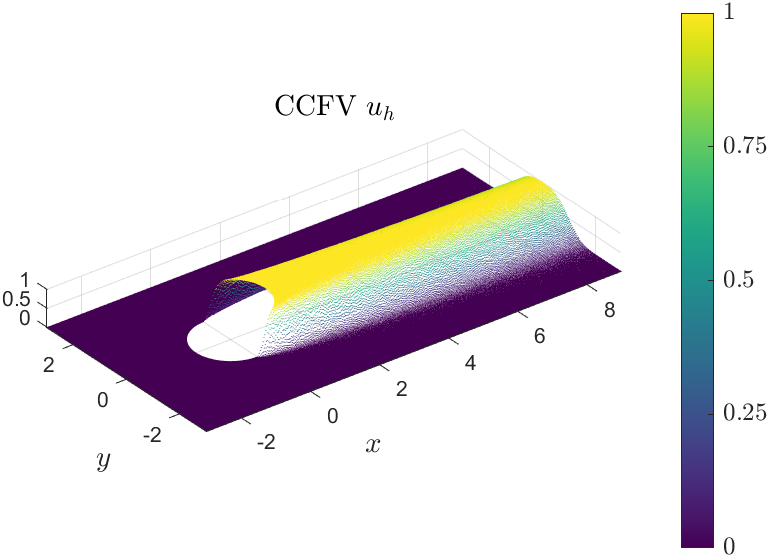}
    \caption{CCFV 65k}\label{figH3:a}
  \end{subfigure}\hfill
  \begin{subfigure}[t]{0.32\textwidth}
    \includegraphics[width=\linewidth]{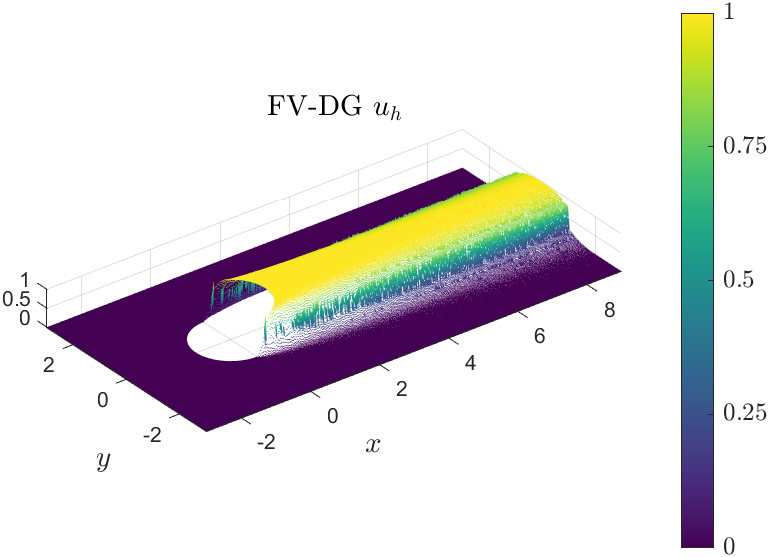}
    \caption{FV-DG 65k}\label{figH3:b}
  \end{subfigure}\hfill
  \begin{subfigure}[t]{0.32\textwidth}
    \includegraphics[width=\linewidth]{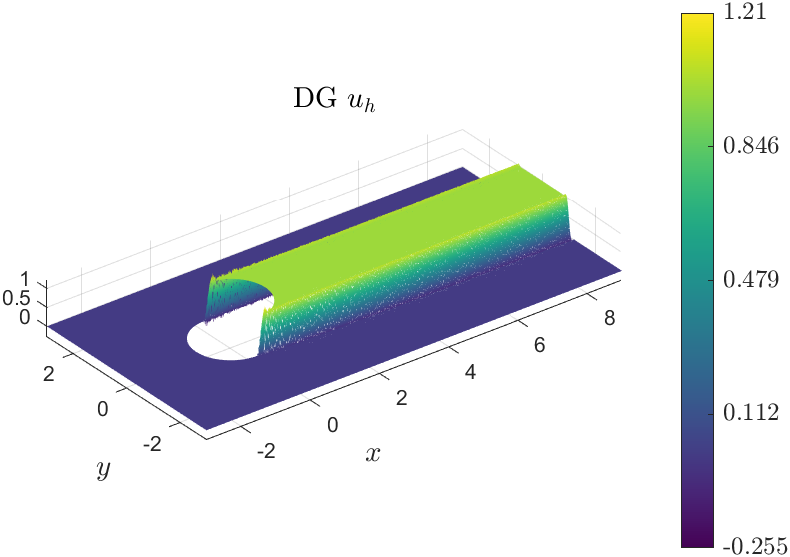}
    \caption{DG1 65k}\label{figH3:c}
  \end{subfigure}  
  \caption{3D visualization the Hemker problem with CCFV, DG1, and FV-DG ($\delta=10^{-9}$).}
  \label{figH3:two-by-three}
\end{figure}    
 
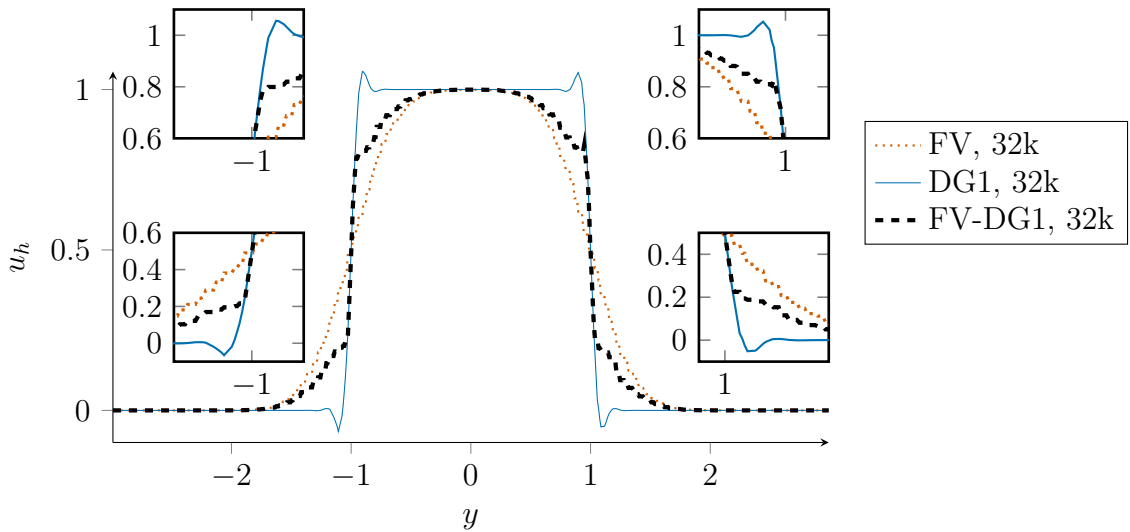
\begin{figure}[H]
  \centering % Row 1
\begin{tikzpicture} 
  \begin{axis}[
  name=main, 
  ymin=-0.1,   % lower bound,      % add 5% padding at bottom
  axis lines=left,      % no top/right edges
  tick align=outside,   % nice tick look
  %grid=both,            % optional
  legend cell align={left},
    width=11cm, height=6.5cm,
    xlabel={$y$}, ylabel={$u_h$},  
%    extra y ticks={-0.15,1.10},
	legend style={at={(1.05,0.7)}, anchor=west},
    clip=false, % let the lens stick outside the axis box 
    spy using outlines={
      circle,
      magnification=4,
      size=3cm,
      connect spies
    },
    unbounded coords=jump % ignore NaNs/Infs
  ]
    \addplot+[thick, mark=none,dotted, OIverm]
      table[
        col sep=space,
        header=false,
        x index=0,  % first column = P(:,2) (your y)
        y index=1   % second column = uline
      ] {hemker_FV_65k_profile_f.dat};
      
    \addplot+[thin, mark=none, OIblue]
      table[
        col sep=space,
        header=false,
        x index=0,  % first column = P(:,2) (your y)
        y index=1   % second column = uline
      ] {hemker_DG_65k_profile_f.dat};     
       
    \addplot+[ultra thick, mark=none, dashed,OIblack]%+[only marks, mark=o, mark size=2pt, mark options={fill=none}, each nth point={10}]%+[thick, mark=none]
      table[
        col sep=space,
        header=false,
        x index=0,  % first column = P(:,2) (your y)
        y index=1   % second column = uline
      ] {hemker_FV_DG_65k_profile_f.dat};   
    \addlegendentry{FV, 32k}
    \addlegendentry{DG1, 32k}
    \addlegendentry{FV-DG1, 32k}
  \end{axis}
   
  % -------- inset (zoom) axis --------
  \begin{axis}[
    name=zoom,
    at={(main.north east)}, anchor=north east, % position relative to main
    xshift= -0.0\linewidth, yshift=-0.13\linewidth, % small offset
    width =0.20\linewidth,
    height=0.20\linewidth,
    xtick={1.0},
    tick style={semithick},
    line width=1pt,
    legend=false,           % no legend in the inset
    xmin= 0.9, xmax=1.4,   % <<< set your zoom window here
    ymin=-0.1, ymax=0.5,
    % ymin=<...>, ymax=<...> % optional: tighten y-range
  ]
    % Re-plot the same data; 'forget plot' keeps the main legend unchanged
    \addplot+[very thick, mark=none, dotted,OIverm]
      table[
        col sep=space,
        header=false,
        x index=0,  % first column = P(:,2) (your y)
        y index=1   % second column = uline
      ] {hemker_FV_65k_profile_f.dat};
      
    \addplot+[thick, mark=none,OIblue]
      table[
        col sep=space,
        header=false,
        x index=0,  % first column = P(:,2) (your y)
        y index=1   % second column = uline
      ] {hemker_DG_65k_profile_f.dat};     
       
    \addplot+[ultra thick, mark=none,dashed,OIblack]
      table[
        col sep=space,
        header=false,
        x index=0,  % first column = P(:,2) (your y)
        y index=1   % second column = uline
      ] {hemker_FV_DG_65k_profile_f.dat}; 
  \end{axis}

  \begin{axis}[
    name=zoom0,
    at={(main.north east)}, anchor=north east, % position relative to main
    xshift= -0.0\linewidth, yshift= 0.05\linewidth, % small offset
    width =0.20\linewidth,
    height=0.20\linewidth,
    xtick={1.0},
    tick style={semithick},
    line width=1pt,
    legend=false,           % no legend in the inset
    xmin= 0.6, xmax=1.2,   % <<< set your zoom window here
    ymin= 0.6, ymax=1.1,
    % ymin=<...>, ymax=<...> % optional: tighten y-range
  ]
    % Re-plot the same data; 'forget plot' keeps the main legend unchanged
    \addplot+[very thick, mark=none, dotted,OIverm]
      table[
        col sep=space,
        header=false,
        x index=0,  % first column = P(:,2) (your y)
        y index=1   % second column = uline
      ] {hemker_FV_65k_profile_f.dat};
      
    \addplot+[thick, mark=none,OIblue]
      table[
        col sep=space,
        header=false,
        x index=0,  % first column = P(:,2) (your y)
        y index=1   % second column = uline
      ] {hemker_DG_65k_profile_f.dat};     
       
    \addplot+[ultra thick, mark=none,dashed,OIblack]
      table[
        col sep=space,
        header=false,
        x index=0,  % first column = P(:,2) (your y)
        y index=1   % second column = uline
      ] {hemker_FV_DG_65k_profile_f.dat}; 
  \end{axis}

  \begin{axis}[
    name=zoom2,
    at={(main.north east)}, anchor=north east, % position relative to main
    xshift= -0.42\linewidth, yshift= 0.05\linewidth, % small offset
    width =0.20\linewidth,
    height=0.20\linewidth,
    xtick={-1.0},
    tick style={semithick},
    line width=1pt,
    legend=false,           % no legend in the inset
    xmin= -1.3, xmax=-0.8,   % <<< set your zoom window here
    ymin=0.6, ymax=1.1,
%    xtick={-1.3,-0.9},
    % ymin=<...>, ymax=<...> % optional: tighten y-range
  ]
    % Re-plot the same data; 'forget plot' keeps the main legend unchanged
    \addplot+[very thick, mark=none, dotted,OIverm]
      table[
        col sep=space,
        header=false,
        x index=0,  % first column = P(:,2) (your y)
        y index=1   % second column = uline
      ] {hemker_FV_65k_profile_f.dat};%hemker_FV_65k_profile
      
    \addplot+[thick, mark=none,OIblue]
      table[
        col sep=space,
        header=false,
        x index=0,  % first column = P(:,2) (your y)
        y index=1   % second column = uline
      ] {hemker_DG_65k_profile_f.dat};     
       
    \addplot+[ultra thick, mark=none,dashed,OIblack]
      table[
        col sep=space,
        header=false,
        x index=0,  % first column = P(:,2) (your y)
        y index=1   % second column = uline
      ] {hemker_FV_DG_65k_profile_f.dat}; %hemker_FVDG_262k_profile
  \end{axis}

  \begin{axis}[
    name=zoom3,
    at={(main.north east)}, anchor=north east, % position relative to main
    xshift= -0.42\linewidth, yshift= -0.13\linewidth, % small offset
    width =0.20\linewidth,
    height=0.20\linewidth,
    xtick={-1.0},
    tick style={semithick},
    line width=1pt,
    legend=false,           % no legend in the inset
    xmin= -1.3, xmax=-0.8,   % <<< set your zoom window here
    ymin=-0.1, ymax=0.6,
%    xtick={-1.3,-0.9},
    % ymin=<...>, ymax=<...> % optional: tighten y-range
  ]
    % Re-plot the same data; 'forget plot' keeps the main legend unchanged
    \addplot+[very thick, mark=none, dotted,OIverm]
      table[
        col sep=space,
        header=false,
        x index=0,  % first column = P(:,2) (your y)
        y index=1   % second column = uline
      ] {hemker_FV_65k_profile_f.dat};%hemker_FV_65k_profile
      
    \addplot+[thick, mark=none,OIblue]
      table[
        col sep=space,
        header=false,
        x index=0,  % first column = P(:,2) (your y)
        y index=1   % second column = uline
      ] {hemker_DG_65k_profile_f.dat};     
       
    \addplot+[ultra thick, mark=none,dashed ,OIblack]
      table[
        col sep=space,
        header=false,
        x index=0,  % first column = P(:,2) (your y)
        y index=1   % second column = uline
      ] {hemker_FV_DG_65k_profile_f.dat}; %hemker_FVDG_262k_profile
  \end{axis}    
\end{tikzpicture}
\caption{Profile along $x = 2$ for the Hemker problem.}
  \label{figH4:two-by-three}
\end{figure}  
   
\section{Iteration counts,  measure for the size of the spurious oscillations}\label{num_M}
	Here we detail some metrics for the adaptive technique from Section~\ref{sec:adapt}.  In \cite{FRERICHS2021113487}, a metric for quantifying spurious oscillations is proposed; we employ that measure here:
\[
\text{osc}(u_h) = 
\frac{1}{|\mathcal{T}_h|} 
\bigg[
\sum_{E\in \mathcal{T}_h} 
\max\{0, -u_{\text{max}} + \max_{(x,y)\in E} u_h(x,y)\}
+
\max\{0,  u_{\text{min}} - \max_{(x,y)\in E} u_h(x,y)\}
\bigg]
,
\]
where $u_h$ is understood to be the solution to~\eqref{eq:discrete-problem}, and the exact solution to~\eqref{model_problem1} satisfies $u_{\text{min}} \le u\le u_{\text{max}}$.  As noted in Section~\ref{secT:oscillations}, $\operatorname{osc}(u_h)$ reports only deviations beyond $\min u$ and $\max u$; in-range oscillations are not captured. A value of $\operatorname{osc}(u_h)$ near zero indicates that the approximation does not exhibit large extremal oscillations.

	Table~\ref{tab:summary} reports the results. The \texttt{iters} column gives the total count of adaptation cycles (activations of the procedure in Section~\ref{sec:adapt}).  Across all cases, we fix the mesh. For the full DG scheme, $\operatorname{osc}(u_h)$ stays on the order of $10^{-1}$ and shows little reduction under mesh refinement. The adaptive strategy of Section~\ref{sec:adapt} yields a bound-preserving FV-DG scheme up to the prescribed tolerance $\delta$. As $\delta$ is decreased, the measured oscillation magnitude drops accordingly, with $\operatorname{osc}(u_h)\approx \delta$.
	
	In the \texttt{iters} column, we see that the adaptation cycles required to meet a certain tolerance remains under 10. As one might expect, as the tolerance $\delta$ is made more strict, addition iterations are required.  A more aggressive patch neighborhood $\mathcal{N}$ can also reduce the number of iterations. Adaptation cycles contribute negligibly to the total runtime, whereas the linear system solve and Voronoi remeshing are the primary bottlenecks.
\begin{table}[H]
\centering
\caption{Summary of metrics for each test case.}
\label{tab:summary}
\setlength{\tabcolsep}{6pt}
\begin{tabular}{@{} l l c c c c c @{}}
\toprule
Problem & Case & $|\mathcal{T}_h|$ & $\mathrm{osc}_{\mathrm{avg}}$ & \texttt{iters} & $\delta$ &   \\
\midrule
\multirow{4}{*}{Triple layer (Section~\ref{sec:Triple})} & DG1     & 262k & 1.280$\cdot 10^{-1}$ & --  & --   & --  \\
                        & FV-DG1  & 262k & 2.128$\cdot 10^{-6}$  & 4   & $10^{-6}$  &   \\
                        & FV-DG1  & 262k & 1.102$\cdot 10^{-9}$  & 6   & $10^{-9}$  &   \\                           
                        & FV-DG1  & 262k & 3.120$\cdot 10^{-13}$ & 7   & $10^{-13}$  &   \\  
                        \\
\addlinespace[2pt]
\multirow{4}{*}{L-Shaped (Section~\ref{num_L})} & DG1     & 262k   & 1.332$\cdot 10^{-1 }$ & --  & --   & --  \\
                          & FV-DG1  & 262k   & 1.923$\cdot 10^{-11}$  & 3   & $10^{-6}$  &   \\
                          & FV-DG1  & 262k   & 1.923$\cdot 10^{-11}$  & 3   & $10^{-9}$  &   \\                           
                          & FV-DG1  & 262k   & 1.817$\cdot 10^{-13}$  & 4   & $10^{-13}$  &   \\  
                        \\
\addlinespace[2pt]
\multirow{4}{*}{Internal layer (Section~\ref{sec:badia})} & DG1 & 262k      &   & --  & --   & --  \\
                                & FV-DG1  & 262k   & 2.832$\cdot 10^{-6}$   & 2   & $10^{-6}$  &   \\
                                & FV-DG1  & 262k   & 9.952$\cdot 10^{-9}$   & 3   & $10^{-9}$  &   \\                           
                                & FV-DG1  & 262k   & 2.256$\cdot 10^{-13}$  & 4   & $10^{-13}$  &   \\                            
                        \\
\addlinespace[2pt]
\multirow{4}{*}{Hemker (Section~\ref{num_H})} & DG1    & 262k & 7.040$\cdot 10^{-2}$ & -- & --       & --  \\
                        & FV-DG1 & 262k & 3.195$\cdot 10^{-6}$  &  5 & $10^{-6}$ &   \\
                        & FV-DG1 & 262k & 1.330$\cdot 10^{-9}$  &  6 & $10^{-9}$ &   \\
                        & FV-DG1 & 262k & 1.162$\cdot 10^{-13}$ &  8 & $10^{-13}$&   \\                         
\bottomrule
\end{tabular}
\end{table} 
\section{Conclusion}
%	In this paper, we solve the convection–diffusion equation using a coupling of cell-centered finite volume and discontinuous Galerkin methods.  We propose a novel adaptive partitioning strategy that automatically determines FV and DG subdomains. If the solution's cell average violates the bounds, all elements in a small neighborhood are added to the FV partition.  We demonstrated that this procedure fits naturally in the standard $hp$-adaptivity context. This process is repeated until all cell averages are bound-preserving (up to some specified tolerance). The method always terminates with cell averages within bounds, because the CCFV method is monotone. Numerous convection-dominated benchmarks confirm the effectiveness of the adaptive technique. 
	In this paper we solve the steady convection–diffusion equation using a multinumerics coupling of cell-centered finite volume and discontinuous Galerkin methods. We introduce a novel adaptive partitioning strategy that automatically selects FV and DG subdomains: whenever the cell average on an element falls outside the admissible range, that element and a small neighborhood are reassigned to the FV partition. The procedure integrates naturally with standard $hp$-adaptivity and is iterated until all cell averages satisfy the bounds up to a user-specified tolerance. Termination is guaranteed because the CCFV discretization is monotone.  Since the DG and FV methods are not modified in any way during adaption, standard stability and error analysis ensure the scheme is grounded in a rigorous framework.  Numerical results on convection-dominated benchmarks demonstrate the effectiveness of the approach.
 
	Future work will investigate alternative selection of the patch neighborhoods, slope limiters for Voronoi meshes, and algorithmic changes to Variants~1–2. Proving early termination of Variant 1 remains an open question.

\section*{Acknowledgments} \label{}
This work was partially supported by the MIT Schwarzman College of Computing.

 \bibliographystyle{elsarticle-num} 
 %\bibliography{cas-refs}
\bibliography{library.bib}

\begin{thebibliography}{10}
\expandafter\ifx\csname url\endcsname\relax
  \def\url#1{\texttt{#1}}\fi
\expandafter\ifx\csname urlprefix\endcsname\relax\def\urlprefix{URL }\fi
\expandafter\ifx\csname href\endcsname\relax
  \def\href#1#2{#2} \def\path#1{#1}\fi

\bibitem{stynes2005steady}
M.~Stynes, Steady-state convection-diffusion problems, Acta Numerica 14 (2005)
  445--508.

\bibitem{stynes2018convection}
M.~Stynes, D.~Stynes, Convection-diffusion problems, Vol. 196, American
  Mathematical Soc., 2018.

\bibitem{roos2014numerical}
H.~Roos, M.~Stynes, L.~Tobiska,
  \href{https://books.google.com/books?id=0nwCswEACAAJ}{Numerical Methods for
  Singularly Perturbed Differential Equations: Convection-Diffusion and Flow
  Problems}, Springer Series in Computational Mathematics, Springer Berlin
  Heidelberg, 2014.
\newline\urlprefix\url{https://books.google.com/books?id=0nwCswEACAAJ}

\bibitem{EYMARD2000713}
R.~Eymard, T.~Gallouet, R.~Herbin,
  \href{https://www.sciencedirect.com/science/article/pii/S1570865900070058}{Finite
  volume methods}, in: Solution of Equation in $\mathbb{R}^n$ (Part 3),
  Techniques of Scientific Computing (Part 3), Vol.~7 of Handbook of Numerical
  Analysis, Elsevier, 2000, pp. 713--1018.
\newblock \href {https://doi.org/https://doi.org/10.1016/S1570-8659(00)07005-8}
  {\path{doi:https://doi.org/10.1016/S1570-8659(00)07005-8}}.
\newline\urlprefix\url{https://www.sciencedirect.com/science/article/pii/S1570865900070058}

\bibitem{dolejvsi2015discontinuous}
V.~Dolej{\v{s}}{\'\i}, M.~Feistauer, Discontinuous {G}alerkin method, Analysis
  and Applications to Compressible Flow. Springer Series in Computational
  Mathematics 48~(234) (2015).

\bibitem{di2011mathematical}
D.~A. Di~Pietro, A.~Ern, Mathematical aspects of discontinuous {G}alerkin
  methods, Vol.~69, Springer Science \& Business Media, 2011.

\bibitem{augustin2011assessment}
M.~Augustin, A.~Caiazzo, A.~Fiebach, J.~Fuhrmann, V.~John, A.~Linke, R.~Umla,
  An assessment of discretizations for convection-dominated
  convection--diffusion equations, Computer Methods in Applied Mechanics and
  Engineering 200~(47-48) (2011) 3395--3409.

\bibitem{aavatsmark2002introduction}
I.~Aavatsmark, An introduction to multipoint flux approximations for
  quadrilateral grids, Computational Geosciences 6~(3) (2002) 405--432.

\bibitem{chidyagwai2011coupling}
P.~Chidyagwai, I.~Mishev, B.~Riviere, On the coupling of finite volume and
  discontinuous {G}alerkin method for elliptic problems, Journal of
  computational and applied mathematics 235~(8) (2011) 2193--2204.

\bibitem{riviere2014convergence}
B.~Riviere, X.~Yang, Convergence analysis of a coupled method for
  time-dependent convection-diffusion equations, Numerical Methods for Partial
  Differential Equations 30~(1) (2014) 133--157.

\bibitem{doyle2020multinumerics}
B.~Doyle, B.~Riviere, M.~Sekachev, A multinumerics scheme for incompressible
  two-phase flow, Computer Methods in Applied Mechanics and Engineering 370
  (2020) 113213.

\bibitem{berger2005analysis}
M.~Berger, M.~Aftosmis, S.~Muman, Analysis of slope limiters on irregular
  grids, in: 43rd AIAA Aerospace Sciences Meeting and Exhibit, 2005, p. 490.

\bibitem{fabien2025high}
M.~S. Fabien, A high-order augmented basis positivity-preserving discontinuous
  {G}alerkin method for a linear hyperbolic equation, arXiv preprint
  arXiv:2503.04039 (2025).

\bibitem{fabien2024positivity}
M.~S. Fabien, Positivity-preserving discontinuous {G}alerkin scheme for linear
  hyperbolic equations with characteristics-informed augmentation, Results in
  Applied Mathematics 22 (2024) 100460.

\bibitem{demkowicz2006computing}
L.~Demkowicz, Computing with hp-adaptive finite elements: volume 1 one and two
  dimensional elliptic and {M}axwell problems, Chapman and Hall/CRC, 2006.

\bibitem{cockburn2010hybridizable}
B.~Cockburn, The hybridizable discontinuous {G}alerkin methods, in: Proceedings
  of the International Congress of Mathematicians 2010 (ICM 2010) (In 4
  Volumes) Vol. I: Plenary Lectures and Ceremonies Vols. II--IV: Invited
  Lectures, World Scientific, 2010, pp. 2749--2775.

\bibitem{sevilla2018face}
R.~Sevilla, M.~Giacomini, A.~Huerta, A face-centred finite volume method for
  second-order elliptic problems, International Journal for Numerical Methods
  in Engineering 115~(8) (2018) 986--1014.

\bibitem{FRERICHS2021113487}
D.~Frerichs, V.~John,
  \href{https://www.sciencedirect.com/science/article/pii/S0377042721001060}{On
  reducing spurious oscillations in discontinuous {G}alerkin ({DG}) methods for
  steady-state convection–diffusion equations}, Journal of Computational and
  Applied Mathematics 393 (2021) 113487.
\newblock \href {https://doi.org/https://doi.org/10.1016/j.cam.2021.113487}
  {\path{doi:https://doi.org/10.1016/j.cam.2021.113487}}.
\newline\urlprefix\url{https://www.sciencedirect.com/science/article/pii/S0377042721001060}

\bibitem{Dolejsi2002}
V.~Dolejsi, M.~Feistauer, C.~Schwab, \href{http://eudml.org/doc/249044}{On
  discontinuous {G}alerkin methods for nonlinear convection-diffusion problems
  and compressible flow}, Mathematica Bohemica 127~(2) (2002) 163--179.
\newline\urlprefix\url{http://eudml.org/doc/249044}

\bibitem{Dolejsi2003}
V.~Dolejsi, M.~Feistauer, C.~Schwab,
  \href{https://ideas.repec.org/a/eee/matcom/v61y2003i3p333-346.html}{On some
  aspects of the discontinuous galerkin finite element method for conservation
  laws}, Mathematics and Computers in Simulation (MATCOM) 61~(3) (2003)
  333--346.
\newblock \href {https://doi.org/None} {\path{doi:None}}.
\newline\urlprefix\url{https://ideas.repec.org/a/eee/matcom/v61y2003i3p333-346.html}

\bibitem{guermond2011entropy}
J.-L. Guermond, R.~Pasquetti, B.~Popov, Entropy viscosity method for nonlinear
  conservation laws, Journal of Computational Physics 230~(11) (2011)
  4248--4267.

\bibitem{michoski2016comparison}
C.~Michoski, C.~Dawson, E.~J. Kubatko, D.~Wirasaet, S.~Brus, J.~J. Westerink, A
  comparison of artificial viscosity, limiters, and filters, for high order
  discontinuous {G}alerkin solutions in nonlinear settings, Journal of
  Scientific Computing 66~(1) (2016) 406--434.

\bibitem{yu2020study}
J.~Yu, J.~S. Hesthaven, A study of several artificial viscosity models within
  the discontinuous {G}alerkin framework, Commun. Comput. Phys. 27~(5) (2020)
  1309--1343.

\bibitem{kuzmin2021new}
D.~Kuzmin, A new perspective on flux and slope limiting in discontinuous
  {G}alerkin methods for hyperbolic conservation laws, Computer Methods in
  Applied Mechanics and Engineering 373 (2021) 113569.

\bibitem{joshaghani2022maximum}
M.~S. Joshaghani, B.~Riviere, M.~Sekachev, Maximum-principle-satisfying
  discontinuous {G}alerkin methods for incompressible two-phase immiscible
  flow, Computer Methods in Applied Mechanics and Engineering 391 (2022)
  114550.

\bibitem{badia2015discrete}
S.~Badia, A.~Hierro, On discrete maximum principles for discontinuous
  {G}alerkin methods, Computer Methods in Applied Mechanics and Engineering 286
  (2015) 107--122.

\bibitem{barrenechea2024finite}
G.~R. Barrenechea, V.~John, P.~Knobloch, Finite element methods respecting the
  discrete maximum principle for convection-diffusion equations, SIAM Review
  66~(1) (2024) 3--88.

\bibitem{EYMARD200131}
R.~Eymard, T.~Gallouët, R.~Herbin,
  \href{https://www.sciencedirect.com/science/article/pii/S0168927400000246}{Finite
  volume approximation of elliptic problems and convergence of an approximate
  gradient}, Applied Numerical Mathematics 37~(1) (2001) 31--53.
\newblock \href {https://doi.org/https://doi.org/10.1016/S0168-9274(00)00024-6}
  {\path{doi:https://doi.org/10.1016/S0168-9274(00)00024-6}}.
\newline\urlprefix\url{https://www.sciencedirect.com/science/article/pii/S0168927400000246}

\bibitem{ebeida2011uniform}
M.~S. Ebeida, S.~A. Mitchell, Uniform random {V}oronoi meshes, in: Proceedings
  of the 20th international meshing roundtable, Springer, 2011, pp. 273--290.

\bibitem{weyer2002automatic}
S.~Weyer, A.~Fr{\"o}hlich, H.~Riesch-Oppermann, L.~Cizelj, M.~Kovac, Automatic
  finite element meshing of planar {V}oronoi tessellations, Engineering
  fracture mechanics 69~(8) (2002) 945--958.

\bibitem{aizinger2017anisotropic}
V.~Aizinger, A.~Kos{\'\i}k, D.~Kuzmin, B.~Reuter, Anisotropic slope limiting
  for discontinuous {G}alerkin methods, International Journal for Numerical
  Methods in Fluids 84~(9) (2017) 543--565.

\bibitem{godunov1959difference}
S.~K. Godunov, A difference scheme for numerical solution of discontinuous
  solution of hydrodynamic equations, Math. Sbornik 47 (1959) 271--306.

\bibitem{fabri2009cgal}
A.~Fabri, S.~Pion, Cgal: The computational geometry algorithms library, in:
  Proceedings of the 17th ACM SIGSPATIAL international conference on advances
  in geographic information systems, 2009, pp. 538--539.

\bibitem{talischi2012polymesher}
C.~Talischi, G.~H. Paulino, A.~Pereira, I.~F. Menezes, Polymesher: a
  general-purpose mesh generator for polygonal elements written in matlab,
  Structural and Multidisciplinary Optimization 45~(3) (2012) 309--328.

\bibitem{FORMAGGIA2004511}
L.~Formaggia, S.~Micheletti, S.~Perotto,
  \href{https://www.sciencedirect.com/science/article/pii/S0168927404000972}{Anisotropic
  mesh adaptation in computational fluid dynamics: {A}pplication to the
  advection–diffusion–reaction and the {S}tokes problems}, Applied
  Numerical Mathematics 51~(4) (2004) 511--533, applied Scientific Computing:
  Advances in Grid Generatuion, Approximation and Numerical Modeling.
\newblock \href {https://doi.org/https://doi.org/10.1016/j.apnum.2004.06.007}
  {\path{doi:https://doi.org/10.1016/j.apnum.2004.06.007}}.
\newline\urlprefix\url{https://www.sciencedirect.com/science/article/pii/S0168927404000972}

\bibitem{Carpio}
J.~Carpio, J.~L. Prieto, R.~Bermejo,
  \href{https://doi.org/10.1137/120874606}{Anisotropic ``goal-oriented'' mesh
  adaptivity for elliptic problems}, SIAM Journal on Scientific Computing
  35~(2) (2013) A861--A885.
\newblock \href {http://arxiv.org/abs/https://doi.org/10.1137/120874606}
  {\path{arXiv:https://doi.org/10.1137/120874606}}, \href
  {https://doi.org/10.1137/120874606} {\path{doi:10.1137/120874606}}.
\newline\urlprefix\url{https://doi.org/10.1137/120874606}

\bibitem{engwirda2018generalised}
D.~Engwirda, Generalised primal-dual grids for unstructured co-volume schemes,
  Journal of Computational Physics 375 (2018) 155--176.

\bibitem{BADIA2015107}
S.~Badia, A.~Hierro,
  \href{https://www.sciencedirect.com/science/article/pii/S004578251400485X}{On
  discrete maximum principles for discontinuous {G}alerkin methods}, Computer
  Methods in Applied Mechanics and Engineering 286 (2015) 107--122.
\newblock \href {https://doi.org/https://doi.org/10.1016/j.cma.2014.12.006}
  {\path{doi:https://doi.org/10.1016/j.cma.2014.12.006}}.
\newline\urlprefix\url{https://www.sciencedirect.com/science/article/pii/S004578251400485X}

\bibitem{oden1998discontinuoushpfinite}
J.~T. Oden, I.~Babuska, C.~E. Baumann, A discontinuous $hp$-finite element
  method for diffusion problems, Journal of computational physics 146~(2)
  (1998) 491--519.

\bibitem{HEMKER1996277}
P.~Hemker,
  \href{https://www.sciencedirect.com/science/article/pii/S0377042796001136}{A
  singularly perturbed model problem for numerical computation}, Journal of
  Computational and Applied Mathematics 76~(1) (1996) 277--285.
\newblock \href {https://doi.org/https://doi.org/10.1016/S0377-0427(96)00113-6}
  {\path{doi:https://doi.org/10.1016/S0377-0427(96)00113-6}}.
\newline\urlprefix\url{https://www.sciencedirect.com/science/article/pii/S0377042796001136}

\end{thebibliography}




% \end{thebibliography}
\end{document}